\documentclass[11pt]{article}        

\usepackage{etoolbox}
\usepackage{hyperref}

\usepackage{amsthm,amssymb}
\usepackage{nicematrix}
\usepackage{graphicx}
\usepackage{comment}
\usepackage{graphicx}
\usepackage{amscd,amssymb}
\usepackage{amsmath}
\usepackage{tikz}
\usepackage{tikz-3dplot}
\usepackage{tikz-cd}
\usepackage{faktor}

\newtheorem{theorem}{Theorem}[section]

\newtheorem*{RefTheorem}{Theorem}

\newtheorem{lemma}[theorem]{Lemma}
\newtheorem{prop}[theorem]{Proposition}

\newtheorem{corollary}[theorem]{Corollary}

\theoremstyle{definition}
\newtheorem{defi}[theorem]{Definition}
\newtheorem{remark}[theorem]{Remark}

\newtheorem{convention}[theorem]{Convention}

\newtheorem*{claim}{Claim}
\newtheorem{example}[theorem]{Example}
\newtheorem*{example*}{Example}

\def\R{\mathbb{R}}

\newcommand{\pp}[2]{\frac{\partial#1}{\partial#2}}
\newcommand{\norm}[1]{\|#1\|}

\newcommand{\vol}{\operatorname{vol}}
\newcommand{\image}{\operatorname{Im}}

\newcommand{\sym}{\operatorname{Sym}}

\newcommand{\transv}{\mathrel{\text{\tpitchfork}}}
\makeatletter
\newcommand{\tpitchfork}{%
  \vbox{
    \baselineskip\z@skip
    \lineskip-.52ex
    \lineskiplimit\maxdimen
    \m@th
    \ialign{##\crcr\hidewidth\smash{$-$}\hidewidth\crcr$\pitchfork$\crcr}
  }%
}

\definecolor{mymagenta}{RGB}{139, 39, 95}
\definecolor{myblue}{RGB}{52,105,165}

\makeatother

\title{Contact topology and intrinsic characterization of transversely vanishing curl eigenforms}

\date{}

\author{Robert Cardona and Willi Kepplinger}

\newcommand{\Addresses}{{
  \bigskip
  \footnotesize

  R.~Cardona, \textsc{Departament de Matem\`atiques i Inform\`atica, Universitat de Barcelona, Gran Via de Les Corts Catalanes 585, 08007 Barcelona, Spain; Centre de Recerca Matemàtica, Campus de Bellaterra, Edifici C, 08193, Barcelona, Spain.}\par\nopagebreak
  \textit{E-mail address}: \texttt{robert.cardona@ub.edu}

  \medskip

  W.~Kepplinger, \textsc{Instituto de Ciencias Matemáticas (ICMAT), Calle Nicolás Cabrera 13-15, 28049 Madrid, Spain.}\par\nopagebreak
  \textit{E-mail address}: \texttt{willi.kepplinger@icmat.es}

}}

\begin{document}

\maketitle

\begin{abstract}
We study the geometry and topology of transversely vanishing eigenforms of the curl operator on $3$-manifolds, placing them in the broader framework of near-contact forms. We prove a Gray-type stability theorem for the kernels of near-contact forms, which we call near-contact structures, and derive a local neighborhood theorem around their singularities. As applications, we obtain a stability result for the contact topology of curl eigenforms, a classification of near-contact structures in terms of formal homotopical data, and a proof of the existence of contact structures on closed orientable 3-manifolds based on the spectral theory of the curl operator. We also study the symplectization of curl eigenforms as a possible bridge to symplectic dynamics. Finally, we characterize those transversely vanishing $1$-forms which are curl eigenforms for some smooth Riemannian metric. We deduce a generalization of the classical correspondence, due to Etnyre and Ghrist, between Reeb-like fields of contact forms and nowhere-vanishing rotational Beltrami fields to the setting of transversely vanishing vector fields. 
\end{abstract}

\tableofcontents

\section{Introduction}\label{sec:Intro}

\paragraph{Context.} In a now classical paper of Etnyre and Ghrist \cite{Etnyre2000I}, the authors establish a correspondence between reparametrizations of Reeb fields of contact forms (``Reeb-like fields") on $3$-dimensional manifolds and (non-vanishing) rotational Beltrami fields, that is,
vector fields which are parallel to their curl with respect to some Riemannian metric.  In terms of forms, every contact form satisfies the Beltrami equation 
\begin{equation}\label{eq:beltrami_equation_form}
    \star_g d\alpha=f\alpha 
\end{equation}
for some Riemannian metric $g$ and some nowhere-vanishing function $f$, and, conversely, a nowhere-vanishing $1$-form solving the Beltrami equation for some Riemannian metric $g$ and some nowhere-vanishing proportionality factor $f$ defines a contact form. This equivalence has opened the door to the introduction of contact topological methods as well as techniques from symplectic dynamics in the study of nowhere-vanishing rotational Beltrami fields. One reason this is important is that Beltrami fields constitute a remarkable class of stationary solutions to ideal magnetohydrodynamics and the Euler equations for ideal fluids. Striking examples of this geometric and topological way of looking at fluid dynamics include a proof of the fact that any analytic non-vanishing stationary solution to the Euler equations in a closed $3$-manifold which is not a torus bundle over the circle admits a periodic orbit \cite{Etnyre2000I, CieliebakVolkov2017}, the construction of stationary solutions to the $3$D incompressible Euler equations that contain every knot and link among its periodic orbits \cite{Etnyre2000III,EncisoPeralta2012}, of a solution which is isolated with respect to the $C^1$-topology \cite{EncisoKepplingerPeralta2025}, or the construction of solutions which are Turing complete \cite{Cardona2021}. 

Using the correspondence in the other direction, that is, solving the PDE to obtain potentially interesting contact structures, has so far proven to be challenging. There are several reasons for this. While there are geometric criteria for compatible pairs of a contact form and a Riemannian metric which guarantee tightness of the underlying contact structure \cite[Theorems 1.1. and 1.3.]{Etnyre2012}, there does not exist a method for constructing such compatible pairs. One possible idea for achieving this on a closed oriented $3$-manifold would be to pick a smooth Riemannian metric $g$ and to solve the curl eigenvalue problem
$$
\star_g d\alpha=\lambda\alpha 
$$
It is well known that the spectrum of curl consists of eigenvalues only and that it is a discrete subset of the real line, and that furthermore all nonzero eigenvalues have finite multiplicity and that the associated eigenforms are smooth. See \cite{Bar} for an account of properties of this spectral problem. Crucially, any nowhere-vanishing eigenform with nonzero eigenvalue will be (weakly) compatible with the Riemannian metric $g$ \cite[Proposition 2.1.]{Etnyre2012}. The main problem with this approach is that the modulus of compatibility can in general not be controlled and that the solutions to this eigenvalue problem are in general somewhere vanishing. Indeed, in the explicitly computable cases it is not hard to see that in general one has to contend with such zeros \cite[Lemma 1]{PeraltaSlobodeanu2023}.

\paragraph{Near-contact forms and structures.} On the other hand, it is known that, for generic Riemannian metrics, all curl eigenforms have only transverse zeros \cite[Theorem 1.1]{Enciso2012} (see also Remark \ref{rmk: regularity of metric transversality statement}). It is therefore of interest to understand such forms, or more generally, $1$-forms which satisfy the Beltrami equation and vanish transversely. To study them, we use the more general notion of near-contact forms.

\begin{defi}\label{def:near_contact_form_intro}
    Let $M$ be an oriented $3$-dimensional manifold. A $1$-form $\alpha$ is called a \textit{positive (negative) near-contact form} if 
\begin{enumerate}
    \item $\alpha\wedge d\alpha>0$ ($<0$) away from $\alpha^{-1}(0)$,
    \item $\alpha^{-1}(0)=d\alpha^{-1}(0)$ and both vanish transversely,
    \item for every $p\in \alpha^{-1}(0)$, the map \begin{align}\label{eq:symmetric_wedge_intro}
    \begin{split}
        \nabla \alpha \wedge \nabla d\alpha|_p: T_pM \times T_pM &\longrightarrow \wedge^3T^*M\\
                (v,w) &\longmapsto \nabla_v\alpha\wedge\nabla_wd\alpha
    \end{split}
\end{align}
is positive (negative) definite, meaning that for any $v\neq 0$ the form $\nabla_v\alpha \wedge \nabla_v d\alpha|_p$ is a positive (negative) three-form with respect to the fixed orientation in $M$.
\end{enumerate}
\end{defi}
For an arbitrary $1$-form $\alpha$, a point where the second and third items above are satisfied is called a \textit{near-contact singularity}, and near such a singularity the form is a near-contact form (Lemma \ref{lem:Taylor_contact_condition}). While preparing the final version of this article, we became aware of \cite{ShevSmir}, where near-contact forms are introduced with a different (but equivalent) definition. We have thus kept their name, and discuss this equivalence in Section \ref{subsec: a comparison with other models for singular}. Some of the basic properties of near-contact forms that we establish overlap with \cite{ShevSmir}. We will point out when this is the case.
\medskip

Near-contact forms are a compelling generalization of contact forms, as we will see in Section \ref{sec:contact_singularities}, for the following reasons.

First, they define a kind of analogue of the Reeb field: In this context there is no canonical vector field determined by $\alpha$, but there is a canonical class of vector fields that coincide up to positive rescaling (and which generalize the positive rescalings of the Reeb field of a contact form), which we call the \emph{Reeb class} $\mathcal{R}_\alpha$ of $\alpha$, see Definition \ref{def:Reeb_class}. Mimicking the contact terminology, a vector field in the Reeb class is called Reeb-like. Quite naturally, a transversely vanishing rotational Beltrami form $\alpha$, that is, a form that satisfies $\star_g d \alpha=f\alpha$ for a nowhere-vanishing proportionality factor $f$, is indeed a near-contact form (positive respectively negative, depending on the sign of $f$). Throughout the paper, we work in the positive case; see Convention \ref{conv:positive}. Moreover, its corresponding metric-dual $X$ solves $\nabla \times_g X=fX$ and is Reeb-like for $\alpha$ (Lemma \ref{lem:transverse}). While it is not immediately obvious from the definition, near-contact singularities only come as zeros of either index $1$ or index $2$ which is related to the volume preserving nature of these dual vector fields.

Secondly, the property of being a near-contact form is preserved both by pullback along diffeomorphisms and multiplication with nowhere-vanishing functions, and it therefore makes sense to speak of near-contact structures, see Definition \ref{def:near_contact_structure}. \medskip

Besides the already mentioned \cite{ShevSmir}, several other authors have introduced singular generalizations of contact forms or structures, for instance, singular contact forms with non-vanishing differential at the zero \cite{JiangMinhZung} or $b$-contact structures in the context of Jacobi geometry \cite{MirandaOms}. Closest to our setting, Pollard and Alexander \cite{PollardAlexander}, motivated by point defects in cholesteric liquid crystals, introduced singular contact forms (contact away from finitely many zeros) and Reeb-like fields for them. Furthermore, they showed that singular Beltrami fields give rise to such forms \cite[Proposition 3.1]{PollardAlexander}, and studied which types of singularities can occur in them. We compare these notions with near-contact forms in Section \ref{subsec: a comparison with other models for singular}. 
\medskip

Our goal in the first part of the present paper is to establish some foundations for the contact topological properties of near-contact forms and transversely vanishing curl eigenforms, with the main goal of understanding the local structure near the singularities. As a first step towards this end, we prove a Gray-stability-type result for near-contact structures (Theorem \ref{thm:singularGray}). The essential technical step, which is inspired by the work of Honda in the context of normal forms for self-dual harmonic $2$-forms \cite{honda2004local}, is to modify the flow constructed in usual Gray stability-type arguments (defined in the complement of the singularities of the near-contact forms) to ensure existence up to time one of the corresponding flow. In our context, this is done by adding a suitably chosen contact Hamiltonian field to the Gray stability vector field. We deduce the following normal form theorem for near-contact structures near their singularities. We state the result for positive near-contact structures (see Convention \ref{conv:positive} for the negative case).

\begin{theorem}[Local model theorem]\label{thm: local model for Beltrami}
    Let $M$ be an oriented $3$-manifold and $\alpha$ a $1$-form with a positive near-contact singularity at $p\in M$. Then there exists a neighborhood $V$ of $p$, a neighborhood $U$ of the origin in $\R^3$, and a homeomorphism 
    $$\phi: (U,0)\subset \mathbb R^3 \longrightarrow (V,p),$$ 
    which is a smooth diffeomorphism in the complement of $0$ and satisfies $$(\phi\vert_{U\setminus \{0\}})^\ast (\alpha)=f\alpha_{std},$$ 
    where $\alpha_{std}$ is equal to
    $$\alpha_{std}^+= xdx+ydy-2zdz + yzdx-xzdy$$
    if the signature of the linearization of $\alpha$ is $1$ at $p$, or equal to
    $$\alpha_{std}^-= -xdx-ydy+2zdz - yzdx+xzdy,$$
    if the signature is $-1$. Furthermore, $f$ is positive and smooth on $U\setminus\{0\}$, and extends continuously to $U$ with $f(0)=1$ and $f-1=\mathcal{O}(\norm{x})$.
\end{theorem}
The model $\alpha^+_{std}$ already appears in \cite[Example 4.1]{PollardAlexander}, and $-\alpha^+_{std}=\alpha^-_{std}$ gives the other possible signature. Theorem \ref{thm: local model for Beltrami} shows that these are universal local models for near-contact singularities, up to a conformal factor. Interestingly, the regularity of $\phi$ in this theorem cannot be improved to be $C^1$ on $\alpha^{-1}(0)$ in general; see Section \ref{subsec: optimal regularity}. This local normal form allows us to study the local contact topology of $\ker \alpha$ near the singularity. In particular, we find a foliation by ``overtwisted" convex spheres (in the sense of Giroux's criterion, i.e., with a disconnected dividing set) surrounding it (Corollary \ref{cor:foliation_convex_spheres}), implying that the corresponding contact structure in $M\setminus \alpha^{-1}(0)$ is overtwisted. An informal overview of the applications of both the Gray-stability result, the local normal form Theorem \ref{thm: local model for Beltrami}, and its consequences concerning the local contact topology near the singularities is the following:
\begin{enumerate}
    \item \textit{Contact topological stability of generic curl eigenforms (Theorem \ref{thm: topological stability of curl eigenforms})}: 
    Given a Riemannian metric $g$ in a certain generic subset of the space of Riemannian metrics, any curl eigenform $\alpha (g)$ associated to $g$ defines a structurally stable contact topology. That is, for any metric $g^\prime$ sufficiently close to $g$ there exists a homeomorphism which takes $f\alpha$ to the corresponding curl eigenform $\alpha (g^\prime)$. Both the homeomorphism and the function $f$ are smooth in the complement of the zeros of $\alpha(g)$. Moreover, just like for Theorem \ref{thm: local model for Beltrami}, the regularity of this homeomorphism cannot be improved to $C^1$ at the zeros of $\alpha$.
    \item \textit{A classification of near-contact structures on closed $3$-manifolds by formal homotopical data (Theorem \ref{thm:classification_near_contact_structures})}: This is a classification of near-contact structures in terms of an $h$-principle. Given two positive near-contact structures with the same set of zeros $Z\neq \emptyset$, they are isotopic by a continuous isotopy of $M$ fixing $Z$, and smooth away from $Z$, if and only if they are homotopic as plane fields in $M\setminus Z$. 
    \item \textit{A spectral-theoretic proof of Martinet's theorem (Theorem \ref{thm:Martinet})}: Every closed orientable 3-manifold admits a contact structure. Classical proofs \cite{Geiges} rely on strong structural results on the structure of $3$-manifolds such as the existence of surgery presentations, open book decompositions, or their realizability as branched coverings. Other approaches use a geometric heat flow \cite{Altschuler1995} or the $h$-principle \cite{El}. Our proof replaces the topological input by spectral theory of the curl operator: For a generic metric, all associated curl eigenforms have nondegenerate zeros only \cite[Theorems 3.5 and 3.7]{Enciso2012}. By Theorem \ref{thm: local model for Beltrami} and Corollary \ref{cor:foliation_convex_spheres}, these zeros are surrounded by convex spheres, and zeros of opposite index can be canceled in pairs using convex surface theory.
\end{enumerate}

\paragraph{Symplectizations of near-contact forms.} We then turn to studying 
the ``symplectization'' of a near-contact form $\alpha$, that is, the $2$-form $d(e^t\alpha)$ on the manifold $\R\times M$. Symplectizations of near-contact forms are shown to be self-dual harmonic $2$-forms for some metric (or, equivalently, near-symplectic forms). Moreover, our local model for near-contact forms carries over to a local model of their symplectizations (Corollary \ref{cor:symplectization_normal_form}). The model $d(e^t\alpha_{std})$ agrees to first order along the singular lines $\R \times \alpha^{-1}(0)$ with the natural non-compact analogue of Honda's local model for self-dual harmonic $2$-forms near their vanishing circles \cite{honda2004local}, but the two cannot be identified on a neighborhood that is uniform along the line (see Remark \ref{rem:Moser_Honda}). We deduce from Theorem \ref{thm: local model for Beltrami} that in the symplectization of a near-contact structure, the corresponding singular lines $\R\times \alpha^{-1}(0)$ have neighborhoods which are foliated by overtwisted contact hypersurfaces (Lemma \ref{lem:contact_foliation}), see Figure \ref{fig:symplectization} for a schematic representation.

Let us also mention that the local normal forms for self-dual harmonic $2$-forms have been the object of significant interest in the context of the geometry and topology of $4$-manifolds; see a series of papers by Taubes \cite{Taubes1999, Taubes2002}, the work of Gay and Kirby \cite{GayKirby2004}, as well as Gerig \cite{Gerig2020, Gerig2022}. We believe that it would be interesting to explore an application of similar techniques in the symplectization of transversely vanishing curl eigenforms, respectively rotational Beltrami forms or near-contact forms, to possibly obtain dynamical information about them. Indeed, extracting dynamical information about (non-vanishing) Reeb flows by studying their corresponding symplectizations has proven extremely fruitful since the foundational work of Hofer \cite{Hofer1993}. As mentioned before, transversely vanishing rotational Beltrami fields occur frequently and naturally both in fluid dynamics and plasma physics (see, e.g., \cite{Woltjer1958,ArnoldKhesin}), so we expect that there would be substantial interest in such developments.

\paragraph{An intrinsic characterization of curl eigenforms.} The second part of the paper is devoted to understanding exactly which transversely vanishing $1$-forms arise as curl eigenforms for some Riemannian metric. This is motivated by the discussion above, but can also be understood as a natural variant of the characterization of intrinsically harmonic $1$-forms, studied by Calabi \cite{Calabi} and later by Farber, Katz and Levine \cite{FarberKatzLevine}, Honda \cite{HondaTransversality2004}, and Volkov \cite{Volkov}.

We call a $1$-form $\alpha$ an \emph{intrinsic curl eigenform} if there exists a Riemannian metric $g$ for which $\alpha$ is an eigenform of the curl operator associated to $g$. Similarly, we call a $1$-form $\alpha$ an \emph{intrinsic rotational Beltrami form} if there exists a metric $g$ such that $\alpha$ is a rotational Beltrami form, namely, we have $\star_g d\alpha=f\,\alpha$ for some non-vanishing function $f$. This condition is equivalent to being the metric dual of a rotational Beltrami field with respect to $g$. 

Provided they vanish transversely, both curl eigenforms and rotational Beltrami forms are near-contact forms. The first observation towards the characterization result is that transversely vanishing curl eigenforms (respectively rotational Beltrami forms) are near-contact forms which additionally satisfy that the map \eqref{eq:symmetric_wedge_intro} is not only positive definite but also symmetric (Lemma \ref{lem:transverse}). The following theorem shows that near-contact forms with this property are exactly those transversely vanishing forms which are intrinsic curl eigenforms, which in turn are precisely those that are intrinsic rotational Beltrami forms. 

\begin{theorem}\label{theorem: intrinsically rotationally Beltrami}
    Let $M^3$ be an oriented $3$-manifold and let $\alpha$ be a transversely vanishing smooth $1$-form on $M$. Then the following are equivalent
\begin{enumerate}
    \item $\alpha$ is an intrinsic curl eigenform
    \item $\alpha$ is an intrinsic rotational Beltrami form 
    \item $\alpha$ is a near-contact form and the map \eqref{eq:symmetric_wedge_intro} is symmetric for every $p\in \alpha^{-1}(0)$.
\end{enumerate}
\end{theorem}
The first equivalence follows from a simple metric-rescaling trick, any $1$-form which is Beltrami with a smooth proportionality factor for some smooth metric is also a curl eigenform for some other smooth metric. This is explained in Remark \ref{rem:rescaling}. The equivalence to Item 3 in Theorem \ref{theorem: intrinsically rotationally Beltrami}, however, is quite subtle.
For a singular contact form, a metric making it Beltrami always exists on the complement of the zeros, but, as observed in \cite[Section 3]{PollardAlexander}, there is no a priori reason for it to extend over those zeros. Theorem \ref{theorem: intrinsically rotationally Beltrami} determines exactly when it does. Theorem \ref{theorem: intrinsically rotationally Beltrami} should be distinguished from \cite[Theorem 4.4, Proposition 4.3]{PollardAlexander}, which characterize which leading-order types of singularity occur in some Beltrami form, rather than whether a given form is Beltrami for some metric.
\medskip

The picture that emerges, then, is the following.
\begin{align*}
&\qquad \{\text{transversely vanishing intrinsic curl eigenforms}\} \\
& \qquad \qquad \qquad \qquad \qquad \qquad \rotatebox[origin=c]{90}{=} \\
&\{\text{transversely vanishing intrinsic rotational Beltrami forms}\}\\
&\qquad \qquad \qquad \qquad \qquad \qquad \rotatebox[origin=c]{270}{$\subsetneq$}\\
&\qquad \qquad \qquad \qquad \{\text{near-contact forms}\}    
\end{align*}
The last inclusion is strict due to the existence of near-contact forms for which the map \eqref{eq:symmetric_wedge_intro} is not symmetric, see Example \ref{ex:nonBeltrami}. However, interestingly enough, any near-contact form is conformal to a transversely vanishing intrinsic curl eigenform, see Proposition \ref{prop:near-contact is conformal to intrinsically Beltrami}. Furthermore, any two such factors $f_1,f_2$ satisfy $d(f_2/f_1)|_p=0$ at every $p\in\alpha^{-1}(0)$.

\begin{remark}
Notice that Theorem \ref{theorem: intrinsically rotationally Beltrami} implies that the set of transversely vanishing intrinsic curl eigenforms (or intrinsic rotational Beltrami forms) is emphatically not open in the set of transversely vanishing $1$-forms. Somewhat in contrast, the set of transversely vanishing intrinsically harmonic $1$-forms is open inside the set of transversely vanishing closed $1$-forms \cite{VolkovPhD}. Moreover, for intrinsically harmonic $1$-forms, Calabi's characterization \cite{Calabi} is a global condition on the foliation together with a mild local condition at the zeros, whereas for curl eigenforms it is precisely the local question at the zeros which is interesting.
\end{remark}

\noindent Theorem \ref{theorem: intrinsically rotationally Beltrami} then has several interesting consequences.
\begin{enumerate}
    \item \emph{A generalization of Etnyre and Ghrist's correspondence between Reeb-like vector fields and rotational Beltrami fields (Theorem \ref{thm: Beltrami-contact for near-contact forms}):}
    Remarkably, Theorem \ref{theorem: intrinsically rotationally Beltrami} can be used to obtain a generalized version of Etnyre and Ghrist's correspondence to the setting of transversely vanishing vector fields and $1$-forms. We prove that transversely vanishing vector fields are rotational Beltrami fields for some Riemannian metric if and only if they are Reeb-like for some near-contact form $\alpha$ for which the map \eqref{eq:symmetric_wedge_intro} is symmetric for every $p\in \alpha^{-1}(0)$.

\item \emph{Curl eigenfields with prescribed singularity germs (Corollary \ref{cor:prescribed_singularities}):}
Determining precisely which (collections of) submanifolds of a given manifold can be realized as zero sets of Laplace eigenfunctions for some metric is a difficult problem; we refer to work of Colin de Verdière \cite{CdV1986}, Enciso and Peralta-Salas \cite{EncisoPeralta2015}, and Lisi \cite{Lisi2011}. For eigenforms of the curl operator on a closed $3$-manifold, the analogous problem is much more tractable: combining Theorem \ref{theorem: intrinsically rotationally Beltrami} with Eliashberg's relative $h$-principle for contact structures, one can prescribe not only the zero set of a transversely vanishing curl eigenform, subject only to the necessary condition that it consists of as many zeros of index $1$ as of index $2$, but even the germ of the eigenform at each zero, which can be any germ of positive near-contact form with symmetric map \eqref{eq:symmetric_wedge_intro}. Notice that these conditions only involve the $2$-jet of the form at the singularity, and thus all the higher jets can be prescribed to be arbitrary.
\item \emph{Regularity is the only distinguishing property of transversely vanishing curl eigenforms for smooth metrics (Corollary \ref{cor:upgrade_regularity}):}
By elliptic regularity, curl eigenforms of the curl operator defined with respect to a smooth metric are themselves smooth. However, it is not clear whether they need to have any additional, hidden qualities. In Corollary \ref{cor:upgrade_regularity} we prove that, in fact, any transversely vanishing smooth $1$-form which is an eigenform for a curl operator defined with respect to a metric which is merely assumed to be continuous already is a curl eigenform with respect to a smooth metric. 
\end{enumerate}

\paragraph{Structure of the paper.} In Section \ref{sec:contact_singularities} we introduce near-contact forms and structures, study their elementary properties, and compare them with other notions of singular contact forms. In Section \ref{sec:singularGray} we prove the singular Gray stability Theorem \ref{thm:singularGray} and deduce the topological stability of generic curl eigenforms. In Section \ref{sec: contact topology of transverse singularities} we prove the local normal form Theorem \ref{thm: local model for Beltrami}, discuss its optimal regularity, and derive the local and global contact topological consequences, including the classification of near-contact structures and the proof of Martinet's theorem. Section \ref{sec:symplectization} is devoted to symplectizations of near-contact forms, and Section \ref{sec: intrinsic rotational Beltrami fields} to the proof of Theorem \ref{theorem: intrinsically rotationally Beltrami} and its applications.

\paragraph{Acknowledgements.} RC acknowledges partial support from the AEI grant PID 2023-147585NA-I00, the Departament de Recerca i Universitats de la Generalitat de Catalunya (2021 SGR 00697), and the Spanish State Research Agency, through the Severo Ochoa and María de Maeztu Program for Centers and Units of Excellence in R\&D (CEX2020-001084-M). The latter provided RC's ``Boosting Career" grant, under which this collaboration was launched.
Throughout the course of this project, WK was supported at various points in time by the Bilateral AEI-DFG project: Celestial Mechanics, Hydrodynamics, and Turing Machines with reference code PCI2024-155062-2, and the grant PID2022-136795NB-I00 funded by MCIN/AEI/10.13039/501100011033, as well as by the Austrian Science Fund (FWF) [10.55776/P34318] and by the European Research Council (ERC) Consolidator Grant through the grant agreement 101001159. RC is grateful to Eduardo Fernández-Fuertes for helpful discussions concerning edge-rounding of convex surfaces. WK is grateful to both Yasha Eliashberg and John Etnyre for helpful discussions.

\paragraph{AI disclosure.}
This project was started in March $2025$, all main ideas and computations are the authors' own. AI models (specifically, Claude Opus 5.5 and ChatGPT 5.6) were used in the final stages mostly for proofreading, detecting minor mistakes and inaccuracies, with the following noteworthy exceptions. Lemma \ref{lem: polynomial metric construction} was originally proved by the authors, with a proof in local coordinates for $g_k^0$ and an intricate combinatorial argument showing that the system of equations for $g_k^s$, with $s>0$, is always solvable. The latter argument was hard to present in written form, so the proof was given to AI models, asking for a simplification. They provided an essentially equivalent coordinate-free argument for the $g_k^0$ case, and a clean alternative argument for the $g_k^s$ case, which we simplified and present in this work. 
Another AI contribution worth mentioning was the idea to use the infinitesimal contact volume growth in Proposition \ref{prop: local invariant for near-contact forms} in order to recover $S_\beta$ even in low regularity, allowing us to lower the assumption on the regularity of the diffeomorphism from $C^2$ to $C^1$.
Lastly, AIs pointed out to us that the continuous extension of $f$ in Theorem \ref{thm:singularGray} over the singular points was already implicit in our computations, and provided the explicit form of the contact vector field $Y$ in the local model for the zeros of near-contact forms. Together, these allowed us to significantly streamline Section \ref{sec:symplectization}.
The authors take full responsibility for the content of this paper. In particular, this paper is entirely written by the authors.

\section{Near-contact forms and structures}\label{sec:contact_singularities}

In this section, we will introduce near-contact forms and near-contact structures as natural singular generalizations of contact forms and contact structures and study some of their most elementary properties. More delicate questions, for example on their topological stability (Section \ref{subsec: first application}) and local invariants of their singularities (Proposition \ref{prop: local invariant for near-contact forms}) are relegated to later sections. 

\subsection{Near-contact and rotational Beltrami forms}

When a $1$-form $\alpha$ and its exterior derivative vanish at a point $p$, one has the following elementary observation.

\begin{lemma}\label{lem:covariant_alpha_and_diagonalization}
    Let $\alpha$ be a $1$-form in $M$ such that $\alpha(p)=0$ and $d\alpha(p)=0$. Then
\begin{enumerate}
    \item The maps
$$\nabla \alpha|_p:T_pM\longrightarrow T^*_pM$$
and 
$$\nabla d\alpha|_p:T_pM \longrightarrow \wedge^2T^*_pM$$
are independent of the metric used to compute the covariant derivative. Moreover, they are linear isomorphisms if $\alpha$ respectively $d\alpha$ vanish transversely.
    \item The bilinear form $(u,v)\mapsto(\nabla_u\alpha)_p(v)$ on $T_pM$ is symmetric. In particular, there exists a basis of $T_pM$ in which its matrix is diagonal.
\end{enumerate}
\end{lemma}
\begin{proof}
Item $1$ follows simply from $\alpha$ and $d\alpha$ vanishing at $p$: in any chart, $\nabla\alpha|_p$ and $\nabla d\alpha|_p$ are given by the differentials $D\alpha|_p$ and $Dd\alpha|_p$ of the coefficients. For Item $2$ we observe that $d\alpha(p)=0$ implies that the linearization $(\nabla_i \alpha)_j$ of $\alpha$ is symmetric, and hence diagonalizes by Sylvester's law of inertia.
\end{proof}
The type of singularities that we want to consider is the following.
\begin{defi}
Let $\alpha$ be a $1$-form on an oriented $3$-manifold $M$. A point $p\in M$ is a \textit{positive (negative) near-contact singularity} of $\alpha$ if both $\alpha$ and $d\alpha$ vanish transversely at $p$ and the map 
    \begin{align}\label{eq:infinitessimal_contact_map}
    \begin{split}
        B_\alpha: T_pM \times T_pM &\longrightarrow \wedge^3T^*M\\
                (v,w) &\longmapsto \nabla_v\alpha\wedge\nabla_wd\alpha
    \end{split}
    \end{align}
    is \textit{positive (negative) definite}, meaning that for any $v\neq 0$ the form $\nabla_v\alpha \wedge \nabla_v d\alpha|_p$ is a positive (negative) three-form with respect to the fixed orientation in $M$.
\end{defi}
We note that, by definition, near-contact singularities are always isolated. 
\begin{remark}\label{rem:matrix_representation}
     Fix coordinates $(x,y,z)$ in a neighborhood $U$ of $p$, the canonical basis in $TM$ and $T^*M$, and the cyclic basis $(dy\wedge dz, dz\wedge dx, dx\wedge dy)$ in $\wedge^2 T^*M$ given by the wedge basis in $T^*M$, then we have the equality
    \begin{equation}\label{eq:wedge_matrix_representation}
        \nabla_u\alpha\wedge \nabla_v d\alpha |_p= {u}^t (D\alpha|_0)^t (Dd\alpha|_0) v \cdot dx\wedge dy \wedge dz.
    \end{equation}
    In particular, for a positive near-contact singularity, the matrix $(D\alpha|_0)^t(Dd\alpha|_0)$ is positive definite. This local representation will sometimes be useful in computations.
\end{remark}

In the case that $\alpha$ has a near-contact singularity at $p$, the bilinear form $B_\alpha$ does indeed capture the lowest-order terms of $\alpha\wedge d\alpha$ near $p$. This property will play a big part later when determining the optimal regularity of the local model for near-contact singularities (Theorem \ref{thm: local model for Beltrami}).

\begin{lemma}
\label{lem:Taylor_contact_condition}
Let $\alpha$ be a $1$-form on $M$ such that $\alpha$ and $d\alpha$ vanish at $p$. Then one has
\begin{align*}
    B_\alpha (v,v)=\lim_{t\to 0}\frac{1}{t^2}(\alpha\wedge d\alpha)(tv).
\end{align*}
\end{lemma}
\begin{proof}
Choose a tangent vector $v\in T_pM$, and consider a curve $\gamma:(-\varepsilon,\varepsilon)\longrightarrow M$ such that $\gamma(0)=p$ and $\gamma'(0)=v$. The Taylor expansion of $\alpha|_{\gamma(t)}$ is then
    $$\alpha|_{\gamma(t)}=\alpha|_p + t\nabla_v\alpha|_p+\mathcal{O}(t^2),$$
    and, similarly, the Taylor expansion of $d\alpha|_{\gamma(t)}$ is
    $$d\alpha|_{\gamma(t)}=d\alpha|_p+t\nabla_vd\alpha|_p + \mathcal{O}(t^2).$$
    Using that $\alpha$ and $d\alpha$ vanish at $p$, we deduce that
    \begin{align}\label{eq:taylor_contact}
        \alpha\wedge d\alpha|_{\gamma(t)}= t^2 \nabla_v\alpha \wedge \nabla_vd\alpha + \mathcal{O}(t^3).
    \end{align}
    from which the claim follows.
\end{proof}

We wish to point out that it was really only the symmetric part $$S_\alpha (v,w)=\frac{1}{2}\left(B_\alpha(v,w)+B_\alpha (w,v)\right)$$ which was used in Lemma \ref{lem:Taylor_contact_condition} and that $S_\alpha$ is positive definite if and only if $B_\alpha$ is. Nevertheless, the antisymmetric part of $B_\alpha$, which is generically nonzero, plays a big role in the narrative of this paper because it is precisely this antisymmetric part which distinguishes general near-contact forms from those that are curl eigenforms for some metric. Interestingly, any near-contact form admits a rescaling for which the antisymmetric part of $B_\alpha$ vanishes; see Proposition \ref{prop:near-contact is conformal to intrinsically Beltrami}.

As we will see, the Gray stability-type result Theorem \ref{thm:singularGray} will allow us to establish a normal form for the kernel of $\alpha$ in a punctured neighborhood of a near-contact singularity. More generally, the stability holds for global forms that fail to be contact along such nondegenerate singularities.

\begin{defi}\label{def:near_contact_form}
    A $1$-form $\alpha$ on an oriented $3$-manifold $M$ is called a \emph{positive (respectively, negative) near-contact form} if $\alpha\wedge d\alpha>0$ ($<0$) in the complement of $\alpha^{-1}(0)$, and if all zeros of $\alpha$ are positive (respectively, negative) near-contact singularities.
\end{defi}
Notice that as a consequence of Lemma \ref{lem:Taylor_contact_condition}, if $p$ is a near-contact singularity of a $1$-form $\alpha$, then $\alpha$ is a near-contact form in a sufficiently small neighborhood of $p$.
\medskip

Near-contact forms are a good candidate for generalizing contact topology and Reeb dynamics to forms with transverse zeros for two main reasons. First, the notion of being near-contact is manifestly invariant under pullback along diffeomorphisms and furthermore also under rescalings by nowhere-vanishing functions (as we will see in Proposition \ref{prop: near contact forms preserved by conformal rescaling}), leading to the notion of near-contact structure (Definition \ref{def:near_contact_structure}). Secondly, they have a kind of analogue to the Reeb field of a contact form. Notice that we cannot define the usual Reeb field of $\alpha$ in $M\setminus \alpha^{-1}(0)$, since it does not extend smoothly to all of $M$, and in fact there is no canonical vector field associated with $\alpha$. However, there is a unique well-defined vector field \emph{up to reparametrization}. Indeed, if we fix an ambient volume form $\mu$ in $M$ inducing the same orientation as $\alpha\wedge d\alpha$ in $M\setminus \alpha^{-1}(0)$, it determines a unique vector field $R_{\mu}$ via the equation $\iota_{R_\mu}\mu=d\alpha$. The zeros of $R_\mu$, which are non-degenerate, are the same as that of $\alpha$, it spans the kernel of $d\alpha$, and satisfies $\alpha(R_\mu)\geq 0$ with equality exactly along $\alpha^{-1}(0)$. Different choices of volume forms yield positive rescalings of $R_\mu$. This leads to the following definition.
\begin{defi}\label{def:Reeb_class}
Let $\alpha$ be a near-contact form. The set $\mathcal{R}_\alpha$ of positive reparametrizations of $R_\mu$, which only depends on $\alpha$, is called the \emph{Reeb class} of $\alpha$.
\end{defi}

A main motivation for studying near-contact forms comes from the fact that they generalize transversely vanishing rotational Beltrami forms, the metric-duals to transversely vanishing rotational Beltrami fields.

\begin{defi}\label{def:rotational_Beltrami_form}
Let $M$ be endowed with a Riemannian metric $g$. A $1$-form $\alpha$ is a rotational Beltrami form if $\star_g d \alpha =f \alpha$ with $f$ a non-vanishing smooth function. We call it positive or negative according to the sign of $f$. Its metric-dual vector field is a rotational Beltrami field.
\end{defi}

Similarly, when $f$ is constant, we get the notion of curl eigenform.

\begin{defi}\label{def:curl_eigenform}
A rotational Beltrami form with $f\equiv \lambda$ constant is called a curl eigenform. We call it positive or negative according to the sign of $\lambda$. Its metric-dual vector field is a curl eigenfield.
\end{defi}

\begin{convention}\label{conv:positive}
Reversing the orientation of $M$ changes the sign of the Hodge star $\star_g$ on $2$-forms, and hence of the curl operator. Consequently, $\alpha$ is a negative rotational Beltrami form (respectively, a negative curl eigenform) on $M$ if and only if it is a positive one, with proportionality factor $-f$ (respectively, eigenvalue $-\lambda$), on $M$ endowed with the opposite orientation. Similarly, since the signs of $\alpha\wedge d\alpha$ and of $B_\alpha$ are measured with respect to the orientation of $M$, a $1$-form is a negative near-contact form if and only if it is a positive near-contact form for the opposite orientation, and its Reeb class is not affected by this change. Throughout the rest of the paper, unless explicitly stated otherwise, all rotational Beltrami forms, curl eigenforms, and near-contact forms and structures are assumed to be positive, and we will often omit the qualifier. Every result has a negative analogue, which follows from the positive one by reversing the orientation of $M$.
\end{convention}

\begin{lemma}\label{lem:transverse}
    Let $\alpha$ be a transversely vanishing smooth $1$-form which is rotational Beltrami for a smooth Riemannian metric, then $\alpha$ is a near-contact form with symmetric $B_\alpha$, and the corresponding rotational Beltrami field belongs to its Reeb class.
\end{lemma}
\begin{proof}
    First, it directly follows from $\star_gd\alpha=f\alpha$ that the zero set of $\alpha$ and $d\alpha$ coincide, and that $\alpha\wedge d\alpha>0$ away from their zeros. Let us show that $d\alpha$ vanishes transversely along its zeros. Recall that two sections $X$ and $Y$ of a vector bundle $E\to M$ are transverse iff the submanifolds $X(M)$ and $Y(M)$ in $E$ are transverse. The map 
    $$f\,\star_g : \Lambda^1 T^\ast M \to \Lambda^2 T^\ast M$$ 
    is a smooth vector bundle isomorphism covering the identity, which is smoothly invertible, and hence a diffeomorphism. Consequently $f\star_g$ takes any pair of transverse submanifolds of $\Lambda^1 T^\ast M$ to a pair of transverse submanifolds of $\Lambda^2 T^\ast M$. Thus $\alpha \transv 0$ implies $d\alpha\transv 0$.  
    
    {Differentiating $d\alpha=f \star_g\alpha$ at $p$ one gets
    $$ \nabla d\alpha|_p= (\nabla f)|_p \star_g\alpha|_p+f(p) \star_g \nabla\alpha|_p=f(p) \star_g \nabla\alpha|_p.$$
    We then have
    \begin{align*}
        \nabla_u \alpha \wedge \nabla_v d\alpha|_p&= \nabla\alpha|_p(u)\wedge \nabla d\alpha|_p(v)\\
        &=f(p) \cdot \nabla\alpha|_p(u)\wedge \star_g \nabla\alpha|_p(v)\\
        &=f(p)\cdot g(\nabla\alpha|_p(u), \nabla\alpha|_p(v))\mu|_p,
    \end{align*}
    where $\mu$ is the Riemannian volume. It follows from the non-degeneracy of $D\alpha|_p$, the positivity of $f$, and the symmetry and positive definiteness of $g$ that $B_\alpha$ is positive definite and symmetric.}

    Lastly, if $X$ is the vector field such that $\alpha=g(X,\cdot)$, the Beltrami equation for $\alpha$ is equivalent to $\iota_{fX}\mu=d\alpha$
\end{proof}

{\begin{remark}\label{rem:transverse_C0}
Lemma \ref{lem:transverse} also holds in the following more general setting: Suppose that $\alpha$ is smooth, $f$ is continuous and positive, and the metric $g$ making $\alpha$ a rotational Beltrami form is only of $C^0$-regularity. The only modifications in the proof are the following. First, the map $f\star_g$ is then only a homeomorphism from $\Lambda^1T^*M$ to $\Lambda^2T^*M$. However, at a zero $p$ of $\alpha$, the continuity of $f\,\star_g$ gives the expansion
$$d\alpha(x)=f(x)\star_{g_x}\alpha(x)=f(p)\star_{g_p}(\nabla\alpha)|_p(x-p)+o(|x-p|)$$
in a chart around $p$, so $\nabla d\alpha|_p=f(p)\star_{g_p}\nabla\alpha|_p$. This shows directly that $d\alpha$ vanishes transversely at $p$, and the rest of the computation of $B_\alpha$ only involves $g_p$, so it is unchanged. The vector field $X=\alpha^{\sharp_g}$ is now only continuous and thus cannot be in $\mathcal{R}_{\alpha}$ (which only contains smooth vector fields by definition), however, it is a positive continuous multiple of any vector field in $\mathcal{R}_\alpha$.
\end{remark}
}
Thus rotational Beltrami forms constitute a distinguished subclass of near-contact forms. Understanding the relation between their contact geometry, contact topology, and the Riemannian geometric nature of these forms is the mathematical question at the heart of this paper. The following example shows that there are near-contact forms which are not rotational Beltrami for any metric.

\begin{example}[A transverse singularity not coming from a rotational Beltrami form] \label{ex:nonBeltrami}
    Consider the $1$-form 
    $$\alpha=xdx+ydy-2zdz+yz dx - (x+y)zdy,$$
    which satisfies
    $$D\alpha|_0= \begin{pmatrix}
        1 &0 & 0\\
        0 & 1 & 0 \\
        0 & 0 & -2
    \end{pmatrix}, \qquad \text{and} \qquad Dd\alpha|_0=\begin{pmatrix}
        1 & 1 & 0\\
        0 & 1 & 0\\
        0 & 0 & -2
    \end{pmatrix}.$$
    In particular, we have
    $$D\alpha|_0\cdot Dd\alpha|_0=\begin{pmatrix}
        1 & 1 & 0 \\
        0 & 1 & 0 \\
        0 & 0 & 4
    \end{pmatrix}.$$
    This is a positive definite matrix that is not symmetric. By Equation \eqref{eq:wedge_matrix_representation}, we know that the map \eqref{eq:infinitessimal_contact_map} is positive definite but not symmetric, so $\alpha$ is a near-contact form in a neighborhood of the origin. In particular, the form $\alpha$ is contact in a punctured neighborhood of the origin. Indeed, one can check directly that 
    \begin{align*}
        \alpha\wedge d\alpha&= \left[x^2+y^2+4z^2+xy\right]dx\wedge dy\wedge dz\\
        &= \left[\left(x+\frac{y}{2}\right)^2+\frac{3}{4}y^2+4z^2\right]dx\wedge dy\wedge dz,
    \end{align*}
    which is positive away from the origin. However, it is not a rotational Beltrami form for any metric by Lemma \ref{lem:transverse} and the fact that $D\alpha|_0\cdot Dd\alpha|_0$ is not a symmetric matrix.
\end{example}

\subsection{Near-contact structures}\label{subsec: infinitesimal properties of generic contact singularities}

An interesting property of near-contact forms is that they remain near-contact under positive rescalings.

\begin{prop}\label{prop: near contact forms preserved by conformal rescaling}
    Let $\alpha$ be a positive (negative) near-contact form, and $f$ a non-vanishing smooth function in $M$. Then $f\alpha$ is also a positive (negative) near-contact form.
\end{prop}
\begin{proof}
Let $\beta=f\alpha$ be the rescaled form. The fact that $\beta\wedge d\beta>0$ away from the singularities is immediate. For any $p\in \alpha^{-1}(0)$ we have
$$(\nabla_u\beta)_p=f(p)(\nabla_u\alpha)_p\quad \text{and} \quad
(\nabla_vd\beta)_p
=
df(p)\wedge(\nabla_v\alpha)_p
+f(p)(\nabla_vd\alpha)_p.
$$
Consequently,
$$B_\beta (u,u)
=
f(p)^2 B_\alpha (u,u),$$
at $p$ which implies that both $\beta$ and $d\beta$ vanish transversely at $p$ by Remark \ref{rem:matrix_representation} and that $B_\beta$ is definite, so $p$ is indeed a near-contact singularity for $\beta$.
\end{proof}

\noindent Proposition \ref{prop: near contact forms preserved by conformal rescaling} implies that we can meaningfully introduce the notion of a \textit{near-contact structure}.

\begin{defi}\label{def:near_contact_structure}
A \emph{positive (negative) near-contact structure} $\xi$ is a singular plane field which is the kernel of a positive (negative) near-contact form. 
\end{defi}
We will call the set of singularities of a near-contact structure its \emph{singular locus}. Singularities of positive or negative near-contact forms and structures fall into two different classes. This is observed in the following lemma, which was previously proven in \cite[Lemma 7]{ShevSmir}. In the following, for a diagonalizable automorphism, we call \emph{signature} the difference between the number of positive and negative eigenvalues.

\begin{lemma}\label{lem:signature}
Let $\alpha$ be a near-contact form on $M$ and let $p\in \alpha^{-1}(0)$. Having fixed some basis in $T_pM$ and its dual basis in $T_p^*M$, the map induced by $(\nabla \alpha)_p$ from $T_pM$ to $T_pM$ has signature $1$ or $-1$. This corresponds to $p$ being a zero of index $1$ respectively $2$.
\end{lemma}
\begin{proof}
We focus on the case of positive near-contact forms, as the negative one follows analogously. Fixing coordinates, by Remark \ref{rem:matrix_representation}, we know that $(D\alpha|_p)^t\cdot(Dd\alpha|_p)$ is positive definite. On the other hand, the fact that $d^2\alpha=0$ implies that the trace of $(Dd\alpha|_p)$ is zero. By contradiction, suppose that $S=D\alpha|_p$ were positive or negative definite, and write $Dd\alpha|_p=S^{-1}P$ with $P=S\cdot Dd\alpha|_p$, which has positive definite symmetric part. Then $\operatorname{tr}(Dd\alpha|_p)=\pm\operatorname{tr}\big(|S|^{-1/2}\,P\,|S|^{-1/2}\big)$, where $|S|=\pm S$, and this is nonzero because $|S|^{-1/2}P|S|^{-1/2}$ also has positive definite symmetric part. This contradicts $\operatorname{tr}(Dd\alpha|_p)=0$.
\end{proof}

\noindent In the upcoming definition, recall that $M$ is oriented and induces orientations on $T_pM$ and $T^*_pM$.

\begin{defi}\label{def:signature}
Let $p$ be a near-contact singularity of a $1$-form $\alpha$ in $M$. We call the \textit{signature} of $\alpha$ at $p$ the signature of $(\nabla \alpha)_p$. 
\end{defi}

\subsection{Other models for singular contact forms and structures}\label{subsec: a comparison with other models for singular}

First, as we mentioned earlier, there is an equivalent definition of near-contact form introduced in \cite{ShevSmir}. 

\begin{defi}[Alternative definition of near-contact form]\label{def:alternative_near_contact_form}
Let $M$ be an oriented $3$-manifold. A $1$-form $\alpha$ is called a positive (negative) near-contact form if it vanishes transversely, and there exists a positive volume form $\mu$ on $M$ such that $\alpha\wedge d\alpha=f\mu$, with $f$ a non-negative (non-positive) function whose only zeros, which coincide exactly with the zeros of $\alpha$, are nondegenerate minima (maxima).
\end{defi}

\begin{prop}\label{prop:equivalent_definitions}
    Definitions \ref{def:near_contact_form} and \ref{def:alternative_near_contact_form} are equivalent.
\end{prop}
\begin{proof}
We will treat the positive case only. Suppose $\alpha$ is a positive near-contact form according to Definition \ref{def:near_contact_form}. Let $\mu$ be any volume form inducing the orientation of $M$ and write $\alpha\wedge d\alpha=f\mu$. First, notice that $f$ vanishes exactly at the zeros of $\alpha$, and these are global minima of $f$ because $\alpha$ is a positive contact form away from the singularities. Let $p$ be any of the singularities of $\alpha$. Recall that by Equation \eqref{eq:taylor_contact}, we have
$$D^2 (\alpha\wedge d\alpha)_p (u,u)= 2(\nabla_u \alpha\wedge \nabla_ud\alpha)_p.$$
On the other hand
$$D^2(f\mu)_p(u,u)= D^2f_p(u,u)\mu_p.$$
It follows that $D^2f_p$ is positive definite if and only if $B_\alpha$ is. This shows that $\alpha$ is a positive near-contact form according to Definition \ref{def:alternative_near_contact_form}. 
\medskip

Conversely, since $\alpha$ is contact away from $p$, we have $\nabla(\alpha\wedge d\alpha)|_p=0$, which implies by the non-degeneracy of $\nabla \alpha|_p$ and the fact that $\alpha|_p=0$ that $d\alpha|_p=0$. By the same computation as before, using that $p$ is a (nondegenerate) minimum of $f$, we know that $B_\alpha$ is positive definite, and by Remark \ref{rem:matrix_representation} it follows that $Dd\alpha|_p$ is nondegenerate.
\end{proof}

\paragraph{Other notions of singular contact forms.} Several other singular generalizations of contact forms appear in the literature; we briefly compare them with near-contact forms.

\emph{Zeros with non-vanishing differential.} Jiang, Minh and Zung \cite{JiangMinhZung}, extending earlier work of Lychagin, Webster and Zhitomirskii, study local normal forms of \emph{singular contact forms}, i.e.\ $1$-forms vanishing at a point $O$ and satisfying the contact condition almost everywhere. Their nondegeneracy assumption requires $(d\alpha)^n(O)\neq0$; in dimension three, $d\alpha(O)\neq0$, so that the set where the contact condition fails is a hypersurface through $O$. In contrast, a transversely vanishing $1$-form which is contact on a punctured neighborhood of a zero $p$ necessarily satisfies $d\alpha(p)=0$ (see the proof of Proposition \ref{prop:equivalent_definitions} and \cite[Lemma 6]{ShevSmir}), so these two classes of singularities are disjoint.

\emph{Isolated zeros.} Pollard and Alexander \cite{PollardAlexander} consider $1$-forms on $3$-manifolds that are contact away from finitely many zeros, motivated by point defects in chiral liquid crystals. Near-contact forms are precisely those among them whose zeros are transverse and for which the map \eqref{eq:symmetric_wedge_intro} is definite.

\emph{Singularities along hypersurfaces.} In the setting of $b$-symplectic and Jacobi geometry, Miranda and Oms \cite{MirandaOms} introduced $b^m$-contact forms, which are smooth sections of the $b^m$-cotangent bundle associated with a hypersurface $Z$ but singular as ordinary forms along $Z$, and whose kernel is tangent to $Z$. Besides the fact that these are not ordinary smooth forms, these structures degenerate along a hypersurface rather than at isolated points, and are thus of a different nature than near-contact forms.

\section{Gray stability for near-contact structures}\label{sec:singularGray}

The aim of this section is to prove the following theorem, which is a generalization of Gray stability to the setting of near-contact forms. This can be seen as a contact version of Moser's path method for self-dual harmonic $2$-forms in four-manifolds, proven by Honda \cite{honda2004local}.

\begin{theorem}\label{thm:singularGray}
Suppose that $(\alpha_s)_{s\in [0,1]}$ is a {smooth} path of near-contact forms on a closed manifold $M$ vanishing on a collection of points $\{p_i(s)\}_{i=1}^n$. Then, there exists {a once-differentiable} family of homeomorphisms 
$$\varphi_s: M \longrightarrow M, \enspace s\in[0,1],$$
such that $\varphi_s(p_i (s))=p_i(0)$, $\varphi_s$ is smooth in $M\setminus \bigcup_{i=1}^n\{p_i(s)\}$, and 
$$\varphi_s^*\alpha_0=h_s\,\alpha_s$$ on $M\setminus\{p_i(s)\}$, where $h_s$ is a positive function, smooth away from the $p_i(s)$, and which extends continuously to $M$ with $h_s(p_i(s))=1$ and $|h_s(x)-1|\le C\,d(x,p_i(s))$ near $p_i(s)$, where $C$ is a constant independent of $s$.
\end{theorem}

The key difficulty in proving Theorem \ref{thm:singularGray} relative to the standard Gray stability theorem for contact structures is that the existence of the near-contact singularities introduces a source of noncompactness that does not exist in the non-singular setting. Interestingly, just as is the case for Theorem \ref{thm: local model for Beltrami}, the lack of regularity in the statement of Theorem \ref{thm:singularGray} is not an artifact of the particular technique we use for remedying the noncompactness but instead there exists a real obstruction to improving it. This is because the regularity of the contactomorphism in Theorem \ref{thm: local model for Beltrami} is inherited from that of Theorem \ref{thm:singularGray}, and we prove in Proposition \ref{prop: local invariant for near-contact forms} that there is an invariant of near-contact singularities which precludes the existence of a version of Theorem \ref{thm:singularGray} for which the homeomorphisms $\varphi_s$ can be upgraded to be of regularity $C^1$ at $p_i(s)$ (see Remark \ref{rem:preclude}).  

An interesting feature of the proof of Theorem \ref{thm:singularGray} is that the fact that the $\alpha_s$ vanish transversely is absolutely crucial, as the construction of the isotopy makes explicit use of that property and it is not clear to the authors how to avoid this assumption.

\begin{remark}\label{rmk:singularGray_openmanifold}
    If $M$ is not closed, the proof still applies near the zeros: If $(\alpha_s)_{s\in[0,1]}$ is a smooth path of near-contact forms vanishing on finitely many points $\{p_i(s)\}$, there exist a neighborhood $U$ of $\{p_i(0)\}$ and a continuous family of homeomorphisms $\psi_s: U \rightarrow V_s$ onto neighborhoods $V_s$ of $\{p_i(s)\}$, with $\psi_s(p_i(0))=p_i(s)$, smooth away from the $p_i(0)$, and such that $\psi_s^*\alpha_s=h_s\alpha_0$, where $h_s$ is positive, continuous, equal to $1$ at the $p_i(0)$, and satisfies $h_s-1=\mathcal{O}(d(x,p_i(0)))$. If the zeros do not move, $\psi_s$ fixes them, and $h_s$ is as in Lemma \ref{lem:conformal_factor}.
\end{remark}

The plan for this section is as follows. In Subsection \ref{subsec: a modified gray stability} we set up the modification of the usual Gray stability trick by the flow of a suitable contact Hamiltonian which cancels the vector field generating the Gray isotopy at the near-contact singularities to zeroth order. This cancellation cannot in general be expected to work to higher order, which explains the lack of regularity in Theorem \ref{thm:singularGray}. The relevant properties of this modified vector field close to the near-contact singularities are established in Subsection \ref{subsec: behavior of the modified vector field}, and the proof of Theorem \ref{thm:singularGray} then follows easily in Subsection \ref{subsec: proof of singular gray}. Finally, we use Theorem \ref{thm:singularGray} to prove that generic curl eigenforms have stable contact topology in Subsection \ref{subsec: first application}.

\subsection{Modifying Gray's flow close near the singularities}\label{subsec: a modified gray stability}
 
In this subsection and the next one, we first deal with a simplified version of the problem, from which we will deduce the general case. Let $\alpha_s$ be a path of near-contact forms in $M$ with common zeroes $p_1,...,p_N$. We will try to apply Moser's path method to $\alpha_s$, and analyze what happens near $p$.

\paragraph{Recalling the standard Moser's path method in contact geometry.}
Let us first recall the standard Gray's stability of contact structures using Moser's path method following \cite[Section 2.2]{Geiges}. Let $\alpha_s$, with $s\in [0,1]$, be a path of contact forms in a manifold $M$, and denote by $R_s$ the Reeb field of $\alpha_s$. We look for a flow $\psi_s$ satisfying
$$\psi_s^*\alpha_s=h_s\alpha_0,$$
for some positive smooth function $h_s\in C^\infty(M)$. Denoting by $X_s$ the vector field generating $\psi_s$ and taking the $s$ derivative of the previous equation, one gets
$$\psi_s^*(\dot \alpha_s + d(\alpha_s(X_s))+\iota_{X_s}d\alpha_s)=\psi_s^*(\mu_s\alpha_s),$$
where $\mu_s=\pp{(\log h_s)}{s}\circ \psi_s^{-1}$. We now make the ansatz of choosing $X_s$ in $\xi_s$, that is $\alpha_s(X_s)=0$, which then implies that it is enough to solve
\begin{equation}\label{eq:standard_Gray}
    \dot \alpha_s + \iota_{X_s}d\alpha_s=\mu_s\alpha_s.
\end{equation}
Contracting this equation with $R_s$ gives $\mu_s=\dot \alpha_s(R_s)$ from which we learn that $\mu_s\alpha_s-\dot \alpha_s$ vanishes on $R_s$. Thus Equation \ref{eq:standard_Gray} admits a unique solution $X_s\in \xi_s$ of the equation above.
\medskip

When $\alpha_s$ are near-contact that have zeros at a collection of points $\{p_i\}_{i=1}^N$, one can still uniquely determine the vector field $X_s$ in $M\setminus \bigcup_{i=1}^N \{p_i\}$ by the same argument. However, such a vector field will in general not have a flow that is defined up to time one as needed due to the non-compactness of $M\setminus \bigcup_{i=1}^N \{p_i\}$. We tackle this issue by modifying $X_s$ into a different vector field with a controlled behavior near the singularity.

\paragraph{Modifying Gray's flow using a contact Hamiltonian}
We apply the construction to the path of contact forms $\alpha_s|_{M\setminus \bigcup_{i=1}^N \{p_i\}}$ in the open manifold $M\setminus \bigcup_{i=1}^N \{p_i\}$. The key observation that allows us to modify Gray's stability trick, to try to make the flow defined up to time one, is the following: One may change the vector field $X_s$ by adding another term $X_{H_s}$ that is, at any given time $s$, Hamiltonian with respect to $\alpha_s$. Indeed, such a flow infinitesimally preserves $\xi_s$, and we will observe in Lemma \ref{lem:Gray_Hamiltonian} below that $X_s+X_{H_s}$ still generates a $1$-parameter family of contactomorphisms. The choice of $s$-dependent Hamiltonian is crucial for the resulting flow to possibly be defined up to time one. Inspired by \cite{honda2004local}, we construct the Hamiltonian $H_s$ as follows. Fix a chart around a zero $p$ of $\alpha_s$, identifying it with $0\in\R^3$, a radius $r_1>0$ such that $\overline{B_{2r_1}(0)}$ lies in the chart, and a bump function $g$ equal to $1$ on $B_{r_1}(0)$ and to $0$ outside $B_{2r_1}(0)$. We first consider
$$f_s= -\frac{1}{2}\sum_{i,j} \pp{\dot \alpha_{s,i}}{x_j}(0)\cdot x_ix_j,$$
which is chosen so that $-\dot \alpha_s$ and $df_s$ coincide up to first order. Then, we set 
\begin{equation}\label{eq:Hamiltonian_Gray}
    H_s:=gf_s, \qquad s\in[0,1].
\end{equation} 
Let $X_{H_s}$ be the Hamiltonian vector field of $H_s$ with respect to $\alpha_s$ and consider the vector field 
$$Z_s:=X_s+X_{H_s},$$ 
which we decompose as 
\begin{equation}\label{eq:Zs_decomposition}
    Z_s=B_sR_s+Y_s,
\end{equation}
where $R_s$ is the Reeb field of the contact form $\alpha_s|_{M\setminus \{p\}}$ and $Y_s\in \xi_s$. In the complement of the $p_i$, $Z_s$ satisfies
\begin{equation}\label{eq:modified_Gray}
\begin{cases}
\alpha_s(Z_s)=f_s \iff B_s=f_s,\\
\iota_{Y_s}d\alpha_{s}|_{\xi_s}=(-\dot\alpha_s - df_s)|_{\xi_s}.
\end{cases}
\end{equation}
Let us now check that $Z_s$ still identifies the corresponding plane fields.
\begin{lemma}\label{lem:Gray_Hamiltonian}
Let $M$ be a manifold and $\alpha_s$ a path of contact forms. Let $X_s$ be the corresponding Gray stability vector field, i.e., the unique vector field satisfying Equation \eqref{eq:standard_Gray} with $\mu_s=\dot\alpha_s(R_s)$ and $X_s\in \xi_s$. For any $s$-dependent Hamiltonian function $H_s$, the flow $\psi_s$ of $Z_s=X_s+X_{H_s}$, where $X_{H_s}$ is the $s$-dependent contact Hamiltonian of $H_s$ with respect to $\alpha_s$, satisfies $\psi_s^*\alpha_s=h_s\alpha_0$ for some positive function $h_s$.
\end{lemma}
\begin{proof}
Recall that $X_s$ satisfies $\dot \alpha_s + \iota_{X_s}d\alpha_s=\mu_s\alpha_s$, where $\mu_s=\dot \alpha_s(R_s)$. On the other hand, by the definition of contact Hamiltonian vector field, we know that 
$$
\begin{cases}
\alpha_s(X_{H_s})=H_s,\\ 
\iota_{X_{H_s}}d\alpha_s=dH_s(R_s)\alpha_s-dH_s.
\end{cases}$$
Hence, we have
\begin{align*}
    \dot \alpha_s+d(\alpha_s(Z_s))+\iota_{Z_s}d\alpha_s&=\dot \alpha_s + dH_s-\dot \alpha_s +\mu_s\alpha_s+dH_s(R_s)\alpha_s-dH_s\\
    &=\left(\mu_s+dH_s(R_s)\right)\alpha_s,
\end{align*}
and we claim that this implies the desired statement. Indeed, if $\psi_s$ is the flow of $Z_s$, we deduce from the previous equality that
$$\psi_s^*\left( \dot \alpha_s+d(\alpha_s(Z_s))+\iota_{Z_s}d\alpha_s\right)=\nu_s \psi_s^*\alpha_s,$$
{where $\nu_s=(\mu_s+dH_s(R_s))\circ\psi_s$. Using that the left hand side is equal to $\frac{d}{ds}(\psi_s^*\alpha_s)$, we conclude that $\frac{d}{ds}(\psi_s^*\alpha_s)=\nu_s\,\psi_s^*\alpha_s$, and hence $\psi_s^*\alpha_s=h_s\alpha_0$ with
\begin{equation}\label{eq:conformal_factor}
h_s=\exp\Big(\int_0^s(\mu_\sigma+dH_\sigma(R_\sigma))\circ\psi_\sigma\,d\sigma\Big).
\end{equation}}
\end{proof}
Notice that $Z_s$ can be extended as $X_s$ away from a neighborhood of $p$, patching up to a vector field $Z_s$ in $M\setminus \bigcup_{i=1}^N\{p_i\}$. We are left with proving that the flow of $Z_s$ is defined up to time one near any zero $p$, and that it extends continuously to those singular points by keeping them fixed.

\subsection{Behavior of the modified field close to the singularities}\label{subsec: behavior of the modified vector field}

In order to show the extendibility of $Z_s=B_sR_s+Y_s$ to the singular points, we will study each term separately. Once again we focus on the neighborhood of a zero $p$ of $\alpha_s$ which we identify with the origin in a local chart $U$. By compactness of $[0,1]$, all the constants below can be chosen independently of $s$. We first establish a linear bound for the term $B_sR_s$.

\begin{lemma}\label{lem:linearbound2}
    In $U$, we have
    $$\norm{B_sR_s}\leq K_1\norm{x}$$
    for some constant $K_1>0$.
\end{lemma}
\begin{proof}
Fix the volume form $\Omega=dx_1\wedge dx_2 \wedge dx_3$, and let ${L_s}$ be defined by $\iota_{L_s} \Omega=d\alpha_s$. This is a vector field with a non-degenerate zero at $p$, which belongs to the Reeb class of $\alpha_s$. If we write $d\alpha_s$ in coordinates (where we drop the $s$-dependence in the $\beta_i$ terms to lighten the notation)
\begin{equation}\label{eq:dalpha_s}
d\alpha_s= \beta_1 dx_2\wedge dx_3 + \beta_2 dx_3\wedge dx_1 + \beta_3 dx_1\wedge dx_2,
\end{equation}
where $\beta_1,\beta_2,\beta_3$ are $s$-dependent smooth functions in $U$. Taking the Taylor expansion of each coefficient, we can write $d\alpha_s$ as 
$$d\alpha_s=(d\alpha_s)_1+(d\alpha_s)_r,$$
where $(d\alpha_s)_1$ is the linear part
\begin{align*}
(d\alpha_s)_1= &\left(\sum_{j=1}^3 \pp{\beta_1}{x_j}|_0 x_j \right) dx_2\wedge dx_3 +\left(\sum_{j=1}^3 \pp{\beta_2}{x_j}|_0 x_j \right) dx_3\wedge dx_1\\ &+ \left(\sum_{j=1}^3 \pp{\beta_3}{x_j}|_0 x_j \right) dx_1\wedge dx_2.    
\end{align*} 
and $(d\alpha_s)_r$ is the remainder, which is second order in $\norm{x}$.
If we let $L_{s,0}$ be defined by $\iota_{L_{s,0}}\Omega=(d \alpha_s)_1$, then we have
$$L_s=L_{s,0}+L_{s,r},$$
where $L_{s,0}$ is first order in $\norm{x}$ and $L_{s,r}$ is second order in $\norm{x}$.
As we did for $d\alpha_s$, we write $\alpha_s$ as
$$\alpha_s= \alpha_{s,1} dx_1 + \alpha_{s,2} dx_2 +\alpha_{s,3} dx_3,$$
where $\alpha_{s,i}$ are $s$-dependent smooth functions in $U$, and split it as
$$\alpha_s= (\alpha_s)_1 + (\alpha_s)_r,$$
where 
$$(\alpha_s)_1= \left(\sum_{j=1}^3 \pp{\alpha_{s,1}}{x_j}|_0 x_j \right) dx_1 + \left(\sum_{j=1}^3 \pp{\alpha_{s,2}}{x_j}|_0 x_j \right) dx_2 + \left(\sum_{j=1}^3 \pp{\alpha_{s,3}}{x_j}|_0 x_j \right) dx_3,$$
and $(\alpha_s)_r$ is the remainder. Then
$$\alpha_s(L_s)=(\alpha_s)_1(L_{s,0}) + (\alpha_s)_1(L_{s,r}) + (\alpha_s)_r(L_s),$$
which implies that
$$\alpha_s(L_s)=(\alpha_s)_1(L_{s,0})+\mathcal{O}(\norm{x}^3),$$
where the first term is second order in $\norm{x}$. If we consider $W_s=\frac{L_s}{\norm{L_s}}$, it follows that
$$\alpha_s(W_s)= (\alpha_s)_1(W_s)+\mathcal{O}(\norm{x}^2).$$
Throughout the rest of the proof, we denote $D\alpha_s|_0$ and $Dd\alpha_s|_0$ simply by $D\alpha_s$ and $Dd\alpha_s$. Since $\norm{L_s}$ is equal to $\norm{Dd\alpha_s x}$ to first order, we have that the leading term can be written as
$$(\alpha_s)_1(W_s)=\frac{x^TD\alpha_s^T D d\alpha_s x}{\norm{D d\alpha_s x}}.$$
The latter satisfies
$$\frac{x^TD\alpha_s^T D d\alpha_s x}{\norm{D d\alpha_s x}}=\norm{x}\frac{\frac{x^T}{\norm{x}}D\alpha_s^TDd\alpha_s \frac{x}{\norm{x}}}{\left\lVert Dd\alpha_s \frac{x}{\norm{x}}\right\rVert}.$$
Recall that by Remark \ref{rem:matrix_representation}, the quadratic form defined by $D\alpha_s^TDd\alpha_s$ is positive definite, since $p$ is a near-contact singularity. Thus, we can find a constant $a_1>0$ such that
$$\frac{x^T}{\norm{x}}D\alpha_s^TDd\alpha_s \frac{x}{\norm{x}}\geq a_1$$
in $U$, and since $Dd\alpha_s$ is non-singular, there is a constant $a_2$ such that
$$\left\lVert Dd\alpha_s \frac{x}{\norm{x}}\right\rVert\leq a_2.$$
Combining these two inequalities, we obtain
$$(\alpha_s)_1(W_s)\geq a_0\norm{x},$$
where $a_0=\frac{a_1}{a_2}$. After possibly picking a slightly smaller $a_0$, we can assume that $\alpha_s(W_s)$ satisfies
\begin{equation}\label{eq:alpha_s_W_s}
    \alpha_s(W_s)\geq a_0\norm{x}
\end{equation}
near $0$. We can finally estimate the norm of $R_s$ by
\begin{equation}\label{eq:bound1}
    \norm{R_s}=\left\lVert \frac{L_s}{\alpha_s(L_s)}\right\rVert=\frac{1}{\alpha_s(W_s)} \leq  C_1\frac{1}{\norm{x}},
\end{equation}
where $C_1=\frac{1}{a_0}$. Recall that the function $f_s$ is defined by 
$$f_s= - \frac{1}{2}\sum_{i,j} \pp{\dot \alpha_{s,i}}{x_j}(0)\cdot  x_ix_j,$$ 
we can thus find a constant $C_2>0$ such that in $U$ we have
\begin{equation}\label{eq:bound2}
    |f_s| \leq C_2\norm{x}^2.
\end{equation}
Using Equations \eqref{eq:bound1} and \eqref{eq:bound2}, we conclude that
$$\norm{B_sR_s}^2=f_s^2\norm{R_s}^2\leq (C_1C_2)^2\norm{x}^2,$$
from which the lemma follows.
\end{proof}

We now establish the same type of bound for $Y_s$.

\begin{lemma}\label{lem:linearbound1}
In $U$, we have 
$$\norm{Y_s}\leq K_2\norm{x}$$
for some constant $K_2>0$.
\end{lemma}
\begin{proof}
Recall that the vector field $Y_s$ satisfies the equation
\begin{equation}\label{eq:Ys1}
\iota_{Y_s}d\alpha_s|_{\xi_s}= (-\dot\alpha_s-df_s)|_{\xi_s}.    
\end{equation}
By construction of $f_s$, the right-hand side vanishes to second order, namely, in $U$ we have
\begin{equation}\label{eq:alpha_dfs_secondorder}
    \norm{\dot\alpha_s+df_s}\le C_3\norm{x}^2.
\end{equation}
Since $\iota_{Y_s}d\alpha_s$ vanishes on $R_s$, Equation \eqref{eq:Ys1} means that, as $1$-forms, $\iota_{Y_s}d\alpha_s=\rho_s-\rho_s(R_s)\alpha_s$ with $\rho_s=-\dot\alpha_s-df_s$. By Equations \eqref{eq:alpha_dfs_secondorder} and \eqref{eq:bound1}, $\rho_s(R_s)\alpha_s=\mathcal{O}(\norm{x}^2)$ as well.
On the other hand, since $d\alpha_s$ vanishes transversely at $0$, we write it as in Equation \eqref{eq:dalpha_s}, and we have that $\beta_i(0)=0$ and the function
\begin{align*}
    F: U\subset \mathbb{R}^3 &\longrightarrow \mathbb{R}^3\\
                x &\longmapsto (\beta_1,\beta_2,\beta_3),
\end{align*}
has a non-singular differential $DF$ at $0$. If we write $Y_s$ as an arbitrary vector field 
$$Y_s= b_1 \pp{}{x_1} + b_2 \pp{}{x_2} + b_3 \pp{}{x_3},$$
then the previous identity and Equation \eqref{eq:alpha_dfs_secondorder} can be combined into
\begin{equation}\label{eq:eq1}
    \begin{cases}
        \beta_2b_3 - \beta_3b_2= \mathcal{O}(\norm{x}^2)\\
        \beta_3b_1 - \beta_1b_3= \mathcal{O}(\norm{x}^2)\\
        \beta_1b_2 - \beta_2b_1 = \mathcal{O}(\norm{x}^2)
        \end{cases}
\end{equation}
Introduce the matrix
$$V(F)=\begin{pmatrix} 
0 & -\beta_3 & \beta_2 \\
\beta_3 & 0 & -\beta_1 \\
-\beta_2  & \beta_1 & 0
\end{pmatrix}$$
and then the system is written as 
\begin{equation}\label{eq:eq2}
    V(F) \cdot b^T=\mathcal{O}(\norm{x}^2),
\end{equation}
where $b=(b_1, b_2, b_3)$. 

For any vector $v$, we have $V(F)v=F\times v$, where $F=(\beta_1,\beta_2,\beta_3)$ is the coordinate expression of $L_s$ (defined as in the proof of Lemma \ref{lem:linearbound2}). Hence $V(F)$ has a one-dimensional kernel in $U\setminus \{p\}$, spanned by $L_s$, or equivalently by $R_s$, and $\norm{V(F)v}=\norm{F}\norm{v}\sin\angle(v,R_s)$. For $v\in\xi_s$, this sine is bounded below by the cosine of the angle between $R_s$ and the normal vector to $\xi_s$, that is, by $\alpha_s(W_s)/\norm{\alpha_s}$ (where $W_s$ is as in the proof of Lemma \ref{lem:linearbound2}), which is bounded below by a constant $c_1>0$ since $\alpha_s(W_s)\geq a_0\norm{x}$ by Equation \eqref{eq:alpha_s_W_s} and $\norm{\alpha_s}=\mathcal{O}(\norm{x})$. Hence
$$\norm{V(F)v}\geq c_1 \norm{F}\norm{v}, \qquad \text{for all} \qquad v\in \xi_s.$$
Using that $DF|_0$ is nondegenerate, there is some constant $c_2>0$ such that
$$\norm{F}\geq c_2 \norm{x}, $$
from which we deduce that
$$\norm{V(F)b^T}\geq c\norm{x}\norm{b},$$
where $c=c_1c_2$. Combining this equation with Equation \eqref{eq:eq2}, we conclude that
$$\norm{Y_s}=\norm{b}\leq K_2\norm{x}$$
for some constant $K_2>0$.
\end{proof}

As a consequence of the previous lemmas, we can extend $Z_s$ to all of $M$ by setting $Z_s(p):=0$, and this yields a continuous vector field on $M$, smooth away from $p$. We keep the notation $Z_s$ for the extension. 

We are now ready to prove that the flow of $Z_s$ is defined up to time one near $p$.

\begin{prop}\label{prop:Zs_flow_defined}
There are constants $K,r_0>0$, independent of $s$, such that for every $s_0\in[0,1]$ and every $x_0$ with $0<\norm{x_0}<r_0e^{-K}$, the integral curve $x(s)$ of $Z_s$ with $x(s_0)=x_0$ exists for all $s\in[0,1]$, stays in $B_{r_0}(0)\setminus\{0\}$, and satisfies
\begin{equation}\label{eq:twosided}
\norm{x_0}e^{-K|s-s_0|}\le\norm{x(s)}\le\norm{x_0}e^{K|s-s_0|}.
\end{equation}
\end{prop}
\begin{proof}
 By Lemmas \ref{lem:linearbound1} and \ref{lem:linearbound2}, there is a constant $K=K_1+K_2$ such that 
\begin{equation}\label{eq:linearbound}
    \norm{Z_s}\leq K\norm{x} \quad \text{in}\quad U.
\end{equation}
Choose $r_0>0$ so that $B_{r_0}$ lies where the bump function $g$ equals $1$. Let $x_0$ be an initial condition in $U\setminus \{p\}$, and $x(s)$ the solution to the ODE defined by $Z_s$, namely, we know
\begin{equation*}
    \begin{cases}
        \pp{x(s)}{s}=Z_s(x(s)),\\
        x(s_0)=x_0.
    \end{cases}
\end{equation*}
Its norm evolves according to 
$$\pp{\norm{x(s)}}{s}=\frac{\langle Z_s(x(s)), x(s)\rangle}{\norm{x(s)}}.$$
By the Cauchy--Schwarz inequality, we deduce that
$$\left|\pp{\norm{x(s)}}{s}\right|\leq  \norm{Z_s(x(s))},$$
and using Equation \eqref{eq:linearbound} we conclude that
$$-K\norm{x(s)}\leq \pp{\norm{x(s)}}{s}\leq  K \norm{x(s)}.$$
Integrating, we get $\norm{x_0}e^{-K|s-s_0|}\leq \norm{x(s)}\leq\norm{x_0}e^{K|s-s_0|}$ for all $s$ such that $x(s)\in U$. If $\norm{x_0}<r_0e^{-K}$, where $B_{r_0}(0)\subset U$, the right inequality shows that $x(s)$ stays in $U$ for $s\in[0,1]$, and the left one that it cannot reach $p$. Hence the solution exists on $[0,1]$ and satisfies \eqref{eq:twosided}.
\end{proof}
{We deduce the following.
\begin{lemma}\label{lem:conformal_factor}
Let $\psi_s$ be the flow of $Z_s$ and let $h_s$ be as in \eqref{eq:conformal_factor}. There is $C>0$, independent of $s\in[0,1]$, such that $|\log h_s(x)|\le C\norm{x}$ for $\norm{x}<r_0e^{-K}$. In particular, $h_s$ extends continuously to $p$ with $h_s(p)=1$, and $h_s-1=\mathcal{O}(\norm{x})$.
\end{lemma}
\begin{proof}
Near $p$ we have $H_s=f_s$, so $\mu_s+dH_s(R_s)=(\dot\alpha_s+df_s)(R_s)$. By Equations \eqref{eq:alpha_dfs_secondorder} and \eqref{eq:bound1}, which hold uniformly in $s$, we have
$$|(\dot\alpha_s+df_s)(R_s)|\le\norm{\dot\alpha_s+df_s}\,\norm{R_s}\le C_3\norm{x}^2\cdot\frac{C_1}{\norm{x}}=C_1C_3\norm{x},$$
and since, by Equation \eqref{eq:twosided}, $\norm{\psi_\sigma(x)}\le e^{K}\norm{x}$ for $\sigma\in[0,s]$, we conclude that
$$|\log h_s(x)|\le\int_0^s C_1C_3\norm{\psi_\sigma(x)}\,d\sigma\le C_1C_3e^K\norm{x}.$$
\end{proof}
}

\subsection{Proof of Theorem \ref{thm:singularGray}}\label{subsec: proof of singular gray}

First, let us reduce to the case where the $\alpha_s$ all have the same fixed zero set.  Indeed, because the family $\alpha_s$ is smooth and the forms vanish transversely, the points $p_i(s)$ vary smoothly with the parameter $s$, see the proof of Proposition \ref{prop: submanifold property of TNC forms}. We can thus choose an isotopy $\theta_s: M\rightarrow M$, with $s \in [0,1]$, such that $\theta_s(p_i(0))=p_i(s)$ and $\theta_0=\mathrm{id}$. The forms $\beta_s=\theta_s^*\alpha_s$ all have the same zero set $\{p_i(0)\}_{i=1}^N$. Supposing that the theorem holds for forms with a fixed zero set, there exists a family of homeomorphisms $\phi_s:M\rightarrow M$ satisfying the claimed properties for the family $\beta_s$. Then the family of homeomorphisms $\varphi_s=\phi_s\circ (\theta_s)^{-1}$ is smooth away from the $p_i(s)$, satisfies $\varphi_s(p_i(s))=\phi_s(p_i(0))=p_i(0)$ and
\begin{align*}
    (\varphi_s|_{M\setminus \bigcup_{i=1}^n \{p_i(s)\}})^*(\ker \alpha_0)&= (\theta_s^{-1})^*\circ(\phi_s|_{M\setminus \bigcup_{i=1}^n \{p_i(0)\}})^*(\ker \beta_0)\\
    &= (\theta_s^{-1}|_{M\setminus \bigcup_{i=1}^n \{p_i(s)\}})^*(\ker \beta_s)
    \\
    &=\ker \alpha_s,
\end{align*}
as we wanted. {If $\phi_s^*\beta_0=k_s \beta_s$, then $\varphi_s^*\alpha_0=(k_s \circ\theta_s^{-1})\,\alpha_s$, and $k_s \circ\theta_s^{-1}$ has the claimed properties at $p_i(s)$, since $\theta_s$ is a smooth isotopy.}
\medskip

Let us then prove the statement under the assumption that the zero set of $\alpha_s$ is $\bigcup_{i=1}^N \{p_i\}$, for a fixed set of points $p_i\in M$. Consider, as we did in Section \ref{subsec: a modified gray stability}, the corresponding Gray-stability vector field $X_s$, whose flow need not be defined up to time one because of the non-compactness of $M\setminus \bigcup_{i=1}^N \{p_i\}$. In small disjoint neighborhoods $U_i$ of the $p_i$, we consider the Hamiltonians $H_s^i$ as defined in Equation \eqref{eq:Hamiltonian_Gray}. Lastly, we consider the modified Gray-stability vector field
$$Z_s= X_s + X_{H_s},$$
where $H_s=\sum_{i=1} H_s^i$.
{Let $\psi_s$ be the flow of $Z_s$ and set $\varphi_s:=\psi_s^{-1}$. By Proposition \ref{prop:Zs_flow_defined}, applied forward and backward in time and to each $p_i$, every integral curve in $M\setminus\{p_i\}$ exists on $[0,1]$ (otherwise it would accumulate at some $p_i$, hence enter $B_{r_0e^{-K}}(p_i)$ at some time $s_0$, and then Equation \eqref{eq:twosided} would keep it at distance at least $\norm{x(s_0)}e^{-K}>0$ from $p_i$), and $\psi_s$, hence $\varphi_s$, extends to a homeomorphism of $M$ fixing the $p_i$ and smooth away from them. {This flow gives a once-differentiable (in the parameter) family of homeomorphisms, since its derivative is given by a continuous vector field $Z_s$.} By Lemma \ref{lem:Gray_Hamiltonian}, $\psi_s^*\alpha_s=\tilde h_s\alpha_0$, where $\tilde h_s$ denotes the function given by Equation \eqref{eq:conformal_factor}, that is, $\varphi_s^*\alpha_0= h_s\alpha_s$ with $h_s=1/(\tilde h_s\circ\varphi_s)$. Lemma \ref{lem:conformal_factor} applied to $\tilde h_s$, for $x$ close enough to the $p_i$, together with $\norm{\varphi_s(x)}\le e^K\norm{x}$, gives $h_s(p_i)=1$ and $h_s-1=\mathcal{O}(\norm{x})$. }

\subsection{First application: Topological stability of generic curl eigenforms}\label{subsec: first application}
In this subsection we explore the stability of near-contact forms as well as generic curl eigenforms. All of these stability statements work in both the $C^k$, $k\geq 2$, and the $C^\infty$ setting with the same proofs (see, for example, \cite[Appendix A]{greilhuber2023arnolds} for the requisite implicit function theory in Fr\'{e}chet spaces), for which reason we will suppress the regularity in the statement as well as the arguments.

While the set of near-contact forms is not open in $\Omega^1(M)$ with respect to even the smooth topology since $\alpha$ and $d\alpha$ having the same zero set is a closed condition, it is still a fairly stable notion. To make this precise, define the subset $$\mathcal{A}=\{\alpha \in \Omega^1(M)\vert \alpha^{-1}(0)=(d\alpha)^{-1}(0)\}\subset \Omega^1(M).$$ In general, this set has a complicated structure, but for a near-contact form $\alpha$ on a closed $3$-manifold $M$ one can show the following local submanifold property.

\begin{prop}\label{prop: submanifold property of TNC forms}
    Let $M$ be a closed $3$-manifold. Let $\alpha_0$ be a near-contact form with $N$ zeros. Then there exists an open neighborhood $\mathcal{U} \subset \Omega^1(M)$ of $\alpha_0$ such that $\mathcal{U}\cap \mathcal{A}$ is a submanifold of $\Omega^1(M)$ of codimension $3N$, and so that all $\alpha \in \mathcal{U}\cap \mathcal{A}$ are near-contact forms with $N$ zeros.
\end{prop}
We note that this is exactly the expected codimension count since requiring the zeros of $\alpha$ and $d\alpha$ to coincide should be precisely a codimension $3$ condition. As this needs to be achieved for all $N$ points, we arrive at the codimension $3N$. The following proof will make this reasoning precise.
\begin{proof}
Denote by $p_1,\dots,p_N$ the zeros of $\alpha_0$. Since $\alpha_0$ vanishes transversely, it is clear that nearby $1$-forms will also vanish transversely. 
Moreover, we can choose neighborhoods $U_i$ of the $p_i$ and a neighborhood $\mathcal{U}$ of $\alpha_0$ such that all $U_i$ will contain precisely one zero of $\alpha$ and $d\alpha$ for all $\alpha \in \mathcal{U}$. By choosing a trivialization of $T^\ast U_i \cong U_i \times \mathbb R^3$ we construct the map
\begin{align*}
    \Phi:\mathcal{U}\times U_i &\to \mathbb R^3\\
    (\alpha,x)&\mapsto \alpha(x)
\end{align*}
Clearly $\Phi^{-1}(0)$ gives all pairs of $(\alpha,p_i(\alpha))$ of $1$-forms $\alpha$ in $\mathcal{U}$ and their unique zero $p_i(\alpha)$ in $U_i$. Differentiating $\Phi$ at such a zero in the $U_i$-factor gives precisely $\nabla \alpha$ at the zero, which by assumption is surjective. As such, the prerequisites of the implicit function theorem are satisfied, and we obtain that the functions $p_i:\mathcal{U}\to U_i$ assigning $1$-forms in $\mathcal{U}$ their unique vanishing point in $U_i$ are smooth.

We will now choose trivializations of $\Lambda^2 T^\ast U_i\cong \mathbb R^3$ and consider the maps
\begin{align*}
    F_i:\mathcal{U}&\to \mathbb R^3\\
    \alpha & \mapsto d\alpha_{p_i(\alpha)}
\end{align*}
In order to differentiate $F_i$ we first differentiate the equation $\alpha(p_i(\alpha))=0$ in direction $\eta$ to obtain
\begin{align*}
    \eta(p_i)+\left(\nabla\alpha_0\right) (Dp_i)(\eta)=0
\end{align*}
implying that $(Dp_i)(\eta)=-\left(\nabla\alpha_0\right)^{-1}(\eta(p_i))$. Thus we have
\begin{align*}
    (D F_i)(\eta)=d\eta (p_i)+\left(\nabla_{(Dp_i)(\eta)} d\alpha_0\right)\vert_{p_i}
\end{align*}
It is easy to see that for any $\omega_i \in \Lambda^2T_{p_i}^\ast M$ there exists a $1$-form $\eta_i$ supported in $U_i$ such that $\eta_i(p_i)=0$ and $d\eta_i (p_i)=\omega_i$. For such $\eta_i$ it is then clear that $(DF_i)(\eta_i)=\omega_i$ and $(DF_j)(\eta_i)=0$ for $j\neq i$, so $DF$ is surjective. The set $\mathcal{U}\cap \mathcal{A}$ is the zero set of
\begin{align*}
    F:\mathcal U &\to \bigoplus_{i=1}^N \Lambda^2 T^\ast_{p_i}U_i\cong \mathbb R^{3N}\\
    \alpha &\mapsto (F_1(\alpha),\dots,F_N(\alpha))
\end{align*}
Since each $F_i$ is a submersion, so is $F$.

Finally, after shrinking $\mathcal U$ in the $C^2$ topology, every $\alpha\in\mathcal U\cap\mathcal A$ is near-contact: $B_\alpha$ at $q=p_i(\alpha)$ is close to $B_{\alpha_0}$ at $p_i$, hence positive definite, and the expansion of Lemma \ref{lem:Taylor_contact_condition} holds uniformly for $\alpha$ $C^2$-close to $\alpha_0$, so $\alpha\wedge d\alpha>0$ on a punctured ball around $q$ of radius independent of $\alpha$. On the complement of these balls, this follows from $\alpha_0\wedge d\alpha_0>0$.
\end{proof}

While Proposition \ref{prop: submanifold property of TNC forms} establishes the stability of the class of near-contact forms, we can actually use the Gray stability theorem to prove that the contact topology of such forms is stable in an appropriate sense.

\begin{prop}\label{prop: stability of near contact forms}
    Let $M$ be a closed $3$-manifold and let $\alpha_0\in \Omega^1(M)$ be a near-contact form. Then there exists a $C^2$-neighborhood $\mathcal U \subset \Omega^1(M)$ such that for all $\alpha \in \mathcal{U}\cap \mathcal{A}$ there exists a homeomorphism $\Phi_\alpha:M\to M$ and a continuous function $f_\alpha$ such that
\begin{align*}
    \alpha_0 = f_\alpha \Phi_\alpha^\ast \alpha
\end{align*}
where $\Phi_\alpha$ and $f_\alpha$ are smooth in the complement of $\alpha_0^{-1}(0)$. 
\end{prop}
\begin{proof}
This proof is essentially standard, but there is a small subtlety: The usual straight-line interpolation $\alpha_t=(1-t)\alpha_0+t\alpha_1$ between near-contact forms $\alpha_0$ and $\alpha_1$ cannot in general be guaranteed to remain near-contact. For example, the two near-contact forms $\alpha^+_{std}$ and a version of $\alpha^+_{std}$ which is shifted by $\varepsilon$ in the $x$ and $z$ directions have the property that their convex combinations are not all near-contact. However, if we assume that the zeros of $\alpha_0$ and $\alpha_1$ as well as their linearizations $\nabla \alpha_0$ and $\nabla \alpha_1$ agree, then the straight line interpolation $\alpha_t = (1-t)\alpha_0+t\alpha_1$ does stay near-contact in a neighborhood of the zero set because $\alpha_t$ all have the same nondegenerate zeros and at those zeros they satisfy that
\begin{align*}
    B_{\alpha_t}&=\nabla\alpha_t\wedge \nabla d\alpha_t \\ &=\nabla \alpha_0\wedge \left((1-t)\nabla d\alpha_0 + t\nabla d\alpha_1\right)\\
    &=(1-t)B_{\alpha_0} + t B_{\alpha_1}
\end{align*}

Our goal then is to construct a family of diffeomorphisms $\{\Psi_\alpha\}_{\alpha\in \mathcal{U}}$ where $\mathcal{U}\subset \Omega^1(M)$ is a neighborhood of $\alpha_0$, so that $(\Psi_\alpha^\ast \alpha)^{-1}(0)=\alpha_0^{-1}(0)$ and so that $\nabla(\Psi_\alpha^\ast \alpha) = \nabla \alpha_0$. The pullback by $\Psi_\alpha$ of course does not change the topology of $\alpha$ and so we may postcompose this pullback with the map given by the Gray isotopy for near-contact forms in Theorem \ref{thm:singularGray} to obtain the result for any individual form in some neighborhood $\mathcal{U}$ of $\alpha_0$. We upgrade this to the full statement claimed in the proposition by considering the following map.
\begin{align*}
    \Omega^1(M)\times [0,1]&\to \Lambda^3 T^\ast M\\
    (\alpha,t)&\mapsto \alpha_t\wedge d\alpha_t, \quad \alpha_t=(1-t)\alpha_0+t\Psi_\alpha^{\ast}\alpha
\end{align*}
Since this map is continuous with respect to the $C^2$ topology and since the right hand side is positive at $\alpha_0$ one may conclude that there exists a $C^2$ neighborhood of $\alpha_0\times [0,1]$ such that $\alpha_t$ satisfies $\alpha_t\wedge d\alpha_t>0$ in the complement of a neighborhood of the zero set of $\alpha_0$ and near the zero set, this follows from the positivity of $B_{\alpha_t}$ and the uniform expansion of Lemma \ref{lem:Taylor_contact_condition}, as in the proof of Proposition \ref{prop: submanifold property of TNC forms}. Intersecting this open neighborhood with $\mathcal{A}$ and applying Theorem \ref{thm:singularGray} finishes this part of the proposition. 

It remains to prove the existence of a neighborhood $\mathcal{U}$ of $\alpha_0$ and the family of diffeomorphisms $\{\Psi_\alpha\}_{\alpha\in \mathcal{U}}$ with the properties claimed above. We will focus on the neighborhood of a single point $p\in \alpha_0^{-1}(0)$ since the construction works locally around every point in the zero set of $\alpha_0$. We work in a local chart containing $p$ but no other zero of $\alpha_0$ and identify $p$ with $0\in \mathbb R^3$. Consider an open set $U\subset \R^3$ containing the origin. By the first step in the proof of Proposition \ref{prop: submanifold property of TNC forms} we know that there exists an open subset $\mathcal{U}\subset \Omega^1(M)$ containing $\alpha_0$ and a smooth function $p:\mathcal{U}\to U$ which assigns to $\alpha\in \mathcal{U}$ its unique zero in $U$. Thus the diffeomorphism $\Sigma_\alpha$ which shifts $\alpha$ by $-p(\alpha)$ is smooth in $\alpha$ and achieves that 
$$\Sigma_\alpha^\ast \alpha(0)=0.$$ 
Moreover, both $\nabla \alpha_0$ and $\nabla\Sigma_\alpha^\ast\alpha=\Sigma_\alpha^\ast\nabla\alpha$ are symmetric bilinear forms with the same signature, so by Sylvester's Theorem there exists $L_\alpha\in GL^+(\R^3)$ such that $$\nabla\alpha_0=L_\alpha^\ast \Sigma_\alpha^\ast\nabla\alpha=\nabla(\Sigma_\alpha \circ L_\alpha)^\ast \alpha.$$
Naturally, $L_\alpha$ is both smooth in $\alpha$ and does not move the zero of $\alpha$ away from the origin in $\R^3$. Upon choosing a smooth cutoff function $\eta$ which equals $1$ on a smaller neighborhood $V$ of $0$ but compactly supported in $U$, we define $$\Psi_\alpha^p(x)=x+\eta(x)\big((\Sigma_\alpha\circ L_\alpha)(x)-x\big),$$
which, for $\alpha$ close to $\alpha_0$, is a diffeomorphism equal to $\Sigma_\alpha\circ L_\alpha$ on $V$ and to the identity outside $U$, since $\Sigma_\alpha\circ L_\alpha$ is $C^1$-close to the identity. The global family of diffeomorphisms is then obtained by repeating this construction around every $p_i \in \alpha_0^{-1}(0)$ and gluing these local constructions together by extending with the identity. The open set $\mathcal U$ of course must be adjusted for both of these steps, but smoothness (which implies continuity) of all operations guarantees the existence of an appropriate open set.
\end{proof}

Since any curl eigenform is naturally an element of $\mathcal{A}$, we can thus prove the topological stability of transversely vanishing curl eigenforms associated to simple eigenvalues. Before we explain the construction, we recall that, for a generic Riemannian metric $g$ on a closed oriented $3$-manifold $M$, all associated curl eigenforms vanish transversely, and furthermore all nonzero eigenvalues are simple (\cite[Theorem 1.1]{Enciso2012}), so everything that follows describes the expected behavior of curl eigenforms. 

Let now $(M,g)$ be a closed oriented Riemannian $3$-manifold, and let $\alpha\in \Omega^1(M)$ satisfy $\star_gd\alpha=\lambda\alpha$ with $\lambda$ nonzero having multiplicity one. Since the eigenvalues of the curl operator move continuously under perturbation of the underlying metric, there exists a neighborhood $V$ of $g$ in the space of metrics, a real number $\varepsilon>0$, and continuous assignments 
\begin{align*}
    \lambda&:V\to (\lambda-\varepsilon,\lambda+\varepsilon)\subset \mathbb R\\
    \alpha&:V\to \Omega^1(M)
\end{align*}
so that $\lambda(g^\prime)$ is the unique eigenvalue of $\star_{g^\prime}d$ in $(\lambda-\varepsilon,\lambda+\varepsilon)$, and $\alpha(g^\prime)$ is the unique eigenform (up to scaling) associated to $\lambda(g^\prime)$ for all $g^\prime\in V$. These are well-known facts for which we refer to \cite{Kato} for a general theory, and to \cite[Section 2]{greilhuber2023arnolds} for the concrete construction in the case of the Laplace operator, which is perfectly analogous to that of the curl operator. We are now ready to state and prove the topological stability theorem announced in the Introduction.

\begin{theorem}\label{thm: topological stability of curl eigenforms}
    Let $M$ be a closed oriented $3$-manifold. There exists a Baire-generic subset $\mathcal{G}_{\text{gen}}$ of the space of Riemannian metrics $\mathcal{G}$ such that the following holds for every $g\in \mathcal{G}_{\text{gen}}$. 
\begin{enumerate}
    \item The nonzero spectrum of $\star_g d$ is simple.
    \item For every eigenform $\alpha$ of $\star_g d$ {with nonzero eigenvalue} there exists an open neighborhood $V_\alpha\subset \mathcal{G}$ of $g$ such that for all $\tilde g \in V_\alpha$ there exists a homeomorphism $\Phi_{\tilde g}:M\to M$ and a continuous function $f_{\tilde g}$ so that
\begin{align*}
    \alpha=f_{\tilde g}\Phi^\ast_{\tilde g} \alpha(\tilde g)
\end{align*}
where $\Phi_{\tilde g}$ and $f_{\tilde g}$ are smooth in the complement of $\alpha^{-1}(0)$.
\end{enumerate}    
\end{theorem}
\begin{proof}
On closed oriented $3$-manifolds there exists a Baire-generic set $\mathcal{G}_{\text{gen}}\subset \mathcal{G}$ such that for all $g\in \mathcal{G}_{\text{gen}}$ all nonzero eigenvalues of $\star_g d$ are simple and all eigenforms vanish transversely \cite[Theorem 1.1]{Enciso2012}. Fixing one such metric and an associated eigenform $\alpha$, the discussion preceding the statement of this theorem implies that we can uniquely and continuously assign a curl eigenform $\alpha(g^\prime)$ corresponding to $\alpha$ to any nearby metric $g^\prime$. As transversely intersecting the zero section is an open property, we conclude that there exists a neighborhood $V_\alpha$ of $g$ so that $\alpha(g^\prime)\in \mathcal{A}$ for all $g^\prime \in V_\alpha$. The theorem then follows by an application of Proposition \ref{prop: stability of near contact forms}.
\end{proof}

\begin{remark}
    The regularity of the map $\Phi_{\tilde g}$ in Theorem \ref{thm: topological stability of curl eigenforms} cannot be improved to $C^1$ at the zero set of $\alpha$, see Corollary \ref{cor: sharpness of local normal forms}.
\end{remark}
\begin{remark}\label{rmk: regularity of metric transversality statement}
While Theorem $1.1$ in reference \cite{Enciso2012} was only proven for finite regularity metrics, it follows for smooth metrics as well by straightforward applications of standard techniques. For the generic simplicity of the spectrum, one may combine the computations in \cite[Theorem 3.4]{kepplinger2022spectral} with the Fr\'echet-theoretic functional analytic setup of \cite[Section 2]{greilhuber2023arnolds}. Upgrading the generic transversality of curl eigenforms from finite regularity to smooth regularity is a straightforward application of elliptic regularity of the curl operator in the complement of its kernel.
\end{remark}

\section{Contact topology of near-contact singularities}\label{sec: contact topology of transverse singularities}

We will now use Theorem \ref{thm:singularGray} to obtain a normal form for the contact structure defined on a punctured neighborhood of a zero of a near-contact form. This will allow us to study the contact topology near the zero. 

\subsection{The local normal form for near-contact singularities}
In this section, we prove a standard neighborhood theorem for near-contact singularities. First, we know that the linearization of $\alpha$ at a point $p$ has either $1$ or $-1$ signature as in Definition \ref{def:signature} by Lemma \ref{lem:signature}. The following preliminary result establishes that the signature is the only local invariant up to homotopy of near-contact forms. This fact was previously observed in \cite[Lemma 8]{ShevSmir}, but we provide an alternate proof which, due to the language we choose to work with in this paper, becomes much shorter and simpler. 

\begin{prop}\label{prop:homotopy_signature}
    Let $\alpha_0$ and $\alpha_1$ be two $1$-forms, each with a positive near-contact singularity at $p\in M$. If the signature of the linearization of $\alpha_0$ and $\alpha_1$ at $p$ coincide, then there is a neighborhood of $p$ and a homotopy of forms $\alpha_t$ connecting $\alpha_0$ to $\alpha_1$ such that $\alpha_t$ has a positive near-contact singularity at $p$ for all $t$.
\end{prop}
\begin{proof}
As $\nabla \alpha_0$ and $\nabla \alpha_1$ are symmetric bilinear forms with the same signature, there exists $L\in GL(\mathbb R^3)$ such that $\nabla\alpha_0=L^\ast\nabla\alpha_1$ by Sylvester's law of inertia, and replacing $L$ by $-L$ if necessary, we may assume $L\in GL^+(\mathbb R^3)$. Pulling back $\alpha_1$ along a path in $GL^+(\mathbb R^3)$ connecting the identity to $L$ gives a homotopy of near-contact forms which makes the linearizations of $\alpha_0$ and $\alpha_1$ agree. We may thus assume that $\nabla\alpha_0 = \nabla\alpha_1$. As shown in the proof of Proposition \ref{prop: stability of near contact forms}, near-contact forms with the same zero set and linearization at their zero sets have the property that convex combinations are again near-contact forms with the same zero set. By compactness of $[0,1]$, all forms in the resulting path are near-contact on a common neighborhood of $p$.
\end{proof}

\begin{remark}
    The proof of Proposition \ref{prop:homotopy_signature} works just as well for $\alpha_0$ and $\alpha_1$, which have $B_{\alpha_0}$ and $B_{\alpha_1}$ that are not only positive definite but also symmetric. That is to say, intrinsic curl eigenforms can be joined, near a common zero, via homotopies through intrinsic curl eigenforms provided the signatures of that zero agree.
\end{remark}

We can now combine Proposition \ref{prop:homotopy_signature} with Theorem \ref{thm:singularGray} to deduce our normal form result, Theorem \ref{thm: local model for Beltrami}. Throughout the rest of the paper, we will prove statements for the case of signature $1$, and analogous statements hold for $-1$. For this reason, we will use the notation $\alpha_{std}$ for $\alpha_{std}^+$.

\begin{proof}[Proof of Theorem \ref{thm: local model for Beltrami}]
Fix a chart $\kappa:(W,0)\to(\kappa(W),p)$ and write $\alpha$ also for $\kappa^*\alpha$. By Proposition \ref{prop:homotopy_signature}, after shrinking $W$ there is a smooth path $\alpha_s$, $s\in[0,1]$, of near-contact forms on $W$, all vanishing only at $0$, with $\alpha_0=\alpha_{std}$ and $\alpha_1=\alpha$. Let $Z_s$ be the modified Gray vector field of Section \ref{sec:singularGray} for this path, and $\psi_s$ its flow. By Proposition \ref{prop:Zs_flow_defined} (see also Remark \ref{rmk:singularGray_openmanifold}), there is a neighborhood $U$ of $0$ such that $\psi_1$ is defined on $U$, is a homeomorphism onto its image fixing $0$, and is smooth on $U\setminus\{0\}$. {By Lemma \ref{lem:Gray_Hamiltonian}, $\psi_1^*\alpha=h_1\alpha_{std}$ on $U\setminus\{0\}$, where $h_1>0$ is given by Equation \eqref{eq:conformal_factor}. By Lemma \ref{lem:conformal_factor}, $h_1$ extends continuously to $U$ with $h_1(0)=1$ and $h_1-1=\mathcal{O}(\norm{x})$. The statement follows with $\phi=\kappa\circ\psi_1$, $V=\phi(U)$ and $f=h_1$.}
\end{proof}
\begin{remark}\label{remark: first vs second order in alphastd}
It is interesting to note that the first-order part of $\alpha_{std}^\pm$ is closed (and even harmonic), so that the contact volume $\alpha^\pm_{std}\wedge d\alpha^\pm_{std}$ comes from the interaction of the first- and second-order parts. More generally, at a zero of any near-contact form the linear part is exact, so a near-contact form must have nonzero quadratic terms, and Theorem \ref{thm: local model for Beltrami} can be understood as saying that one can always remove the higher-order parts, up to pullback along a homeomorphism smooth away from the zero and a conformal factor.
\end{remark}

\subsection{Optimal regularity of the normal form via a local invariant}\label{subsec: optimal regularity}

We now turn to investigate the regularity of the map which identifies a neighborhood of the singularity with the local model in Theorem \ref{thm: local model for Beltrami}. The most obvious candidate for an invariant which obstructs smooth conjugation to the local model is the set of eigenvalues of the linearization of $\alpha$. However, since linearizations of $1$-forms transform by congruence as opposed to conjugation, we have to find a natural way of constructing a vectorial version of $\nabla\alpha$. For this, we will make use of its nondegeneracy. We first compute the transformation behavior of $B_\alpha$ under pullback. Let $\varphi:M\to N$ be a $C^1$ diffeomorphism and let $\alpha$ be a $1$-form with a positive near-contact type singularity at $p=\varphi(q)$. Set $\beta=f\,\varphi^\ast \alpha$ where $f$ is a positive continuous function. The invariant of a singularity of $\alpha$ will be constructed using the symmetrization $S_\beta$ of the bilinear form $B_\beta$. On the face of it, this requires the existence of the exterior derivative and at least one more derivative, and so, in order to make sense of this, there are several regularity issues that we need to address. We first prove that one can define $d\beta$ at zeros of $\beta$, which will allow us to find the contact volume $\beta\wedge d\beta$. We then use the second-order approximation formula $B_\beta(v,v)=\lim_{t\to 0}\frac{1}{t^2}\beta\wedge d\beta(tv)$ from Lemma \ref{lem:Taylor_contact_condition} to recover, or rather define, $S_\beta$ at $q$. The tensor $S_\beta$ thus defined then satisfies a good transformation law relative to $S_\alpha$ under pullback by $\varphi$.

In general, the pullback of a smooth $1$-form along a $C^1$ diffeomorphism can only be expected to be $C^0$, and of course $f$ is not assumed to be differentiable at all. In both cases, it is precisely the fact that we only care about these quantities at a vanishing point of $\alpha$ which makes the construction work. We explain how to deal with the apparent lack of differentiability of $\varphi^\ast \alpha$ in the following lemma.

\begin{lemma}\label{lem: differentiability of the pullback}
Let $M$ and $N$ be smooth manifolds, $\varphi:M\to N$ be a map in $C^1(M,N)$, and let $\alpha\in \Omega^1(N)$ be a smooth $1$-form on $N$. Then $\varphi^\ast \alpha$ is a continuous $1$-form on $M$ which is differentiable at $\varphi^{-1}(\alpha^{-1}(0))$. Moreover, its exterior derivative satisfies
\begin{align*}
    d (\varphi^\ast \alpha)_q=\varphi^\ast \left(d\alpha\right)_p,
\end{align*}
where $p\in \alpha^{-1}(0)$ and $q\in \varphi^{-1}(p)$.
\end{lemma}
\begin{proof}
We prove this statement in local coordinates centered at $p$ and $q$ in which we assume that both points correspond to $0$. To make the following computations easier to read, we introduce the notation $L=(D\varphi)_0$ and $A=(D\alpha)_0$. We use the assumptions we have on $\varphi$ and $\alpha$ to expand around zero, obtaining
\begin{align*}
    \varphi(x)&=L x + r_\varphi (x)\quad \text{where}\quad r_\varphi(x)=o(|x|),\\
    \left(D\varphi\right)(x) &= L + r_{D\varphi}(x)\quad \text{where}\quad r_{D\varphi}(x)=o(1),\\
    \alpha(y)&=Ay+r_\alpha (y)\quad \text{where}\quad r_\alpha(y)=o(|y|).
\end{align*}
We substitute these expressions into the formula for the pullback $(\varphi^\ast \alpha)(x)=D\varphi^T(x)\alpha(\varphi(x))$ to obtain
\begin{align*}
    (\varphi^\ast \alpha)(x)=\big(L+r_{D\varphi}(x)\big)^T\big(A (Lx+r_\varphi(x))+r_\alpha(\varphi(x))\big).
\end{align*}
The candidate for the linearization of $\varphi^\ast \alpha$ at $0$ is $(D(\varphi^\ast \alpha))_0=L^T A Lx$, so we subtract this term from the previous one and compute
\begin{align*}
  (\varphi^\ast \alpha)(x) - L^T ALx &=  L^T Ar_\varphi(x)+L^Tr_\alpha(\varphi(x))+r_{D\varphi}^T(x)ALx\\
  &+r_{D\varphi}^T(x)Ar_\varphi (x)+r_{D\varphi}^T(x)r_\alpha(\varphi(x)).
\end{align*}
Since $\varphi(x)=Lx+r_\varphi(x)$ is $\mathcal{O}(|x|)$, all the terms on the right-hand side are $o(|x|)$, and thus the derivative exists. This reasoning works for all first derivatives and thus in particular for the exterior derivative.
\end{proof}
Note that the assumption $\alpha(p)=0$ in Lemma \ref{lem: differentiability of the pullback} is what makes the second derivative of $\varphi$ disappear: In the derivative of $(\varphi^\ast\alpha)(x)=D\varphi^T(x)\,\alpha(\varphi(x))$ at $q$, the only term involving $D^2\varphi$ is multiplied by $\alpha(\varphi(q))=\alpha(p)=0$.

A similar but better known and completely elementary trick which will allow us to compute $d\beta=d(f\varphi^\ast \alpha)$ at a zero is the following: Let $f,h:\mathbb R\to \mathbb R$ satisfy that $f$ is continuous, $h$ is $C^1$, and $h(0)=0$. Then the product $fh$ is differentiable at $0$ and its derivative is given by $(fh)^\prime(0)=f(0)h^\prime(0)$. With this in mind and assuming henceforth that $\beta$ is sufficiently differentiable off its zero set we can finally compute 
\begin{align*}
    d\beta_q=d(f\varphi^\ast\alpha)_q=f(q)\varphi^\ast d\alpha_p,
\end{align*}
and, for $x\neq q$ near $q$,
\begin{align*}
    (\beta\wedge d\beta)_x=f(x)^2\,\big(\varphi^\ast (\alpha\wedge d\alpha)\big)_x .
\end{align*}
In spite of the lack of regularity, we can use this observation to define $S_\beta$ as the leading-order behavior of $\beta\wedge d\beta$ at $q$ and obtain the transformation law
\begin{align*}
    S_\beta (q)(v,v)=\lim_{t\to 0}\frac{1}{t^2}(\beta\wedge d\beta)_{tv}=f(q)^2\left(\varphi^\ast S_\alpha\right)(q)(v,v).
\end{align*}

Hence $S_\beta$ is positive definite if and only if $S_\alpha$ is. However, we can construct a finer-grained local invariant of near-contact singularities which is also invariant under contactomorphism. 

For this, choose volume forms $\mu_N \in \Lambda^3 T_p^\ast N$ and $\mu_M \in \Lambda^3 T_q^\ast M$. Since $p$ is a positive near-contact singularity, $S_\alpha$ is positive definite, and we define $h_\alpha$ as the unique inner product on $T_pN$ such that $h_\alpha\, \mu_N = S_\alpha$. Similarly, $h_\beta$ is the inner product on $T_qM$ such that $h_\beta\, \mu_M = S_\beta$. With this, we may introduce natural vectorial versions of the linearizations of $\alpha$ and $\beta$: using $h_\alpha$ to identify $T_pN$ with $T_p^\ast N$, we let $T_\alpha$ be the unique endomorphism of $T_pN$ such that
\begin{align*}
    h_\alpha (T_\alpha u,v)=\left(\nabla_u \alpha\right)(v) \quad \text{for all } u,v\in T_pN,
\end{align*}
that is, $T_\alpha=h_\alpha^{-1}\circ(\nabla\alpha)_p$, and we define $T_\beta$ analogously by $h_\beta (T_\beta u,v)=\left(\nabla_u \beta\right)(v)$. Here, the right-hand side makes sense because of Lemma \ref{lem: differentiability of the pullback} and the discussion following it. Since $d\alpha(p)=0$, the bilinear form $h_\alpha(T_\alpha\,\cdot,\cdot)=(\nabla\alpha)_p$ is symmetric, that is, $T_\alpha$ is $h_\alpha$-self-adjoint. In particular, $T_\alpha$ is diagonalizable with real nonzero eigenvalues, and the same holds for $T_\beta$.

We claim that the unordered projective triple $[\lambda_1:\lambda_2:\lambda_3]$ of eigenvalues of $T_\alpha$ is an invariant of $\alpha$ up to contactomorphism. We call this quantity the \emph{unordered projective spectrum} of $T_\alpha$ and we have the following.

\begin{prop}\label{prop: local invariant for near-contact forms}
    Let $\alpha\in \Omega^1(N)$ be a $1$-form with a positive near-contact singularity at $p$, and let $\varphi:M\to N$ be a $C^1$-diffeomorphism which is smooth in the complement of $q=\varphi^{-1}(p)$. Consider a $1$-form $\beta =f\,\varphi^\ast \alpha$ for some nowhere-vanishing $C^0$-function $f$. Given the unordered spectrum $\sigma(T_\alpha)=(\lambda_1,\lambda_2,\lambda_3)$ of $T_\alpha$, there exists $s\in \mathbb R\setminus \{0\}$ such that
    \begin{align*}
        \sigma(T_\beta)=s\,\sigma(T_\alpha).
    \end{align*}
\end{prop}
\begin{proof}
We continue using the notation established in the discussion preceding this proposition. Denote by $J$ the Jacobian defined by $\varphi^*\mu_N=J\mu_M$. The pullback of $h_\alpha$ under $\varphi$ computes as
\begin{align*}
    h_\beta(q) (u,v)=f^2(q)J (\varphi^\ast h_\alpha)(p)(u,v) . 
\end{align*}
Combining this with the transformation $\nabla\beta = f(q) \varphi^\ast \nabla\alpha$ and
\begin{align*}
    h_\beta (T_\beta u,v)&=\nabla \beta (u,v) \\
    &=f(q)(\varphi^\ast \nabla \alpha)(u,v)\\
    &=f(q)\, \nabla \alpha (\varphi_\ast u,\varphi_\ast v)\\
    &= f(q)\, h_\alpha (T_\alpha \varphi_\ast u,\varphi_\ast v),
\end{align*}
gives 
\begin{align*}
    T_\beta = \frac{1}{f(q) J}\varphi_\ast ^{-1} T_\alpha \varphi_\ast .
\end{align*}
Clearly, the choice of volume forms only changes these computations by a constant factor, and so the projective triple of eigenvalues is left invariant.
\end{proof}

As a consequence, we learn that the regularity of all quantities in Theorem \ref{thm: local model for Beltrami} is essentially sharp. 

\begin{corollary}\label{cor: sharpness of local normal forms}
    Theorems \ref{thm: local model for Beltrami} and \ref{thm: topological stability of curl eigenforms} do not hold if the homeomorphisms are required to be of regularity at least $C^1$ at the zero set of $\alpha$. 
\end{corollary}
\begin{proof}
For the claim regarding Theorem \ref{thm: local model for Beltrami} we merely need to exhibit examples of near-contact singularities which do not have the same unordered projective spectrum. We will construct an analytic $1$-parameter family of such near-contact forms which furthermore have symmetric $B_\alpha$. Constructing explicit metrics for which they are curl eigenforms for eigenvalue $1$ then completes the claim of optimal regularity for Theorem \ref{thm: topological stability of curl eigenforms}. We start with the first step. For this, consider the following family of $1$-forms
\begin{align*}
    \alpha_c=\frac{x}{c}dx+ydy-\frac{z}{c+1}dz +yzdx-cxzdy,
\end{align*}
where $c\in (0,\infty)$. We compute that $B_{\alpha_c}$ is symmetric and that
\begin{align*}
    S_{\alpha_c}(v,v)=(v_x^2+v_y^2+v_z^2)\,dx\wedge dy\wedge dz >0,
\end{align*}
meaning that all $\alpha_c$ are positive near-contact forms. Choosing $\mu_0=dx\wedge dy\wedge dz$ it is immediate that $S_{\alpha_c}=h_{\alpha_c}\,\mu_0$ with $h_{\alpha_c}=dx^2+dy^2+dz^2$. We can then recover
\begin{align*}
T_{\alpha_c}=h_{\alpha_c}^{-1}\left(\nabla \alpha_c\right),
\end{align*}
which gives $\sigma(T_{\alpha_c})=(\frac{1}{c},1,-\frac{1}{c+1})$. While not all $\alpha_c$ have distinct $\sigma(T_{\alpha_c})$ (for example, the unordered projective spectra coincide for $\alpha_c$ and $\alpha_{\frac{1}{c}}$), it is clear that $\alpha_1$ does not have an unordered projective spectrum that is equal to any other element in the family. The quickest way to see this is that the spectrum at $c=1$ has a double eigenvalue while the one for $c\neq 1$ does not.

For the construction of compatible metrics, consider first the auxiliary Reeb-like vector fields
\begin{align*}
    X_c=cx\partial_x + y \partial_y -(c+1)z\partial_z
\end{align*}
and symmetric tensors
\begin{align*}
    \bar g_c=\begin{pmatrix}
        \frac{1}{c^2} & 0 & -\frac{y}{c+1}\\
        0&1&\frac{cx}{c+1}\\
        -\frac{y}{c+1}&\frac{cx}{c+1}&\frac{1}{(c+1)^2}
    \end{pmatrix}.
\end{align*}
These quantities then satisfy the relations
\begin{align*}
&d\alpha_c =\iota_{X_c}\mu_0\\
    &\bar g_c\left(X_c,-\right)=\alpha_c\\
    &\det \bar g_c=\frac{1-c^2(x^2+y^2)}{c^2(1+c)^2}=:|\bar g_c|
\end{align*}
In particular, we can find $\varepsilon>0$ and a neighborhood $U$ of the origin such that $\bar g_c$ is a Riemannian metric on $U$ for all $c\in(1-\varepsilon,1+\varepsilon)$. We then use the standard identity $\star_{\bar g_c}\left(\iota_Y \mu_{\bar g_c}\right)=\bar g_c (Y,-)$ to get
\begin{align*}
    \star_{\bar g_c}d\alpha_c&= |\bar g_c|^{-1/2} \alpha_c
\end{align*}
Consequently, by defining $g_c=|\bar g_c|^{-1}\bar g_c$ and using $\star_{f^2g}=f^{-1}\star_g$ we conclude that $\star_{g_c}d\alpha_c=\alpha_c$.
\end{proof}

\begin{remark}\label{rem:preclude}
The regularity of the homeomorphisms in Theorem \ref{thm:singularGray} cannot be improved to $C^1$ in general either. Indeed, let $\alpha_0$ be a positive near-contact form on a closed $3$-manifold $M$ with a zero $p$ near which, in an orientation-preserving chart, $\alpha_0=\alpha^+_{std}$. Let $\chi$ be a cut-off function supported in this chart and equal to $1$ near $p$, and consider
$$\alpha_c=\alpha_0+c\,\chi\, xz\,dy .$$
Where $\chi=1$ we have $\alpha_c\wedge d\alpha_c=\big((1-c)x^2+y^2+(4-2c)z^2\big)\,dx\wedge dy\wedge dz$, and $\alpha_c$ vanishes only at $p$ there. Elsewhere, $\alpha_c$ is a small perturbation of the contact form $\alpha_0$. Hence, for $c\in[0,c_0]$ with $c_0>0$ small, $\alpha_c$ is a smooth path of near-contact forms with the same zero set as $\alpha_0$. At $p$ we have $D\alpha_c|_p=\mathrm{diag}(1,1,-2)$ and $Dd\alpha_c|_p=\mathrm{diag}(1-c,1,c-2)$, so that $S_{\alpha_c}=\mathrm{diag}(1-c,1,4-2c)$ and
$$\sigma(T_{\alpha_c})(p)=\Big(\frac{1}{1-c},\,1,\,-\frac{1}{2-c}\Big).$$
For $c=0$ the positive eigenvalue is double, while for $c\in(0,c_0]$ it is not, so the unordered projective spectra of $\alpha_0$ and $\alpha_c$ at $p$ differ. If the homeomorphism $\varphi_c$ given by Theorem \ref{thm:singularGray}, which satisfies $\varphi_c^*\alpha_0=h_c\alpha_c$ with $h_c$ continuous and $h_c(p)=1$, were a $C^1$ diffeomorphism near $p$, this would contradict Proposition \ref{prop: local invariant for near-contact forms}.
\end{remark}

\begin{remark}\label{q:invariants}
Is there an invariant of more contact geometric nature than the projective spectrum of the singularity? By the Grobman--Hartman theorem (\cite{Hartman1960}), there cannot be an invariant coming from the $C^0$-dynamics of the Reeb-like flows associated to the near-contact form; the invariant would have to use more subtle geometric quantities. Relatedly, we do not know whether the contactomorphism in the statement of Theorem \ref{thm: local model for Beltrami} can be upgraded to a strict contactomorphism. 
\end{remark}

\subsection{The local contact topology of near-contact structures}
Having established a local normal form for the contact structure in a punctured neighborhood of the singularity, we can now analyze its topological features.

\subsubsection{Overtwistedness and convex surfaces}
An important tool in the analysis of contact structures on $3$-manifolds is the notion of a convex surface. To define it, let us first introduce the characteristic foliation of a surface
\begin{defi}
    Let $S$ be an oriented surface in a contact manifold $(M,\xi)$. The characteristic foliation is the singular foliation defined by a vector field $X$ uniquely defined by $\iota_X\Omega=\alpha_S$, where $\Omega$ is some area form in $S$, compatible with its orientation, and $\alpha_S$ is the restriction of some defining contact form $\alpha$ to $S$.
\end{defi}
One can see that the foliation, which has one-dimensional and zero-dimensional leaves, only depends on $S$ and $\xi$. The points where the leaf is zero-dimensional correspond to those points where $TS$ coincides with $\xi$, while at any other point $p$, the tangent space to the one-dimensional leaf is given by $T_pS\cap \xi_p$. With this notion in hand we can define one of the central objects of $3$-dimensional contact topology.
\begin{defi}
    Let $(M,\xi)$ be a contact manifold. A disk $D\subset (M,\xi)$ is called overtwisted if its characteristic foliation is conjugate to that of $$D_0=\{z=0,r\leq \pi\}\subset (\mathbb R^3, \alpha_{OT}=\cos(r)dz+r\sin(r) d\theta ).$$ A contact structure containing an overtwisted disk is called an overtwisted contact structure. Otherwise, it is called a tight contact structure.
\end{defi}
This notion, while superficially innocuous, has been the principal tool in the rough classification of contact structures \cite{El}. 

We will now recall a technique which allows for efficient cut-and-paste processes in $3$-dimensional contact manifolds.

\begin{defi}
    A surface $S$ is convex if there exists a vector field $Y$ in $M$ that is transverse to $S$ and which is a contact vector field (that is, its flow $\psi_t$ satisfies $(\psi_t)^*(\xi)=\xi$ for all $t$). The set $\Gamma=\{p\in S \mid  Y_p\in \xi_p\}$, determined up to isotopy, is called the dividing set.
\end{defi}
While $\Gamma$ depends on the choice of contact vector field $Y$, it is an elementary fact that its isotopy class does not. Moreover, the contact structure in a neighborhood of a convex surface is entirely determined by its dividing set. Remarkably, Giroux found a simple criterion on dividing sets which fully answers the question of when a convex surface has a neighborhood containing an overtwisted disk.

\begin{theorem} {Giroux \cite[Th\'{e}or\`{e}me 4.5.]{Giroux2001}}\label{theorem: giroux}
    Let $S\subset (M,\xi)$ be a closed convex surface. A neighborhood of $S$ is overtwisted if and only if
\begin{enumerate}
    \item $S$ is homeomorphic to $\mathbb S^2$ and $\Gamma$ has two or more components, or
    \item $S$ is not homeomorphic to a sphere and $\Gamma$ has a nullhomotopic component.
\end{enumerate}
\end{theorem}

\subsubsection{The local contact topology of a near-contact singularity}

The following lemma tells us that transverse zeros of near-contact forms have neighborhoods which are foliated by convex spheres.

\begin{lemma}\label{lem:convexspheres}
    The coordinate spheres $S_r (0)$ centered at zero in $(\mathbb R^3,\alpha_{std})$ are convex. Their dividing set is $r$-invariant and has $2$ connected components. In particular, every punctured neighborhood of zero is overtwisted.
\end{lemma}
\begin{proof}
We expand $\alpha_{std}$ in spherical coordinates 
\begin{align*}
    \alpha_{std}=r(\sin^2\theta-2\cos^2\theta)\,dr+3r^2\sin\theta\cos\theta\,d\theta-r^3\sin^2\theta \cos\theta\,d\varphi
\end{align*}
Thus pulling back to the coordinate spheres $S_r(0)$ we get
\begin{align*}
    \alpha_{S_r}&:=\iota_{S_r(0)}^\ast \alpha_{std}=3r^2 \sin\theta \cos\theta\,d\theta -r^3\sin^2\theta\cos\theta \,d\varphi
\end{align*}

It is easy to see that the vector field 
\begin{align*}
    X=-r\sin\theta\cos\theta\partial_\theta-3\cos\theta\partial_\varphi
\end{align*}
spans the kernel of $\alpha_{S_r}$ for $\theta \in (0,\pi)$. Furthermore, the characteristic foliation induced by $\alpha_{std}$ is singular precisely when $\theta\in \{0,\frac{\pi}{2},\pi\}$, i.e., on the poles and the equator. 

We compute the divergence of $X$ using the induced volume form $$\omega_r=r^2\sin\theta \,d\theta\wedge d\varphi,$$
obtaining
\begin{align*}
    \mathcal{L}_X \omega &= d\iota_X \omega\\
    &= d\alpha_{S_r}\\
    &= r^3 \sin\theta \left(\sin^2\theta-2\cos^2\theta\right)\,d\theta\wedge d\varphi
\end{align*}
As $\sin\theta >0$ on $\theta\in (0,\pi)$, the sign of the divergence changes exactly on the set $\Gamma:=\{\sin^2\theta = 2\cos^2\theta\}$ which is independent of $r$ and corresponds to 
$$\Gamma=\left\{\left(C_1,\varphi\right),\varphi\in [0,2\pi]\right\} \sqcup \left\{\left(C_2,\varphi\right),\varphi\in [0,2\pi]\right\},$$
where $C_1=2\arctan(\sqrt{2+\sqrt{3}})$ and $C_2=2\arctan(\sqrt{2-\sqrt{3}})$. The divergence of $X$ is positive on the region $R_+=\{\sin^2\theta>2\cos^2\theta\}$ between the two circles of $\Gamma$, which contains the equator, and negative on the two polar caps $R_-$. The vector field $X$ is transverse to $\Gamma$ and points out of $R_+$ along $\Gamma$, since $X^\theta=-r\sin\theta\cos\theta$ is negative on the upper circle and positive on the lower one. Hence $\Gamma$ divides the characteristic foliation, and by Giroux's convexity criterion \cite{giroux1991convexity} the sphere $S_r(0)$ is convex with dividing set $\Gamma$. 

Alternatively, the vector field $$Y=\frac{(x^2+y^2+2z^2)(x\partial_x+y\partial_y)+4z^3\partial_z}{x^2+y^2+4z^2},$$
which will come in handy later in Section \ref{sec:contact_foliation}, satisfies $\mathcal{L}_Y\alpha_{std}=2\alpha_{std}$ and $\langle x,Y\rangle>0$, so it is a contact vector field transverse to every $S_r(0)$, and $\{\alpha_{std}(Y)=0\}=\{x^2+y^2=2z^2\}=\Gamma$.

The foliation by spheres is schematically depicted in Figure \ref{fig:convex_spheres}, with the blue color indicating the circle of singularities of the characteristic foliation, and the magenta color the dividing set. Theorem \ref{theorem: giroux} allows us to conclude that every (punctured) neighborhood of $0$ is overtwisted.
\end{proof}

\begin{figure}
    \centering
    \begin{tikzpicture}
    \node at (0,0) {\includegraphics[width=0.4\linewidth]{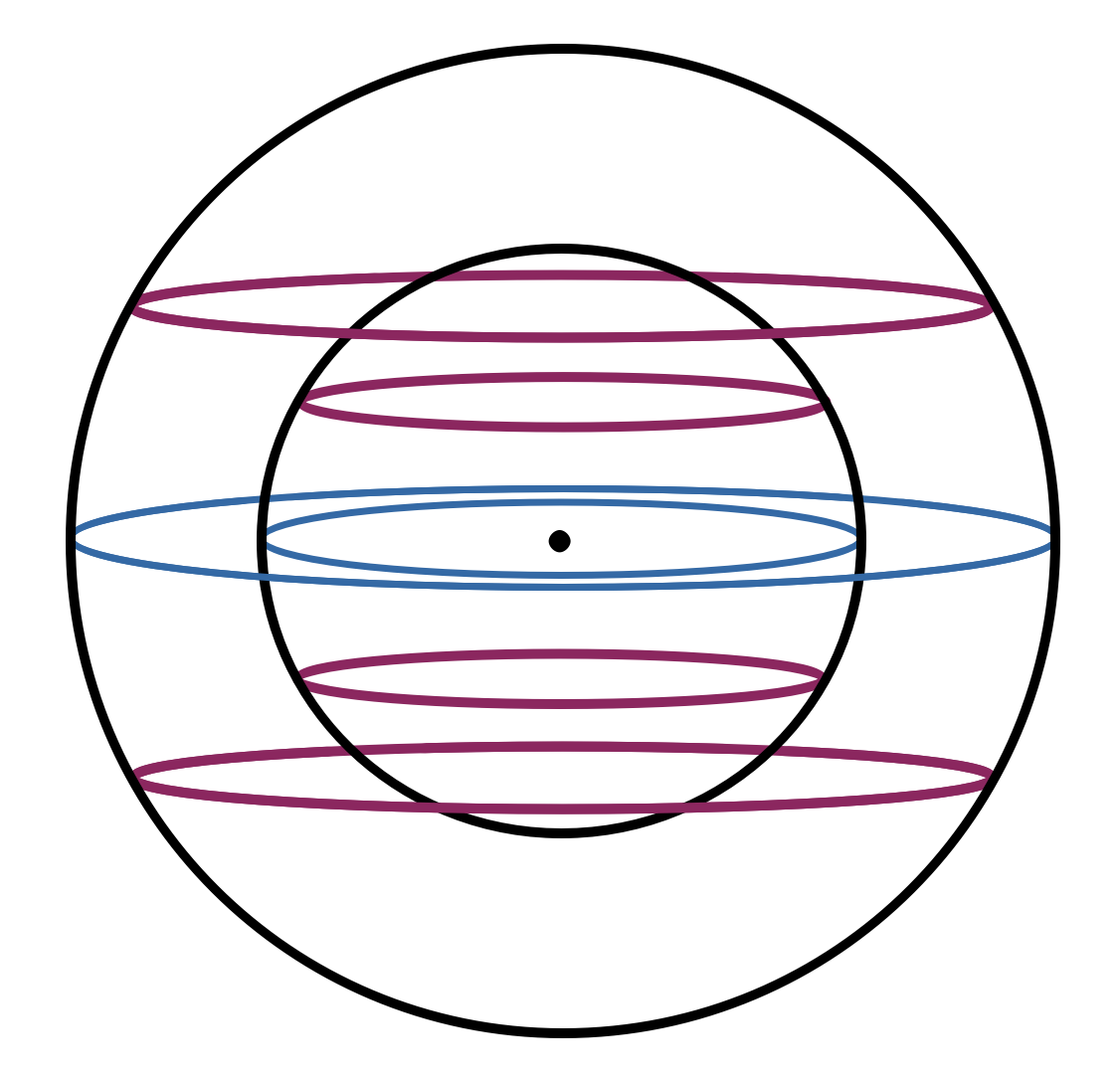}};
    \node[color=mymagenta] at (-2.7,1.5) {$\Gamma$};
    \end{tikzpicture}
    \caption{Foliation by convex spheres}
    \label{fig:convex_spheres}
\end{figure}

\begin{remark}\label{rem:sign_divergence}
    For $\alpha_{std}^+$, the sign of the divergence of $X$ is negative near the poles, and positive between both dividing curves.  For $\alpha_{std}^-$, the sign of the divergence is the opposite.
\end{remark}
As a consequence of Theorem \ref{thm: local model for Beltrami} and Lemma \ref{lem:convexspheres}, we deduce the following.
\begin{corollary}\label{cor:foliation_convex_spheres}
Let $p$ be a near-contact singularity of a $1$-form $\alpha$. Then a punctured neighborhood $U$ of $p$ is foliated by convex spheres whose dividing set has two connected components. In particular, the contact structure $\ker \alpha$ in $U$ is overtwisted.
\end{corollary}
Pollard and Alexander \cite[Theorem 5.6]{PollardAlexander} state that ``singular contact structures" with singularities of nonzero index, and locally defined by an analytic singular contact form, are overtwisted, sketching an argument based on the existence of a convex sphere and Giroux's criterion. The local model of Theorem \ref{thm: local model for Beltrami} allows us to give a complete proof of the local statement for every (smooth) near-contact singularity, finding even a foliation by overtwisted convex spheres.

\subsection{Second application: Global classification of near-contact structures}

The study of the local model has revealed the fact that near-contact structures are overtwisted contact structures in the complement of their singularities. Such structures are well known for their flexibility: They satisfy an $h$-principle that allows for a classification in terms of formal homotopical data \cite{El}. It is then natural to combine our local analysis with Eliashberg's results to give a similar classification result for near-contact structures.

\paragraph{Near-contact forms with the same zeros.}
The natural situation to address is that of two near-contact forms with the same set and type of zeros. That is, let $\alpha,\beta$ be two near-contact forms on a closed manifold $M$, and suppose that they have the same set of zeros $Z\neq \emptyset$ with matching indices. To identify them (assuming the formal homotopical data), the natural strategy is to first identify them near the singularities, keep track of the homotopy class relative to a neighborhood of the singularities, and then use the relative $h$-principle from the contact setting. This strategy is also the one used, for instance, for establishing a similar result for fully overtwisted $b$-contact structures \cite{CO}.

First, we show that we can assume that $\ker \alpha=\ker \beta$ in a punctured neighborhood of $Z$, up to continuous isotopy smooth away from $Z$.
\begin{lemma}\label{lem:equal_near_zeros}
There exists a continuous isotopy $\phi_t:M\longrightarrow M$ fixing $Z$ and smooth in the complement of $Z$, such that $\phi_1^*\beta=f\alpha$, for some positive function $f$, in a punctured neighborhood of $Z$.
\end{lemma}
\begin{proof}
Arguing as in the proof of Theorem \ref{thm: local model for Beltrami}, in a (punctured) neighborhood $U$ of $Z$ there is such an isotopy $\phi_t^{local}$ satisfying $(\phi_1^{local})^*\beta=f\alpha$, generated by a vector field as in the proof of the Gray-stability type Theorem \ref{thm:singularGray}. Cut off the vector field generating the isotopy, so that there is a smaller neighborhood $V\subset U$ such that the isotopy $\phi_t$ generated by it satisfies
\begin{itemize}
    \item[-] there exists a smaller neighborhood $V'\subset V$ such that $\phi_t(V')\subset V$ for all $t\in [0,1]$,
    \item[-] $\phi_t|_{V'}=\phi_t^{local}$ for $t\in[0,1]$,
    \item[-] $\phi_t|_{M\setminus U}=\operatorname{id}$.
\end{itemize}
Then $\phi_t$ is the desired isotopy.
\end{proof}
The following result gives a complete homotopical characterization of when these two structures are isotopic.
\begin{theorem}\label{thm:classification_near_contact_structures}
   Let $\ker \alpha$ and $\ker \beta$ be two positive near-contact structures with the same singular locus $Z\neq \emptyset$. There exists a continuous isotopy $\varphi_t: M\longrightarrow M$ fixing $Z$, smooth in $M\setminus Z$, and such that $\varphi_1^*\beta=f\alpha$ in $M\setminus Z$ with $f>0$, if and only if $\ker \alpha$ and $\ker \beta$ are homotopic (as co-oriented plane fields) in $M\setminus Z$.
\end{theorem}
\begin{proof}
We may assume that $M$ is connected. One implication is direct, namely, if such an isotopy exists, then it defines a homotopy in $M\setminus Z$ between $\ker \alpha$ and $\ker \beta$. Let us then prove the difficult implication. First, because the plane fields are homotopic away from $Z$, the indices at each point in $Z$ must coincide. By Lemma \ref{lem:equal_near_zeros}, we may thus assume that $\ker \alpha=\ker \beta$ in a neighborhood of $Z$. Moreover, up to modifying $\beta$ by an isotopy fixing $Z$, the homotopy between $\ker\alpha$ and $\ker\beta$ can be taken relative to a neighborhood of $Z$. Indeed, for each $p\in Z$ and a small ball $B_p$ on which both structures agree, the homotopy restricted to $B_p\setminus\{p\}\simeq S^2$ defines a loop of plane fields. Since $B_p\setminus \{p\}$ deformation retracts to a sphere around $p$ in $B_p$, this loop of plane fields can be understood up to homotopy as a loop of maps from $S^2$ to $S^2$ that have degree $-1$ (respectively $1$) if $p$ has positive (respectively negative) signature. Thus, it defines an element in $\pi_1(\mathrm{Map}_{\pm 1}(S^2,S^2))\cong \mathbb{Z}/2$ (where the $\pm 1$ specifies the degree of the maps). The loop can be made constant by a homotopy relative to endpoints (which is equivalent to making the homotopy fixed in a punctured neighborhood of $p$) if and only if this element in $\pi_1(\mathrm{Map}_{\pm 1}(S^2,S^2))\cong \mathbb{Z}/2$ is trivial. If this class is non-trivial, the issue is that the relative homotopy classes of $\xi_\alpha$ and $\xi_{\beta}$, relative to a punctured neighborhood of $Z$, do not coincide. However, this can be corrected by a preliminary isotopy near $p$ (which keeps fixed the global homotopy class, but not the relative one). 

First, recall that if $R_{2\pi t}(x,y,z)$ denotes a rotation of $2\pi t$ around the $x$ axis, the loop $t\mapsto h\circ R_{2\pi t}$, $t\in [0,1]$, for any $h\in \mathrm{Map}_{\pm 1}(S^2,S^2)$ is a generator of $\pi_1(\mathrm{Map}_{\pm 1}(S^2,S^2))$. Secondly, observe that $R_\pi(x,y,z)=(x,-y,-z)$ satisfies $R_\pi^*\alpha_{std}^\pm=\alpha_{std}^\pm$. Consider, in the chart of Theorem~\ref{thm: local model for Beltrami} around $p$, the isotopy 
$$\tau_t=R_{\pi t\chi(|x|)},$$ 
where $\chi$ is a cut-off function equal to $1$ near $0$. It is a continuous isotopy fixing $p$ and smooth away from it. Additionally, where $\chi=1$, we have $\tau_1^*\beta= g\beta$ for some positive function $g$, since $\beta$ writes in this chart as $f\alpha_{std}^\pm$ for some positive function $f$, and 
$$R_{\pi}^*(f\alpha_{std}^\pm)=f\circ R_\pi \cdot \alpha_{std}^\pm=\frac{f\circ R_\pi}{f} f\alpha_{std}^\pm.$$ 
We claim that replacing $\beta$ by $\tau_1^*\beta$ changes the class above by the generator. Indeed, the path $t\mapsto \tau_t^*\xi_\beta$, $t\in[0,1]$, is a homotopy of plane fields on $M\setminus Z$ from $\xi_\beta$ to $\tau_1^*\xi_\beta$, so concatenating it with the given homotopy from $\xi_\alpha$ to $\xi_\beta$ yields a homotopy from $\xi_\alpha$ to $\tau_1^*\xi_\beta$, whose class at $p$ is the sum of the previous class and that of the loop $t\mapsto R_{\pi t}^*\xi_{std}$, where $\xi_{std}=\ker \alpha_{std}^\pm$, on a small sphere $S_r$ where $\chi=1$. 

It remains to check that the class of the loop is given by the generator. Up to homotopy through loops, we may replace $\alpha_{std}$ by its linear part $x\mapsto \hat Ax$, with $\hat A=\mathrm{diag}(1,1,-2)$, and then deform $\hat A$ to $A=\mathrm{diag}(1,1,-1)$ through matrices commuting with $R_\pi$. The loop of normal maps on $S_r$ then becomes $$x\mapsto R_{\pi t}^T A R_{\pi t}x/|R_{\pi t}^T A R_{\pi t}x|,$$ and since $R_\theta^T A R_\theta=A R_{2\theta}$, this is the degree $-1$ map $x\mapsto A x/|A x|$ precomposed with the full rotation loop $R_{2\pi t}$, i.e., the generator (for $\alpha^-_{std}=-\alpha^+_{std}$, we use $-A$ instead, which gives the degree $1$ map $x\mapsto -Ax/|Ax|$ precomposed with the same loop). After doing this at every $p$ where the class is nontrivial, the homotopy extension property yields a homotopy relative to a neighborhood $N$ of $Z$, contained in the region where $\chi=1$, on which $\ker\alpha=\ker\tau_1^*\beta$. We keep denoting $\tau_1^*\beta$ by $\beta$.
\medskip

By Corollary \ref{cor:foliation_convex_spheres}, the zeros are surrounded by convex spheres with a non-connected dividing set. Take a sufficiently small closed neighborhood $K$ of $Z$, and another closed neighborhood $V$ strictly containing $K$, such that $\ker \alpha$ and $\ker \beta$ coincide in some open neighborhood of $V$ contained in $N$. One of the convex spheres is contained in $N\setminus V$, both for $\xi_\alpha=\ker \alpha$ and for $\xi_\beta=\ker \beta$ since they coincide in $N$. By Giroux's criterion, there must be an overtwisted disk in any neighborhood of those spheres, and thus there is an overtwisted disk $\Delta$ in $N\setminus V$ for both $\xi_\alpha$ and $\xi_\beta$. In particular, the contact structures $\xi_\alpha$ and $\xi_\beta$ in $M\setminus \operatorname{int}(K)$ (which is a compact manifold with boundary) satisfy the following properties:
\begin{itemize}
\item[-] They coincide in $V\setminus \operatorname{int}(K)$,
\item[-] they share a fixed overtwisted disk $\Delta\subset N\setminus V$,
\item[-] they are homotopic as plane fields in $\widehat M=M\setminus \operatorname{int}(K)$ relative to $V\setminus \operatorname{int}(K)$.
\end{itemize}
We can thus choose a smooth family of plane fields $\xi_t$ with $t\in [0,1]$ such that $\xi_0=\xi_\alpha$, $\xi_1=\xi_\beta$, and $\xi_t|_{N}=\xi_\alpha=\xi_\beta$ for all $t$.
We have that $\widehat M\setminus V$ is connected, the plane fields $\xi_t$ are contact near $V$ for all $t$, and $\Delta$ is an overtwisted disk for all $\xi_t$, since they all coincide with $\xi_\alpha$ on $N$. By the relative $h$-principle for overtwisted contact structures \cite[Theorem 3.1.1]{El} (see also \cite[Theorem 1.2]{BEM}), there is a path $\xi'_t$, $t\in[0,1]$, of contact structures on $\widehat M$ such that $\xi'_0=\xi_\alpha$, $\xi'_1=\xi_\beta$, and $\xi'_t=\xi_\alpha$ near $V$ for all $t$. By Gray's stability applied to $\xi'_t$, which is constant near $V$, there is an isotopy of $\widehat M$ equal to the identity near $\partial\widehat M$ pulling $\xi_\beta$ back to $\xi_\alpha$. Extended by the identity on $K$ and composed with the isotopies of Lemma \ref{lem:equal_near_zeros} and the isotopies $\tau_t$ (constructed near some of the points in $Z$), it gives the desired isotopy.
\end{proof}
\begin{remark}\label{rem:homotopy_near_contact}
That two near-contact structures with the same set and type of zeros, and the same homotopy class of plane fields away from the singular locus, can be connected by a homotopy (which, unlike an isotopy, does not identify their isomorphism class) of near-contact structures is observed in \cite[Proposition 2]{ShevSmir}. We included in our proof a discussion of the small subtlety that after applying Lemma \ref{lem:equal_near_zeros} one might still get plane fields which are not homotopic relative to a neighborhood of $Z$ (because of the obstruction in $\mathbb{Z}/2$), which can be fixed using $R_\pi$.
\end{remark}

\begin{remark}\label{rem:different_zeros}
Theorem \ref{thm:classification_near_contact_structures} also covers near-contact forms $\alpha_1,\alpha_2$ whose sets of zeros $Z_1,Z_2$ differ. A necessary condition is that $Z_1$ and $Z_2$ have the same number of zeros of each index. Since, by the Poincaré--Hopf theorem, there are as many zeros of index $1$ as of index $2$, both sets then consist of $n$ zeros of each index. Given an isotopy $\psi_t$ of $M$ such that $\psi_1$ maps $Z_1$ onto $Z_2$ preserving indices, the forms $\alpha_1$ and $\psi_1^*\alpha_2$ have the same set of zeros with matching indices, and Theorem \ref{thm:classification_near_contact_structures} gives the following. There exists a continuous isotopy $\varphi_t$ of $M$, smooth away from $Z_1$, with $\varphi_1(Z_1)=Z_2$ and $\varphi_1^*\alpha_2=f\alpha_1$ on $M\setminus Z_1$ with $f>0$, if and only if, for some such $\psi_t$, the plane fields $\ker\alpha_1$ and $\ker\psi_1^*\alpha_2$ are homotopic as co-oriented plane fields on $M\setminus Z_1$. A priori, whether this holds depends on the choice of $\psi_t$, for instance on the bijection $Z_1\to Z_2$ it induces.
\end{remark}

\subsection{Third application: A spectral theoretic proof of Martinet's theorem}

We now give an alternative proof of Martinet's theorem. In this proof, we will assume some familiarity with contact topology; we refer to \cite{Geiges} and \cite{EtnyreNotes} for more background.

\begin{theorem}\label{thm:Martinet}
Any closed orientable $3$-manifold $M$ admits a (positive) contact structure.
\end{theorem}
\begin{proof}
The starting point of the proof is to choose a metric $g$ such that the curl operator has a transversely vanishing curl eigenform $\alpha$ with positive eigenvalue. This is always possible by \cite[Theorem 1.1]{Enciso2012}. By Remark \ref{rmk: regularity in proof of martinet theorem} we may furthermore assume $\alpha$ to be smooth. If $\alpha$ has no zeros, it is a contact form, and we are done. Since a curl eigenform can only have nondegenerate zeros of index $1$ and $2$ and since $M$ has vanishing Euler characteristic, we see that there must be $k$ nondegenerate zeros of index $1$ and $k$ nondegenerate zeros of index $2$. These correspond to near-contact singularities of signature $1$ and $-1$ respectively. We will now pair these off and argue that for each pair we can achieve a cancellation on the contact topological level. We denote the zero of index $1$ by $p_+$ and that of index $2$ by $p_-$. According to Theorem \ref{thm: local model for Beltrami} and, more specifically, Corollary \ref{cor:foliation_convex_spheres}, there exist convex spheres around $p_+$ and $p_-$ as depicted in Figure \ref{fig:convex_spheres} (with opposite signs of the divergence of the characteristic field; see Remark \ref{rem:sign_divergence}). 

Fix $S_+$ and $S_-$ to be one of those convex spheres around $p_+$ and $p_-$, respectively. Choose points $x\in S_+$ and $y\in S_-$ such that $x$ lies on the southern dividing circle of $S_+$ and $y$ lies on the northern dividing circle of $S_-$. Let $\gamma$ be a Legendrian arc connecting these two points, and intersecting the spheres only along its endpoints $x,y$. We claim that there exists a cylinder $C$ surrounding $\gamma$ with the following properties:
\begin{itemize}
    \item[-] it is convex,
    \item[-] its dividing set is given by two arcs, each joining one boundary component to the other,
    \item[-] its boundary, which is also its intersection with $S_+$ and $S_-$, is Legendrian and intersects the dividing circles containing $x$ and $y$ twice.
\end{itemize}
The existence of such a cylinder is somewhat ``folklore" in contact topology; we sketch a proof of this general fact about convex surfaces in the appendix for completeness (see Lemma \ref{lem:folklore}). Applying it replaces $S_\pm$ by $C^0$-close convex spheres with the same dividing sets (up to isotopy), and these still bound balls containing $p_\pm$. The situation is schematically represented in Figure \ref{fig:connected_spheres}. 
\begin{figure}[!h]
    \centering
    \begin{tikzpicture}
    \node at (0,0) {\includegraphics[width=0.6\linewidth]{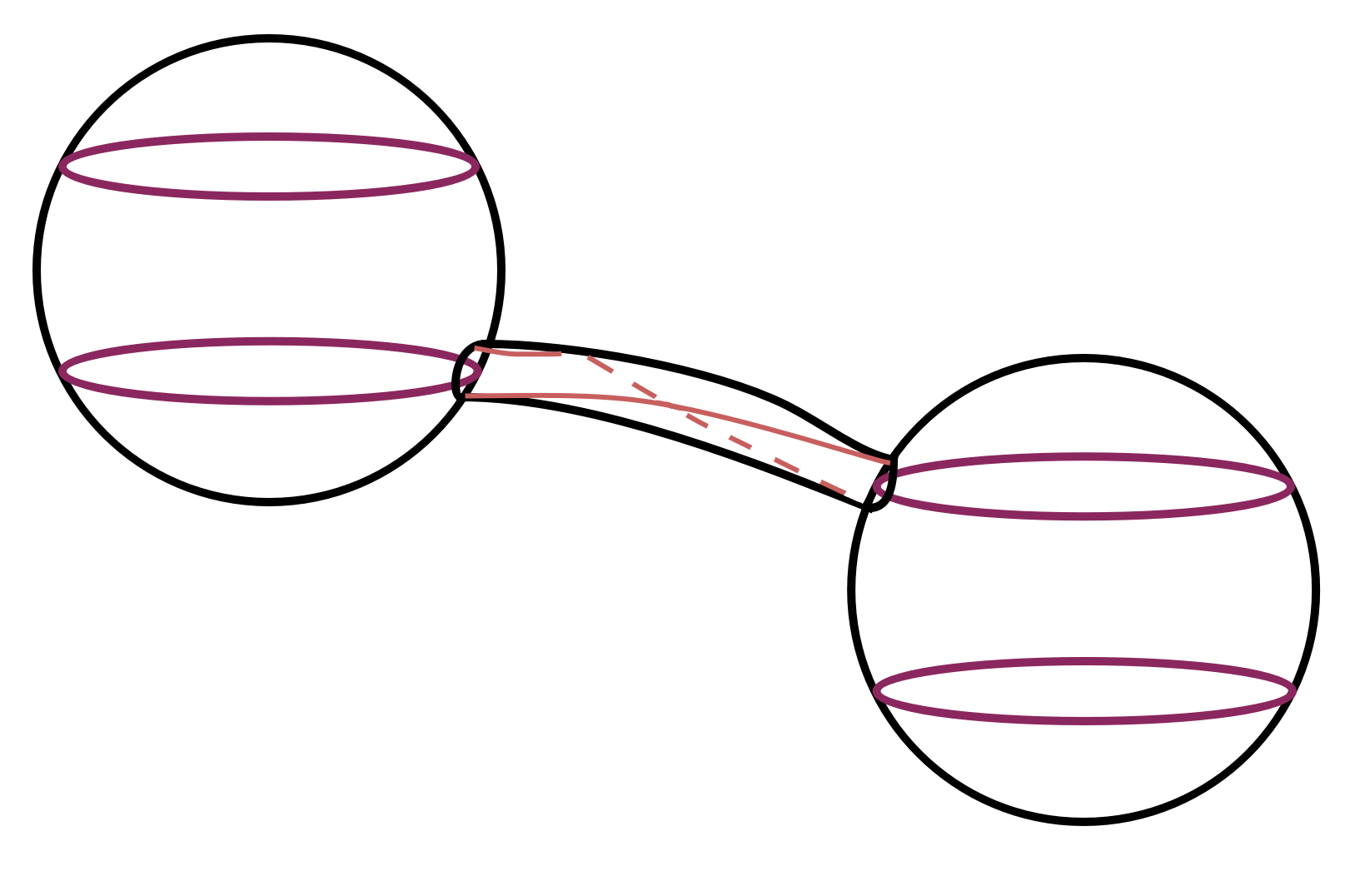}};
    \node at (-4.5,2) {$S_+$};
    \node at (4.5,0) {$S_-$};
    \node at (-2.5,2.15) {$-$};
    \node at (-2.5,1.1) {$+$};
    \node at (-2.5, 0) {$-$};
    \node at (2.55, 0.2) {$+$};
    \node at (2.55, -0.9) {$-$};
    \node at (2.55, -2) {$+$};
    \end{tikzpicture}
    \caption{Convex spheres connected by a convex cylinder}
    \label{fig:connected_spheres}
\end{figure}

Let us argue why the dividing arcs of the cylinder do behave as in the picture. First, along the intersection of $C$ with $S_+$ (or $S_-$), the intersection points with the dividing sets of $C$ and $S_+$ will alternate (see \cite{HondaClassification}). Edge-rounding the surface $(S_+\setminus D_+)\cup C\cup(S_-\setminus D_-)$, where $D_\pm\subset S_\pm$ are the disks bounded by $\partial C$ around $x$ and $y$, yields a convex sphere $S$, and its dividing set is obtained from a right-hand rule (we refer to \cite[Lemma 4.3]{EtnyreNotes} and \cite{HondaClassification} for more details). In particular, the dividing set of $S$ consists either of three parallel circles or of a circle enclosing two other circles, depending on which regions of $S_+$ and $S_-$ the dividing arcs of the cylinder connect. Each component of $S\setminus\Gamma_S$ has a single sign, and on $S_\pm$ the positive and negative regions of $S$ agree with those of $S_\pm$, since the orientation of $S$ agrees with the outward orientation of $S_\pm$. By Remark \ref{rem:sign_divergence}, the band of $S_+$ is positive and its caps are negative, while the band of $S_-$ is negative and its caps are positive. Hence the dividing arcs of the cylinder join the band of $S_+$ to the northern cap of $S_-$, and the southern cap of $S_+$ to the band of $S_-$, so the dividing set of $S$ consists of three parallel circles. The resulting sphere is schematically represented in Figure \ref{fig:rounded_sphere}.

\begin{figure}[!h]
    \centering
    \begin{tikzpicture}
    \node at (0,0) {\includegraphics[width=0.6\linewidth]{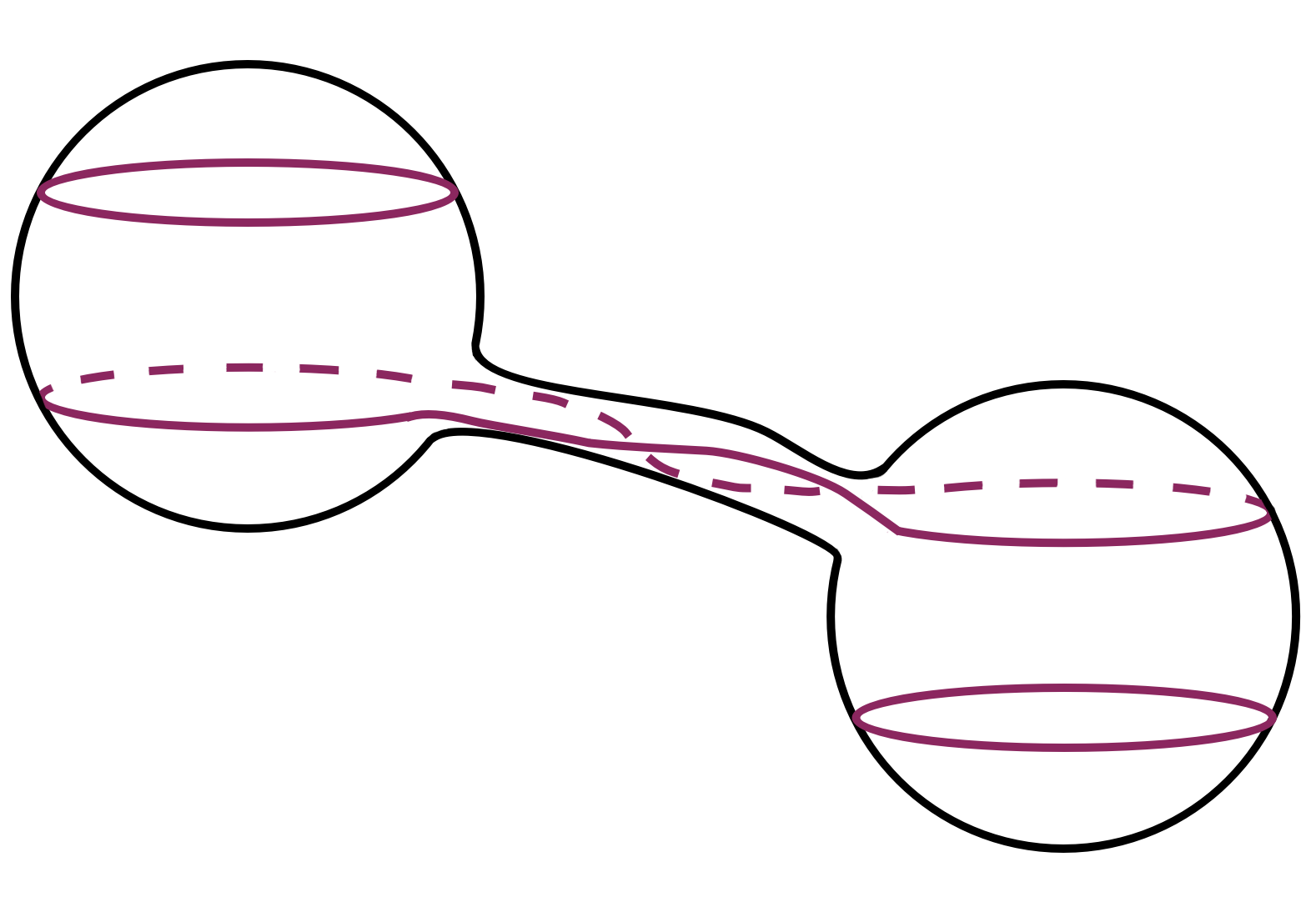}};
    \node at (-2.6,2.1) {$-$};
    \node at (-2.6, 1) {$+$};
    \node at (2.6, -1.1) {$-$};
    \node at (2.6, -2.2) {$+$};
    \end{tikzpicture}
    \caption{Smoothed convex sphere $S$}
    \label{fig:rounded_sphere}
\end{figure}

To conclude, recall that the boundary of a ball of radius between $\pi$ and $2\pi$ in $\R^3$ with the standard overtwisted contact structure (which writes in cylindrical coordinates as $\alpha_{OT}=\cos r\,dz+r\sin r\, d\theta$) is a convex sphere with such a dividing set. By Giroux's flexibility theorem \cite{giroux1991convexity}, the contact structure near a convex surface is determined up to isotopy by its dividing set. We may therefore replace the contact structure inside the ball bounded by $S$, which contains $p_+$ and $p_-$, by this overtwisted ball. Doing this for every pair of zeros of different index (choosing all spheres, arcs and cylinders below pairwise disjoint and disjoint from the remaining zeros), we construct a contact structure on $M$.
\end{proof}

\begin{remark}\label{rmk: regularity in proof of martinet theorem}
We sketch two different arguments for why the near-contact form $\alpha$ in the proof of Theorem \ref{thm:Martinet} may be assumed to be smooth. This is only a technical point, and we have thus decided to separate it from the actual proof.
\begin{enumerate}
    \item As mentioned in Remark \ref{rmk: regularity of metric transversality statement}, the regularity of Theorem $1.1$ in \cite{Enciso2012} may be improved to $C^\infty$. Elliptic regularity of the curl operator in the complement of its kernel then implies that all associated eigenforms are smooth as well.
    \item It is not hard to see that the space of smooth near-contact forms is dense in the space of $C^k$ near-contact forms. This follows from the submanifold property and the density of smooth $1$-forms in the space of $1$-forms with regularity $C^k$. 
\end{enumerate}
\end{remark}

\section{Symplectizations of near-contact forms and self-dual harmonic $2$-forms}\label{sec:symplectization}

Given a near-contact form $\alpha$ one can consider the ``symplectization" of $(M,\alpha)$, that is, the manifold $\R\times M$ endowed with the $2$-form
$$\omega=d(e^t\alpha).$$
This form is symplectic away from the set $\R\times \alpha^{-1}(0)$. In this section, we study this symplectization by first showing that the symplectization of a near-contact form is a so-called near-symplectic form (Section \ref{subsec: symplectizations of near-contact froms are near-symplectic}), and then proving that $\R\times \alpha^{-1}(0)$ admits a standard neighborhood theorem as well in Section \ref{subsec: local normal forms for the symplectization}. In Section \ref{sec:contact_foliation}, we deduce that the singular lines in the symplectization of a near-contact form are surrounded by overtwisted contact hypersurfaces.

\subsection{Symplectizations of near-contact forms are near-symplectic}\label{subsec: symplectizations of near-contact froms are near-symplectic}
As far as the authors can tell, the notion of near-symplectic forms first appeared in a paper of Auroux, Donaldson, and Katzarkov \cite[Definition 1]{AurouxDonaldsonKatzarkov2005}. However, the equivalent geometric object of a transversely vanishing self-dual harmonic $2$-form on a $4$-manifold had been identified as an interesting structure and treated before that by Taubes  \cite{Taubes1999} and Honda \cite{honda2004local}, and has later been used to great effect by Gerig in \cite{Gerig2020,Gerig2022}. The following definition follows the presentation of \cite[Definition 1.1.1]{PerutzThesis2005} since the authors of \cite{AurouxDonaldsonKatzarkov2005} have omitted an important infinitesimal property of such forms which was however clearly intended to be part of the definition.

\begin{defi}\label{def: near symplectic form}
    Let $W$ be a $4$-manifold. A closed $2$-form $\omega\in \Omega^2 (W)$ is called (positive) near-symplectic if for every $x\in M$
\begin{enumerate}
    \item either $\omega_x \wedge \omega_x >0$
    \item or $\omega_x=0$ but $\image(\nabla\omega_x)\subset \Lambda^2 T_x^\ast W$ has dimension $3$, and the wedge product restricts to a positive definite inner product on this subspace.
\end{enumerate}
\end{defi}

It was proven in \cite[Proposition 1]{AurouxDonaldsonKatzarkov2005} (see also \cite[Proposition 1.1.4]{PerutzThesis2005} and \cite[Lemma 2.1]{Perutz2006} for more complete accounts) that any near-symplectic form is self-dual and harmonic for some smooth Riemannian metric, and, conversely, every self-dual harmonic $2$-form which vanishes transversely is a near-symplectic form. In particular, this implies that it is always possible to find near-symplectic forms on closed $4$-manifolds with $b_2^+>0$, see \cite[Theorem 1.1]{HondaTransversality2004}. Indeed, the previous reference establishes that the zero set of self-dual harmonic $2$-forms consists of finitely many embedded circles for generic Riemannian metrics.

The next lemma will streamline the computations that follow and emphasize the conceptual similarity of near-contact and near-symplectic forms.

\begin{lemma}\label{lem: alternative positivity near symplectic}
 The second item in Definition \ref{def: near symplectic form} can equivalently be replaced by the requirement that $\image(\nabla \omega)$ has dimension $3$ and that the symmetric form
 \begin{align*}
     B_\omega (u,v)=\nabla_u\omega\wedge \nabla_v \omega
 \end{align*}
 be positive definite on $T_xW /\ker (\nabla\omega\vert_x)$.
\end{lemma}
\begin{proof}
    This follows immediately upon noticing that $B_\omega=\left(\nabla \omega\right)^\ast Q\vert_{\image \nabla \omega}$, where $Q$ is the bilinear pairing given by the wedge product on $2$-forms.
\end{proof}

We can now show that the symplectization of a near-contact form is near-symplectic (see also \cite[Lemma 10]{ShevSmir}), where $\R\times M$ is oriented by $dt\wedge \operatorname{vol}_M$.

\begin{prop}\label{prop: symplectization of TNS is near symplectic}
    Let $M^3$ be an oriented $3$-manifold and let $\alpha$ be a positive near-contact form. Then the symplectization $(\mathbb R \times M, \omega=d(e^t \alpha))$ is a positive near-symplectic form.
\end{prop}
\begin{proof}
It is obvious that $\omega$ is a symplectic form whenever $(t,p)$ is such that $p\notin \alpha^{-1}(0)$. In case that $p\in \alpha^{-1}(0)$ we compute
\begin{align*}
    \nabla \omega=e^t\left(dt\wedge \nabla \alpha + \nabla d\alpha\right)
\end{align*}
and so it is clear that $\nabla d(e^t \alpha)$ has rank $3$ and that the kernel $\ker \left(\nabla \omega\right)$ is spanned precisely by $\partial_t$. Moreover, we have that for $p\in \alpha^{-1}(0)$
\begin{align*}
    \nabla_u \omega \wedge \nabla_v \omega = e^{2t} dt \wedge \left(\nabla_u \alpha \wedge \nabla_v d\alpha + \nabla_v \alpha \wedge \nabla_u d\alpha\right).
\end{align*}
So $\nabla \omega \wedge \nabla \omega$ is positive definite on $(T_{(t,p)}\mathbb R\times M)/\langle \partial_t\rangle\cong T_p M$ if and only if the symmetrization of $B_\alpha$ is positive definite on $T_p M$, which of course is equivalent to $B_\alpha$ itself being positive definite on $T_p M$. We conclude by Lemma \ref{lem: alternative positivity near symplectic}.
\end{proof}
Note that no symmetry property of $B_\alpha$ entered into the proof of Proposition \ref{prop: symplectization of TNS is near symplectic} since the antisymmetric part vanishes in the central computation.  

In case the $1$-form $\alpha$ on $M$ is a rotational Beltrami $1$-form for a metric $g$ we can even explicitly find a metric which makes the symplectization self-dual and harmonic. In the following we do not assume that $\alpha$ vanishes transversely.

\begin{prop}[Symplectizations of Beltrami fields are harmonic and self-dual]
Let $(M^3,g)$ be an oriented Riemannian manifold, and let
$\alpha\in\Omega^1(M)$ satisfy
$$
\star_g d\alpha=f\alpha
$$
for some positive function $f$. Equip $\mathbb{R}_t\times M$ with the product orientation $dt\wedge\vol_g$ and the
Riemannian metric
$$
G=f^{-2}dt^2+g.
$$
Then the $2$-form $\omega=d(e^t\alpha)$ is self-dual and harmonic with respect to $G$. The vanishing set of $\omega$ is given precisely by $\mathbb R \times \alpha^{-1}(0)$.
\end{prop}

\begin{proof}
Since $\star_g^2=1$ on $1$- and $2$-forms in dimension three, the
Beltrami equation is equivalent to $d\alpha=f\star_g\alpha$.
Set $e^0=f^{-1}dt$ so that $e^0$ is a unit $1$-form for $G$. For every $1$-form $\beta$
and every $2$-form $\eta$ pulled back from $M$, the Hodge stars of
$g$ and $G$ are related by
\begin{align*}
\star_G(e^0\wedge\beta)&=\star_g\beta,\\
\star_G\eta&=e^0\wedge\star_g\eta.
\end{align*}
It follows that
\begin{align*}
&\star_G(dt\wedge\alpha)
=f\star_g\alpha=d\alpha,\\
&\star_Gd\alpha=e^0\wedge\star_gd\alpha=e^0\wedge f\alpha=dt\wedge\alpha.
\end{align*}
As $\omega=e^t\left(dt\wedge\alpha+d\alpha\right)$ it clearly is self-dual and, as an exact form, evidently closed as well. Harmonicity follows from closedness and self-duality. Finally,

\begin{align*}
\omega\wedge\omega
=2e^{2t}\,dt\wedge\alpha\wedge d\alpha=2f e^{2t}|\alpha|_g^2\,
dt\wedge\vol_g. 
\end{align*}
Thus $\omega$ is nondegenerate exactly where $\alpha$ is nonzero. In
particular, if $\alpha$ is nowhere vanishing, then
$\alpha\wedge d\alpha>0$, so $\alpha$ is a contact form and $\omega$
is symplectic. 
\end{proof}
\noindent We note that for a negative proportionality factor $f$ we get anti-self-dual $2$-forms in the symplectization.

\subsection{Local normal forms for the symplectization of near-contact singularities}\label{subsec: local normal forms for the symplectization}

As mentioned previously, Honda proved the existence of a local normal form for neighborhoods of the vanishing set of a transversely vanishing self-dual harmonic $2$-form $\omega$ on closed $4$-manifolds $W$. In his setting, those vanishing sets are closed $1$-dimensional submanifolds of $W$, and they admit one of two possible normal forms (see \cite[Section 4]{honda2004local}). In $S^1_\theta\times D^3$, the normal forms are given by either
$$\omega^+=x(d\theta \wedge dx + dy\wedge dz) + y(d\theta\wedge dy + dz\wedge dx) - 2z (d\theta \wedge dz + dx\wedge dy),$$
or
$$\omega^-=2x(d\theta \wedge dx + dy\wedge dz)-y(d\theta\wedge dy + dz\wedge dx) - z (d\theta \wedge dz + dx\wedge dy).$$
Slightly more precisely, \cite{honda2004local} establishes that any connected component $Z$ of the vanishing set of $\omega$ with orientable normal bundle has a neighborhood which is symplectomorphic to a neighborhood of $S^1\times \{0\}$ in either $(S^1\times D^3,\omega^+)$ or $(S^1\times D^3,\omega^-)$ via a continuous symplectomorphism which is smooth in the complement of $Z$.

In our setting the vanishing sets of $d(e^t\alpha)$ will be noncompact and it is thus natural to consider the analogous canonical $2$-forms along a critical line instead of a critical circle. We thus consider $\R_t\times\R^3$ equipped with either
$$\omega^+_H=x(dt \wedge dx + dy\wedge dz) + y(dt\wedge dy + dz\wedge dx) - 2z (dt \wedge dz + dx\wedge dy),$$
or
$$\omega^-_H=2x(dt \wedge dx + dy\wedge dz) - y(dt\wedge dy + dz\wedge dx) - z (dt \wedge dz + dx\wedge dy)$$
which satisfy ${\omega^{\pm}_H}^2\geq 0$ where equality holds precisely on $\R\times \{0\}$. 

\begin{remark}
We point out that $\omega_H^\pm$ are not, in the coordinates $(t,x,y,z)$, symplectizations $d(e^t\beta)$ of a $1$-form $\beta$ on $\R^3$, nor of the form $dt\wedge\beta+d\beta$. Indeed, in the latter case $\beta=\iota_{\partial_t}\omega^\pm_H$ would be closed (since for $\omega^+_H$ it equals $d\rho$ with $\rho=\frac12(x^2+y^2-2z^2)$), so the restriction of $\omega^\pm_H$ to the slices $\{t\}\times\R^3$, which equals $d\beta$, should vanish. However, this restriction is nonzero. In the former case, even disregarding the $e^t$ factor, one can see that any near-contact form has to be of at least second order at a given zero, and thus so will its symplectization. On the other hand $\omega_H$ is defined by first order expressions. Nevertheless, as explained in Remark \ref{rem:Moser_Honda} below, the two can be identified away from the line on a neighborhood which is however not uniformly thick in $t$.
\end{remark}

Our normal form theorem for near-contact forms implies a normal form near the singular line in its symplectization.

\begin{corollary}\label{cor:symplectization_normal_form}
Let $\alpha$ be a near-contact form on $M$ with a zero at $p$, and let $\phi:(U,0)\to(V,p)$, $f$ and $\alpha_{std}$ be as in Theorem \ref{thm: local model for Beltrami}. Then the map
$$\Phi:\R\times U\longrightarrow\R\times V,\qquad \Phi(t,x)=\big(t-\log f(x),\,\phi(x)\big),$$
is a homeomorphism with $\Phi(t,0)=(t,p)$, smooth on $\R\times(U\setminus\{0\})$, where it satisfies
$$\Phi^*(e^t\alpha)=e^t\alpha_{std},\qquad\text{and hence}\qquad \Phi^*d(e^t\alpha)=d(e^t\alpha_{std}).$$
\end{corollary}
\begin{proof}
On $\R\times(U\setminus\{0\})$ we have
$$\Phi^*(e^t\alpha)=e^{t-\log f}\,\phi^*\alpha=e^t f^{-1}\,f\,\alpha_{std}=e^t\alpha_{std}.$$
By Theorem \ref{thm: local model for Beltrami}, $f$ is positive and continuous on $U$ with $f(0)=1$, so $\Phi$ is continuous with $\Phi(t,0)=(t,p)$. Its inverse $(t,y)\mapsto\big(t+\log f(\phi^{-1}(y)),\,\phi^{-1}(y)\big)$ is also continuous. Since $\phi$ and $f$ are smooth on $U\setminus\{0\}$, $\Phi$ is smooth on $\R\times(U\setminus\{0\})$.
\end{proof}

\begin{remark}
Since $d(e^t\alpha^+_{std})=e^t\big(\omega^+_H+z\,dt\wedge(y\,dx-x\,dy)\big)$, the symplectization of $\alpha^+_{std}$ agrees with $e^t\omega_H^+$ to first order along the singular line. In this sense, Corollary \ref{cor:symplectization_normal_form} is the analogue for symplectizations of Honda's normal form, with $d(e^t\alpha_{std})$ playing the role of $\omega_H$.
\end{remark}

\begin{remark}\label{rem:Moser_Honda}
A Moser-type argument along the path $(1-s)\,\omega^+_H+s\,d(e^t\alpha^+_{std})$, similar to the proof of Theorem \ref{thm:singularGray}, allows one to identify a neighborhood of the singular line in the symplectization of $\alpha^+_{std}$ with a neighborhood of $\R\times\{0\}$ in Honda's model $(\R\times\R^3,\omega^+_H)$, symplectically away from the line. Combined with Corollary \ref{cor:symplectization_normal_form}, this relates the symplectization of a near-contact form near its singular line to Honda's local model. We do not pursue this here, since the resulting neighborhood is not uniform along the line: The neighborhood obtained in this way shrinks as $t\to-\infty$, because the symplectic volume of $\{t<0\}$ is finite for $d(e^t\alpha^+_{std})$ but infinite for $\omega^+_H$. Corollary \ref{cor:symplectization_normal_form} instead provides a uniform tube $\R\times U$.
\end{remark}

\subsection{Overtwisted contact foliation near the line singularity}\label{sec:contact_foliation}

Corollary \ref{cor:symplectization_normal_form} allows us to extract geometric information about the symplectization of a near-contact form near its singular line from the model $d(e^t\alpha_{std})$. The following is the analogue, in our setting, of the discussion in \cite[Section 5]{honda2004local}. We denote by $S_r\subset\R^3$ the sphere of radius $r$ centered at the origin, and write $\alpha_{std}=\alpha^+_{std}$ (the case of $\alpha^-_{std}=-\alpha^+_{std}$ is identical).

\begin{lemma}\label{lem:contact_foliation}
In $\big(\R\times(\R^3\setminus\{0\}),\,d(e^t\alpha_{std})\big)$, the cylinders $C_r=\R\times S_r$, $r>0$, are contact-type hypersurfaces. Each $C_r$ is the concave boundary of $\R\times(\R^3\setminus \overline{B_r(0)})$, and the induced contact structures are overtwisted.
\end{lemma}

\begin{figure}[!h]
    \centering
    \begin{tikzpicture}
    \node at (0,0) {\includegraphics[width=0.9\linewidth]{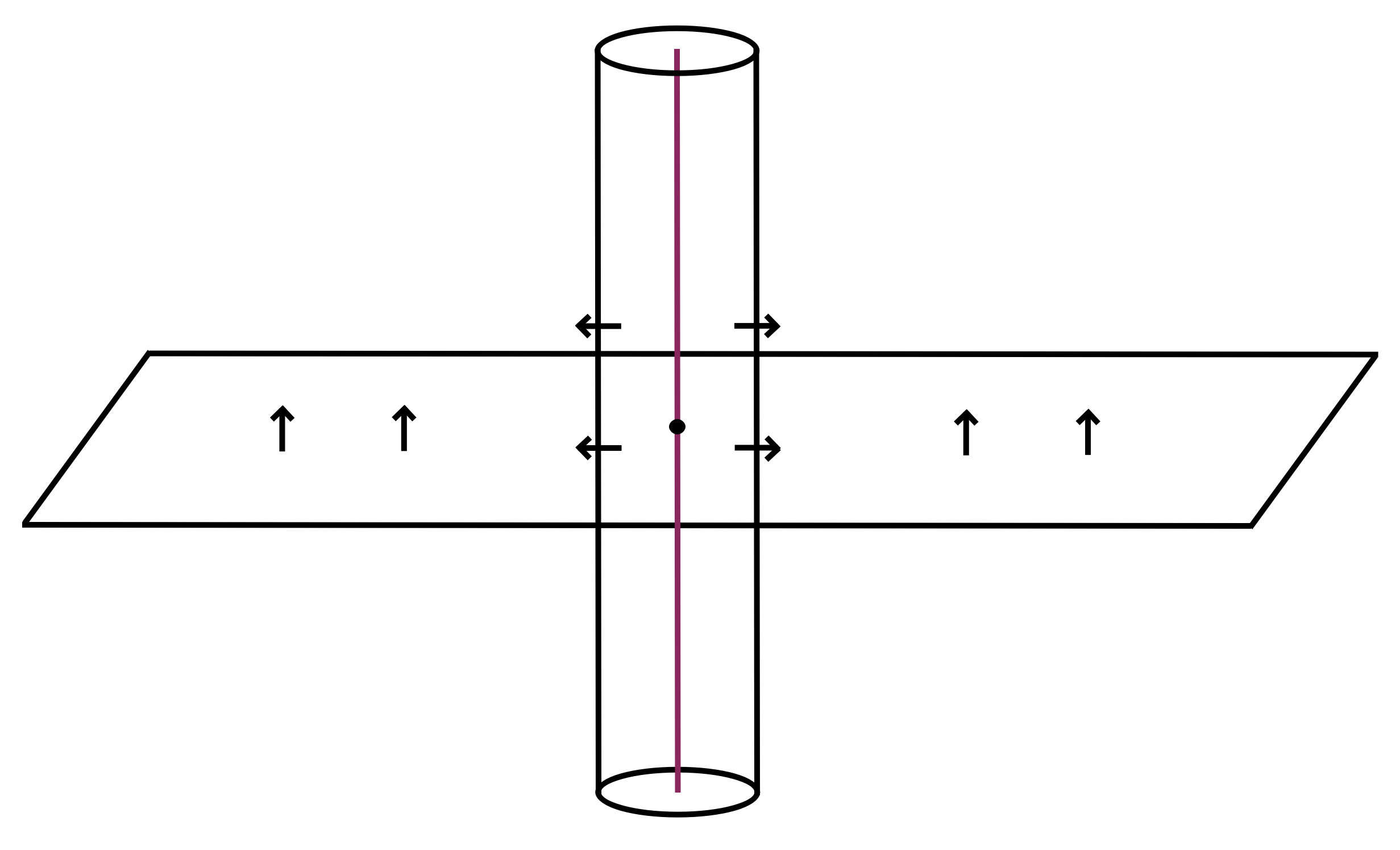}};
    \node at (0.1,0) {$p$};
    \node at (5,1) {$M$};
    \node at (-1.5, 2.1) {$C_{r}$};
    \end{tikzpicture}
    \caption{Symplectization near the singularity}
    \label{fig:symplectization}
\end{figure}

\begin{proof}
Consider the vector field
$$Y=\frac{(x^2+y^2+2z^2)(x\,\partial_x+y\,\partial_y)+4z^3\,\partial_z}{x^2+y^2+4z^2}$$
on $\R^3\setminus\{0\}$. A direct computation shows that
$$\mathcal{L}_Y\alpha_{std}=2\alpha_{std},\qquad \alpha_{std}(Y)=H:=x^2+y^2-2z^2,$$
and
$$\langle (x,y,z),Y\rangle=\frac{(x^2+y^2)^2+2(x^2+y^2)z^2+4z^4}{x^2+y^2+4z^2}>0.$$
Hence $Z=-\partial_t+Y$ satisfies $\mathcal{L}_Z(e^t\alpha_{std})=e^t\alpha_{std}$, so it is a Liouville vector field for $d(e^t\alpha_{std})$. Moreover, $Z$ is transverse to every $C_r$ and points outwards. Therefore, each $C_r$ is of contact type, with contact form
$$\lambda_r=\iota_Z\, d(e^t\alpha_{std})\big|_{C_r}=e^t\big(\alpha_{std}-H\,dt-dH\big)\big|_{C_r},$$
and it is the concave boundary of the region outside it.

To see that $\ker\lambda_r$ is overtwisted, fix $t_0$ and consider the disk $D=\{t_0\}\times S_r^+$, where $S_r^+$ denotes the upper hemisphere. In spherical coordinates,
$$(\alpha_{std}-dH)|_{TS_r}=-3r^2\sin\theta\cos\theta\,d\theta-r^3\sin^2\theta\cos\theta\,d\varphi,$$
which vanishes along the equator $\partial D$ and, in the interior of $D$, only at the pole. Along $\partial D$ we have $H=r^2\neq0$, so $\lambda_r=-e^{t_0}r^2dt$ there, and $\ker\lambda_r=TD$ along $\partial D$. Hence $D$ is an overtwisted disk.
\end{proof}

One can thus schematically understand the symplectization near a singularity as in Figure \ref{fig:symplectization}. The arrows represent the Liouville vector field $Z$, transverse to the contact hypersurfaces $C_r$ (i.e., the direction in which the symplectic form is expanded).

\section{Intrinsic curl eigenforms and rotational Beltrami forms}\label{sec: intrinsic rotational Beltrami fields}

The aim of this section is to establish Theorem \ref{theorem: intrinsically rotationally Beltrami}, completely characterizing those transversely vanishing $1$-forms which are intrinsic rotational Beltrami and intrinsic curl eigenforms.

\begin{remark}\label{rem:rescaling}
We point out that, in Theorem \ref{theorem: intrinsically rotationally Beltrami}, given a smooth metric $g$ such that $\star_g d\alpha=f\alpha$ with $f>0$, for any smooth positive function $\tilde f$ the metric $\tilde g=(f/\tilde f)^2g$ satisfies $\star_{\tilde g} d\alpha=\tilde f\alpha$, since $\star_{cg}=c^{-1/2}\star_g$ on $2$-forms in dimension three. Thus, in particular, the proportionality factor can be made constant, making $\alpha$ into a curl eigenform.
\end{remark}
In view of Theorem \ref{theorem: intrinsically rotationally Beltrami}, one can thus introduce the following two equivalent definitions.

\begin{defi}
A transversely vanishing $1$-form is called an \emph{intrinsic curl eigenform} if it satisfies one of the following equivalent conditions:
\begin{enumerate}
    \item  $\star_g d\alpha=\lambda\alpha$ for some smooth metric $g$ and some nonzero $\lambda$.
    \item is near-contact (positive or negative according to the sign of $\lambda$), and the map \eqref{eq:symmetric_wedge_intro} is symmetric at each singularity.
\end{enumerate}
\end{defi}
Similarly, of course, it makes sense to introduce the notion of \emph{intrinsic rotational Beltrami forms}, but as we have just argued, this definition really coincides with that of an intrinsic curl eigenform. 

Let us also point out that $B_\alpha$ being symmetric is not a condition that is invariant under rescaling by positive functions. Consequently, unlike being near-contact, being intrinsic rotational Beltrami is a property of the $1$-form, not of the associated singular plane field, respectively conformal class of $1$-forms. On the other hand, given a near-contact form $\alpha$ there exists a positive function $f$ such that $B_{f\alpha}$ is in fact symmetric, as we show in the following proposition.

\begin{prop}\label{prop:near-contact is conformal to intrinsically Beltrami}
    Given a near-contact form $\alpha$ there always exists a factor $f\in C^\infty_{>0}(M)$ such that $B_{f\alpha}$ is symmetric. Consequently, any near-contact form is conformal to an intrinsic rotational Beltrami form.
\end{prop}
\begin{proof}
Given $\tilde\alpha=f\alpha$ it is elementary to check that at any zero $p$ of $\alpha$ we have
\begin{align*}
    B_{\tilde \alpha}(u,v)=f(p)^2B_\alpha(u,v)+f(p)\left(\nabla_u \alpha\right)_p \wedge df(p) \wedge \left(\nabla_v \alpha\right)_p
\end{align*}
We note that the second term is antisymmetric in $u$ and $v$. Upon rescaling $f$ so that $f(p)=1$, the question is thus whether the second term can cancel the antisymmetric part of $B_\alpha (u,v)$. We answer this question by showing that the map
\begin{align*}
    T_p^\ast M &\longrightarrow \Lambda^2T_p^\ast M \otimes \Lambda^3 T_p^\ast M \\
    \eta &\mapsto [(u,v)\mapsto \left(\nabla_u \alpha\right)_p \wedge \eta \wedge \left(\nabla_v \alpha\right)_p]
\end{align*}
is an isomorphism. It suffices to show injectivity. Suppose that 
$$\left(\nabla_u \alpha\right)\vert_p \wedge \eta \wedge \left(\nabla_v \alpha\right)_p =0$$ for all $u,v\in T_pM$. Since $u\mapsto\left(\nabla_u \alpha\right)_p$ is an isomorphism, this is equivalent to saying that $\theta_1\wedge \eta \wedge \theta_2=0$ for all $\theta_1,\theta_2 \in T_p^\ast M$, thus $\eta=0$.
\end{proof}
Inspecting the proof of Proposition \ref{prop:near-contact is conformal to intrinsically Beltrami}, we see that, upon fixing the value of $f$ at $p$, $df$ at $p$ is uniquely determined by the requirement that $B_{f\alpha}$ be symmetric.
\medskip

One implication of Theorem \ref{theorem: intrinsically rotationally Beltrami} follows directly from Lemma \ref{lem:transverse}; for the converse (and difficult) implication, we need to construct a metric making $\alpha$ a rotational Beltrami form. For this, we prove some preparatory lemmas in Section \ref{subsection: preparatory lemmas}, give the proof (assuming the key Lemma \ref{lem: polynomial metric construction}) of Theorem \ref{theorem: intrinsically rotationally Beltrami} in Section \ref{subsec: proof of theorem intrinsically rotational beltrami}, and prove the key Lemma \ref{lem: polynomial metric construction} in Section \ref{subsection: proof of polynomial metric lemma}.
After this, we work out the applications of Theorem \ref{theorem: intrinsically rotationally Beltrami} announced in the introduction.

\subsection{Preparatory Lemmas}\label{subsection: preparatory lemmas}
As we explained earlier, any near-contact form $\alpha$ admits Reeb-like vector fields. The following lemma records the extent to which the corresponding Reeb-like vector fields associated with polynomial approximations of $\alpha$ approximate those of the original near-contact form $\alpha$.

\begin{lemma}\label{lem: local Reeb vector field}
Let $\Omega=dx\wedge dy\wedge dz$ be the standard volume form on $\mathbb R^3$, and let $\alpha$ be a smooth $1$-form defined near $0\in\mathbb R^3$. Furthermore, define a vector field $X$ by setting 
\begin{align*}
    \iota_X \Omega=d\alpha.
\end{align*}
For each $k\geq 2$, define
$$
\alpha_{k}:=j^{k}_0\alpha
$$
to be the $k^{th}$ jet of $\alpha$ at the origin. Then there exists a sequence of polynomial vector fields $X_k$ of degree at most $k-1$, defined near $0$, such that
$$
\iota_{X_k}\Omega=d\alpha_{k}.
$$
Furthermore, we have
\begin{align*}
   \lVert X-X_k \rVert=\mathcal{O}(|x|^{k}), \qquad \text{and} \qquad \lVert X_{k+1}-X_k \rVert=\mathcal{O}(|x|^{k}),
\end{align*}
and, in particular, $j^{k-1}_0 X=X_k$.
\end{lemma}

\begin{proof}
There exists a unique smooth vector field $X$ such that
$$
\iota_X\Omega=d\alpha.
$$
Similarly, since contraction with $\Omega$ defines an isomorphism between the space of polynomial vector fields of degree at most $k$ and the space of $2$-forms of degree at most $k$, we may conclude that there is a unique polynomial vector field $X_k$ of degree at most $k-1$ satisfying
$$
\iota_{X_k}\Omega=d\alpha_{k}.
$$
Using that exterior differentiation commutes with taking jets, we have
$$
d\alpha_{k}=j_0^{k-1}(d\alpha)=j_0^{k-1}(\iota_X\Omega)=\iota_{j_0^{k-1}X}\Omega.
$$
Thus uniqueness of $X_k$ implies $X_k=j_0^{k-1} X$ and Taylor's theorem implies
$$
\lVert X-X_k \rVert = \mathcal{O}(|x|^{k}),
$$
and since $X_{k+1}=j_0^{k}X$, the second estimate follows.
\end{proof}

The following lemma is the key technical step in the proof of Theorem \ref{theorem: intrinsically rotationally Beltrami}. We relegate its proof to Section \ref{subsection: proof of polynomial metric lemma} but state it here so we can proceed with the proof of Theorem \ref{theorem: intrinsically rotationally Beltrami}.

\begin{lemma}\label{lem: polynomial metric construction}
    Let $\alpha_k$ be a polynomial $1$-form of degree $k$ defined in a neighborhood $U$ of $0$ in $\mathbb R^3$ with $\alpha_k(0)=0$, {with $k\geq2$, which has a positive near-contact singularity at $0$ at which the map \eqref{eq:symmetric_wedge_intro} is symmetric}. Then there exists a polynomial symmetric tensor $g_k$ of degree {$k-1$}, defined {on $U$ and positive definite on a possibly smaller neighborhood $V_k\subset U$}, so that
    {\begin{align*}
    g_k(X_k,-)=\alpha_k+\mathcal{O}(|x|^{k+1}),
\end{align*}
where $X_k$ is the vector field given by Lemma \ref{lem: local Reeb vector field}, and consequently}
\begin{align*}
    \star_{g_k}d\alpha_k=f_k \alpha_k + \mathcal{O}(|x|^{k+1}).
\end{align*}    
 Here the proportionality factor is given by $f_k=|g_k|^{-\frac12}$. In particular, the functions $f_k$ are positive and smooth in {$V_k$}. {Moreover, the polynomial tensor $g_k$ constructed in Section \ref{subsection: proof of polynomial metric lemma} is ``canonical", in the sense that its homogeneous part of degree $j$ depends only on $j_0^{j+2}\alpha_k$ for $j\leq k-2$. In particular, if $j_0^{s}\alpha_k=\alpha_s$ with $k\geq s$, then $j_0^{j}g_k=j_0^jg_s$ for all $j\leq s-2$.}
\end{lemma}

\subsection{Proof of Theorem \ref{theorem: intrinsically rotationally Beltrami}}\label{subsec: proof of theorem intrinsically rotational beltrami}
We already explained in Remark \ref{rem:rescaling} that the first two items are equivalent. 

{To see that the first item implies the third one, if $\alpha$ is a rotational Beltrami form, by Lemma \ref{lem:transverse} it is a near-contact form whose map \eqref{eq:symmetric_wedge_intro} is symmetric.}

\medskip

{For the converse implication, let $\alpha$ be a $1$-form satisfying Item 3, and we want to show that it is a rotational Beltrami form for some metric. Fix a positive volume form $\mu$ on $M$, let $X$ be defined by $\iota_X\mu=d\alpha$, and work in a chart around a zero of $\alpha$, identified with $0$, in which $\mu=\Omega$. Set $\alpha_k=j^k_0\alpha$. For $k\geq 2$, $\alpha_k$ has a positive near-contact singularity at $0$ with symmetric $B_{\alpha_k}=B_\alpha$, since $B_\alpha$ only depends on $j^2_0\alpha$. Applying Lemma \ref{lem: polynomial metric construction} to each $\alpha_k$, we find polynomial symmetric $(0,2)$-tensors $g_k$ of degree $k-1$ such that $g_k(X_k,\cdot)=\alpha_k+\mathcal{O}(|x|^{k+1})$, and $j_0^{j}g_k=j_0^jg_s$ for each $j\leq s-2$ and $k\geq s$.

By the compatibility just stated, for each $s$ the degree-$s$ part of $g_k$ is independent of $k$ provided $k\geq s+2$, and we denote it by $\gamma^s$. By Borel's lemma, there is a smooth symmetric tensor $g_\infty$ with $j_0^{N}g_\infty=\sum_{s\leq N}\gamma^s$ for every $N$. Given any $k$, we know that $g_\infty-g_k=\mathcal{O}(|x|^{k-1})$ and
$X-X_k=\mathcal{O}(|x|^{k})$, so, since $X$ vanishes at the origin, we have
\begin{align*}
    g_\infty(X,\cdot)-\alpha
    &=(g_\infty-g_k)(X,\cdot)+g_k(X-X_k,\cdot)+\big(g_k(X_k,\cdot)-\alpha_k\big)
    +(\alpha_k-\alpha)\\
    &=\mathcal{O}(|x|^{k})
\end{align*}
for any $k$. It follows that $\alpha_{\infty}:=g_\infty(X,\cdot)$ and $\alpha$ have the same
Taylor expansion at the origin. Since $g_\infty(0)=\gamma^0$ is positive definite, $g_\infty$ is a metric near the origin. 

Consider now the $1$-form
$$\eta:= \alpha-\alpha_\infty,$$
which necessarily vanishes to infinite order at zero. Define the symmetric tensor
$$h:=\frac{1}{\alpha_{\infty}(X)}\left( \eta\otimes \alpha_{\infty} + \alpha_{\infty} \otimes \eta\right)-\frac{\eta(X)}{\alpha_{\infty}(X)^2} \alpha_{\infty} \otimes \alpha_{\infty}. $$
It satisfies
$$h(X,\cdot)=\frac{1}{\alpha_{\infty}(X)}\left(\eta(X)\alpha_{\infty}+\alpha_{\infty}(X)\eta\right)-\frac{\eta(X)}{\alpha_{\infty}(X)}\alpha_{\infty}=\eta.$$
If $h$ extends smoothly as zero along the origin, then $g:=g_\infty+h$ is a smooth metric in a neighborhood of zero and $g(X,\cdot)=\alpha$.
\begin{claim}
    The tensor $h$ extends smoothly as zero along the origin.
\end{claim}
\begin{proof}
Recall that $\alpha_{\infty}(X)=g_\infty(X,X)=|X|^2_{g_\infty}$, so that, $X$ having a
nondegenerate zero at the origin, we have $c|x|^2\le \alpha_{\infty}(X)\le C|x|^2$ near the origin for some $0<c\le C$. If we write $h$ coordinate-wise, each coordinate is a sum of functions of the following form: The product of a flat function $F$ (a product of smooth functions with either a component of $\eta$ or with $\eta(X)$) and a function $G=u/v$ with $u$ smooth and $v$ equal to $\alpha_{\infty}(X)$ or $\alpha_{\infty}(X)^2$. The Faà di Bruno formula together with the above bounds gives $\partial^\ell G=\mathcal{O}(|x|^{-4-|\ell|})$, so, $F$ being flat, the Leibniz rule yields $\partial^\ell(FG)=\mathcal{O}(|x|^N)$ for every $\ell$ and every $N$. A function smooth away from the origin all of whose derivatives tend to zero at the origin extends smoothly by zero, flatly, and hence so does $h$.
\end{proof}
This concludes the construction of $g$ near the singular set of $\alpha$. Away from the zeros, $\alpha(X)>0$, so metrics with $g(X,\cdot)=\alpha$ exist locally (e.g.\ as in \cite{Etnyre2000I}). Since this condition is affine and positivity is convex, a partition of unity extends $g$ to all of $M$ with $g(X,\cdot)=\alpha$. Then $\star_g\alpha=\iota_X\mathrm{vol}_g=J\,\iota_X\mu=J\,d\alpha$ with $J=\mathrm{vol}_g/\mu>0$, so $\star_gd\alpha=\frac{1}{J}\alpha$ and $\alpha$ is metric-dual to a rotational Beltrami field with smooth positive proportionality factor $f=\frac{1}{J}$.}  

\subsection{Proof of Lemma \ref{lem: polynomial metric construction}}\label{subsection: proof of polynomial metric lemma}
The structure of the proof is the following. Given $\alpha_k$ the idea is to find a metric $g_k$ such that $g_k(X_k,-)=\alpha_k$. If this were achieved we would get that $\star_{g_k} d\alpha_k =f_k \,\alpha_k$ for some smooth $f_k$. We will see that in general we can only expect this to hold up to order $(k+1)$. As a first step, we note that $X_k$ is of degree $(k-1)$ and we will decompose $\alpha_k$, $X_k$, and $g_k$ into homogeneous polynomial parts as follows.
$$
\begin{cases}
    \alpha_k&=\alpha^1_k + \dots + \alpha^k_k\\
    X_k&=X^1_k+\dots+X^{k-1}_k\\
    g_k &=g^0_k+g^1_k+\dots+g^{k-1}_k
\end{cases}
$$
Importantly, $g_k$ actually has a zeroth-order part, and moreover $g_k^0$ will have to be positive definite, a necessary and sufficient condition for $g_k$ to be positive definite in a small neighborhood of $0$.
In order to find $g_k$ we thus have to solve a system of linear equations which schematically looks as follows.
\begin{align*}
    g_k^0(X_k^1,-)&=\alpha_k^1\\
    g_k^0(X_k^2,-)+g_k^1(X_k^1,-)&=\alpha_k^2\\
    g_k^0 (X_k^3,-)+g_k^1(X_k^2,-)+g_k^2(X_k^1,-)&=\alpha_k^3\\
    &\vdots\\
    g_k^1 (X_k^{k-1},-)+g_k^2(X_k^{k-2},-)+\dots + g_k^{k-1}(X_k^1,-)&=\alpha^k_k\\
    g_k^2 (X_k^{k-1},-)+g_k^3(X_k^{k-2},-)+\dots +g_k^{k-1}(X_k^2,-)&= 0\\
    &\vdots\\
    g^{k-1}_k(X_k^{k-1},-)&=0
\end{align*}
There are several observations to be made about this system of equations. We will see that solving the first equation is no issue, and in fact there is a unique solution for the coefficients of $g_k^0$. Moreover, the positive definiteness and symmetry of $g_k^0$ will be guaranteed by the properties of the map $(v,w)\mapsto \nabla_v\alpha_k\wedge\nabla_w d\alpha_k$ (namely, its positive definiteness and symmetry). However, from the second equation onward, it seems that the structure of the system of equations implies that conditions on higher jets of $X_k$ might be needed. This is not so. Indeed, we can rewrite the first $k$ equations of this system in the following way
\begin{equation}\label{equ: system for gks}
\begin{cases}
       g_k^0(X_k^1,-)&=\alpha_k^1\\
    g_k^1(X_k^1,-)&=\alpha_k^2-g_k^0(X_k^2,-)\\
    g_k^2(X_k^1,-)&=\alpha_k^3-g_k^0 (X_k^3,-)-g_k^1(X_k^2,-)\\
    &\vdots\\
    g_k^{k-1} (X_k^1,-)&=\alpha_k^k-g_k^1(X_k^{k-1},-)-\dots-g_k^{k-2} (X_k^2,-) 
\end{cases}
\end{equation}
That is, the left-hand side of the $s^{th}$ order equation now only contains $X_k^1$ and $g_k^{s-1}$, while the right-hand side comprises $\alpha_k^s$, higher jets of $X_k$, and orders of the metric we have solved for in previous equations. If we can now show that the linear system for the coefficients of $g_k^{s-1}$, which is determined by $X_k^1$, has full rank, we are done. For $g_k^1$ there is still a unique solution to this problem, but the number of unknowns (coefficients of the metric) quickly outgrows the number of equations. As this is a system of equations which grows quadratically in $k$, the task may nevertheless seem daunting. It turns out, however, that there is an explicit formula for a particular solution at every order. 

The last remaining issue is, of course, that after having solved for all the $\{g_k^s\}_{s=0}^{k-1}$, the last $(k-2)$ equations will not in general be satisfied automatically. As a consequence, the metric $g_k$ will only satisfy
\begin{align*}
    g_k(X_k,-)=\alpha_k + \mathcal{O}(|x|^{k+1}).
\end{align*}
This condition already implies that $\star_{g_k} d \alpha_k=f_k\,\alpha_k + \mathcal{O}(|x|^{k+1})$. Indeed, write the Riemannian volume form of $g_k$ as $\vol_{g_k}=J_k\Omega$, where $J_k=|g_k|^{\frac12}$. Using the identity
$$
\star_{g_k}\alpha_k = 
\iota_{X_k}\vol_{g_k}+\mathcal{O}(|x|^{k+1})=J_k\,\iota_{X_k}\Omega + \mathcal{O}(|x|^{k+1})=J_k d\alpha_k + \mathcal{O}(|x|^{k+1}),
$$
and applying $\star_{g_k}$ once more we obtain 
$$
\star_{g_k}d\alpha_k = J_k^{-1}\alpha_k + \mathcal{O}(|x|^{k+1}).
$$
Consequently, the proportionality factor is
$$
f_k=J_k^{-1}=|g_k|^{-\frac12}.
$$
The constant term $g^0=g_k(0)$ is independent of $k$ and is positive definite. Hence
$$
|g_k|(0)=|g^0|>0.
$$
After possibly shrinking the neighborhood of the origin, $|g_k|$ remains positive, and therefore $f_k$ is a smooth positive function on this neighborhood. In particular, it extends smoothly across the origin, with
$f_k(0)=|g_k|^{-\frac12}(0)>0$, and $f_k$ has order zero at the origin.
The last assertion of the lemma follows from the construction in Sections \ref{ss:g^0} and \ref{ss:g^s}: the degree-$j$ part of $g_k$ depends only on $j_0^{j+2}\alpha_k$ for $j\le k-2$ (see Remark \ref{rem:g_s_and_g_k}).
\medskip

We proceed now to describe how to solve the system of equations \eqref{equ: system for gks}. Since the solution for the $1^{st}$ order equation is qualitatively different from the rest and, in fact, is the only step that really uses the algebraic properties of the map $(v,w)\mapsto \nabla_v \alpha_k \wedge \nabla_w d\alpha_k$ we will separate the proof into two parts: the first solving for $g_k^0$ and the second for $g_k^s$ with $s\geq 1$.

\subsubsection{Solving for $g_k^0$}\label{ss:g^0}

We construct the metric $g_k^0$ from the natural inner product we obtain on $T_0 \mathbb R^3$ from the map $(v,w)\mapsto \nabla_v \alpha_k\wedge \nabla_w d\alpha_k$. The symmetry and positive definiteness of this map precisely guarantee that we get an inner product. The construction works as follows.

Denote by $V=T_0\mathbb R^3$ and define the automorphism $L:V\to V$ by virtue of $\iota_{L(v)}\Omega=(\nabla_v d\alpha_k)_0$. In the metric-free setting we work in right now, this is the best way of formulating the derivative of $d\alpha_k$; we interpret $L(v)$ as the vectorial version of $(\nabla_v d\alpha_k)_0$. Indeed, $X_k^1 (x)=L(x)$ since $X_k$ is determined by the equation
\begin{align*}
    \iota_{X_k}\Omega=d\alpha_k.
\end{align*}
Now define the infinitesimal pairing
\begin{align*}
    B(v,w)\,\Omega&=(\nabla_v\alpha_k)_0\wedge (\nabla_w d\alpha_k)_0\\
    &=(\nabla_v\alpha_k)_0\wedge \iota_{L(w)}\Omega\\
    &=(\nabla_v\alpha_k)_0 (L(w))\,\Omega,
\end{align*}
from which we get that $B(v,w)=(\nabla_v\alpha_k)_0 (L(w))$. This is the most natural generalization of a pairing between the linearizations of $\alpha_k$ and $X_k$. The map $B$ is symmetric and positive definite by assumption.
We now define
\begin{align*}
    g_k^0:=(L^{-1})^*B.
\end{align*}
This satisfies our condition because
\begin{align*}
    g_k^0(L(v),w)&=(L^{-1})^{\ast}B(L(v),w)\\
    &=B(v,L^{-1}w)\\
    &=\left(\nabla_v\alpha_k\right)_0(w),
\end{align*}
which implies
\begin{align*}
    g_k^0(X_k^1,-)=\alpha_k^1.
\end{align*}

\subsubsection{Solving for $g_k^s$ when $s\geq 1$}\label{ss:g^s}
In order to solve 
\begin{align*}
    g_k^s(X_k^1,-)=\alpha_k^{s+1}-g_k^0(X_k^{s+1},-)-g_k^1(X_k^s,-)-\dots-g_k^{s-1}(X_k^2,-)
\end{align*}
we want to first convince ourselves that this is possible from a dimension-counting point of view. Note that at this point $g_k^0,\dots, g^{s-1}_k$ have already been solved for and that $X_k^i$ only depend on $\alpha_k$ up to order $(i+1)$. This system can be read as a linear map from the space of homogeneous polynomials of degree $s$ (times 6 since the metric has $6$ independent components) to the space of homogeneous polynomials of degree $(s+1)$ (times three since the $1$-form has $3$ components). We will denote the space of homogeneous polynomials of degree $s$ in $3$ variables by $\mathcal{P}_s$; its dimension is given by $d(s)=\binom{s+2}{2}$. 
The expected dimension of the space of solutions is thus 
\begin{align*}
    6\,d(s)-3\,d(s+1)=\frac{3}{2}(s+2)(s-1)
\end{align*}

The $s=0$ case was already dealt with in the previous section. For $s=1$ the dimensions agree exactly, and we thus get unique solvability. For all $s>1$ we see that this is an underdetermined linear system. We must now prove what this dimension count would naively suggest, namely that this system is solvable for all $s$. 

Since $X_k$ has a nondegenerate zero at the origin there exists $A\in \mathrm{GL}(3,\mathbb R)$ such that $X_k^1(x)=Ax$. We represent $g_k^s$ by a symmetric $3\times 3$ matrix $g(x)$ whose entries are homogeneous polynomials of degree $s$, and we name the right-hand side of this equation $\beta^{s+1}$ and represent it by a column vector $\beta(x)$ whose components are homogeneous polynomials of degree $(s+1)$. We then need to solve the equation
$$
g(x)Ax=\beta(x)
$$
for $g(x)$. Denote by $\sym_3(\mathcal P_s)$ the space of symmetric $3\times 3$ matrices with entries in $\mathcal P_s$. It is therefore enough to prove that the linear map
\begin{align*}
\mathcal L_{s,A}:\sym_3(\mathcal P_s)
&\longrightarrow \mathcal P_{s+1}^3,\\
g&\longmapsto g(x)Ax,
\end{align*}
is surjective. We reduce the problem to the case of $X_k^1$ being the radial vector field. Set $y=Ax$ and define $\tilde\beta(y):=\beta(A^{-1}y)$. Suppose that there exists a symmetric matrix $\tilde g(y)$, whose entries are homogeneous polynomials of degree $s$, such that $\widetilde g(y)y=\widetilde\beta(y)$.
Define
$$
g(x):=\tilde g(Ax).
$$
Then $g$ is symmetric and homogeneous of degree $s$, and
$$
g(x)Ax= \tilde g(Ax)Ax= \widetilde\beta(Ax)=
\beta(x).
$$
Consequently, the surjectivity of $\mathcal L_{s,A}$ follows from the surjectivity of $\mathcal L_{s,id}$. It therefore suffices to show that one can solve
$$
g(x)x=\beta(x),
$$
for an arbitrary homogeneous polynomial vector field $\beta$ of degree $s+1$.

Quite surprisingly, one can construct an explicit solution to this system. For this we define $q(x)=x^T\beta(x)=\langle \beta(x),x\rangle$ and consider the Ansatz
\begin{align*}
    g=a\left(D\beta+D\beta^T\right)+bD^2q.
\end{align*}
This is just a linear combination of the two most natural symmetric matrices of degree $s$ which can be built out of $\beta$. Both the first term and the second term satisfy a special identity by virtue of the homogeneity. We recall that homogeneity of degree $m$ means that
\begin{align*}
    p(tx)=t^m p(x).
\end{align*}
Upon differentiating in $t$ and evaluating at {$t=1$} we obtain
\begin{align*}
    \nabla p(x) \cdot x = mp(x).
\end{align*}
For vectors whose entries are homogeneous polynomials of degree $m$ we similarly get
\begin{align*}
    D\vec p(x)\cdot x=m\vec p(x).
\end{align*}
Now we differentiate $q(x)$ twice, obtaining
$$\nabla q(x)=\beta(x)+D\beta^T\cdot x, \qquad \text{and} \qquad D^2q(x)\cdot x=(s+1)\nabla q(x).$$
Thus we may reformulate the system as
\begin{align*}
    g(x)\cdot x&=a\left(D\beta(x)+D\beta^T(x)\right)\cdot x+bD^2q(x)\cdot x\\
    &=a\left( (s+1)\beta(x)+\nabla q(x)-\beta(x)\right)+b(s+1)\,\nabla q(x)\\
    &=as\beta(x)+(a+b(s+1))\,\nabla q(x).
\end{align*}
Consequently if we choose $a=\frac{1}{s}$ and $b=-\frac{1}{s(s+1)}$ we see that
\begin{align*}
    g(x)=\frac{1}{s}\left(D\beta(x)+D\beta^T(x)\right)-\frac{1}{s(s+1)}D^2q(x)
\end{align*}
solves the problem. As is evident from the dimension count, for $s=1$ this is the unique solution. (Here $s\geq1$, the case $s=0$ was treated in Section \ref{ss:g^0}.) 

\begin{remark}\label{rem:g_s_and_g_k}
    Notice that the solution $g_k^s$ of the $(s+1)$-st equation only depends on the $(s+2)$-jet of $\alpha_k$. Indeed, $\beta$ is defined in terms of previous iterations (i.e.\ $g_k^j$ with $j<s$), $\alpha_k^{s+1}$, and $X_k^j$ with $j\leq s+1$, and the latter are determined by $j_0^{s+2}\alpha_k$ by Lemma \ref{lem: local Reeb vector field}.
\end{remark}

\subsection{Applications of the intrinsic characterization}
In this last section, we give some applications of the characterization of intrinsic curl eigenforms, Theorem \ref{theorem: intrinsically rotationally Beltrami}.

\subsubsection{Relating Reeb-like vector fields to rotational Beltrami fields in the setting with transverse zeros}
In contact geometry, the so-called Reeb vector field associated to a contact form $\alpha$ is the unique vector field $R_\alpha$ satisfying
\begin{align*}
    \iota_{R_\alpha}d\alpha=0\quad\text{and}\quad \alpha(R_\alpha)=1.
\end{align*}
The class of vector fields which span the kernel of $d\alpha$, that is, the rescalings of $R_\alpha$ by positive functions, is referred to as \emph{Reeb-like}. Reeb-like vector fields can easily be seen to be equivalently characterized as those vector fields for which there exists a (positive) volume form $\mu$ on $M$ such that 
\begin{align}\label{eq: Reeb-like}
    d\alpha=\iota_X \mu.
\end{align}

In their interesting paper \cite{Etnyre2000I}, Etnyre and Ghrist related rotational Beltrami fields, that is, vector fields which satisfy $\nabla\times_g X=f\,X$ for some Riemannian metric $g$, to Reeb-like vector fields in the following way:

\begin{RefTheorem}[{\cite[Theorem 2.1]{Etnyre2000I}}] \label{thm: restate etnyre ghrist}
Let $M$ be a smooth oriented $3$-manifold. A nowhere-vanishing vector field is a rotational Beltrami field for some Riemannian metric if and only if it is a Reeb-like vector field for some contact form.
\end{RefTheorem}

As discussed in Section \ref{sec:contact_singularities}, in the case of a near-contact form $\alpha$, the notion of the Reeb vector field is problematic in that the normalization condition $\alpha(R_\alpha)=1$ forces the vector field to not be defined at the vanishing set of $\alpha$. One can, however, still generalize the notion of Reeb-like vector fields for such $\alpha$ by defining them to be precisely those vector fields for which there exists a volume form $\mu$ on $M$ so that Equation \eqref{eq: Reeb-like} holds. It is then clearly true that Reeb-like vector fields of $\alpha$ are transversely vanishing vector fields whose zero set coincides with that of $\alpha$ and $d\alpha$. Moreover, they span the kernel of $d\alpha$ in the complement of its vanishing set, and any two Reeb-like vector fields for the same near-contact form are related via a rescaling by a positive function. With this terminology in hand, we can formulate a generalization of the statement above, namely \cite[Theorem 2.1]{Etnyre2000I}, in the present setting of near-contact forms.

In the following formulation of the theorem, we will write ``$\alpha$ a near-contact form with symmetric $B_\alpha$'' instead of ``$\alpha$ a transversely vanishing intrinsic rotational Beltrami form'' to emphasize the non-trivial content of the theorem. That is, the real achievement lies in understanding precisely which $1$-forms are rotational Beltrami for some metric, since the rest of the proof is mostly a straightforward generalization of the above theorem by Etnyre and Ghrist.

\begin{theorem}[Beltrami-contact correspondence for near-contact forms]\label{thm: Beltrami-contact for near-contact forms}
Let $M$ be a (not necessarily closed) oriented $3$-manifold. A transversely vanishing vector field is a rotational Beltrami field for some Riemannian metric if and only if it is Reeb-like for some positive near-contact form $\alpha$ for which the map \eqref{eq:symmetric_wedge_intro} is symmetric for every $p\in \alpha^{-1}(0)$.
\end{theorem}
\begin{proof}
The equation $\nabla\times_g X= f\, X$ translates to $dg(X,-)=f\,\iota_X\mu_g$. Thus, the forward direction of the theorem will be shown once we prove that $\alpha=g(X,-)$ is near-contact with symmetric $B_\alpha$. By assumption, $X$ is transverse to the zero section, which, jointly with the rotational Beltrami equation, implies that both $\alpha$ and $d\alpha$ vanish transversely and have the same zero set. For $p\in \alpha^{-1}(0)$ one has 
\begin{align*}
    \left(\nabla_u \alpha\right)_p=\left(\nabla_u X\right)^\flat_p \quad \text{and} \quad \left(\nabla_u d\alpha\right)_p = f(p)\,\iota_{\left(\nabla_u X\right)_p} \,\mu_g
\end{align*}
where the second equality follows from differentiating $d\alpha=f\,\iota_X\mu_g$ at $p$. Putting these terms together yields
\begin{align*}
    B_\alpha (u,v) &= \left(\nabla_u \alpha\right)_p \wedge \left(\nabla_v d\alpha\right)_p\\
    &= f(p)\, \left(\nabla_u X\right)^\flat_p \wedge \iota_{\left(\nabla_v X\right)_p} \,\mu_g \\
    &= f(p)\, g(\nabla_u X,\nabla_v X)_p\,\mu_g,
\end{align*}
which implies that $B_\alpha$ is symmetric in $u$ and $v$, and positive definite, as $f>0$, so $\alpha$ is a positive near-contact form. Moreover, $d\alpha=\iota_X(f\mu_g)$ shows $X$ is Reeb-like.

For the converse direction, let $Y$ be a Reeb-like vector field for a positive near-contact form $\alpha$ with symmetric $B_\alpha$. By Theorem \ref{theorem: intrinsically rotationally Beltrami} we know that there exists a metric $g$ so that $\star_g d\alpha = f\,\alpha$ for some positive function $f$. The dual vector field $X$ defined by $g(X,-)=\alpha$ then satisfies $\nabla\times_g X=f\,X$, and is Reeb-like for $\alpha$. In particular, we have $Y=hX$ for some positive function $h$. Then, the vector field $Y$ is dual to $\alpha$ via the metric $h^{-1}g$. We thus have
\begin{align*}
    \star_{h^{-1}g}d\alpha=h^{\frac12}\star_g d\alpha = h^{\frac12}f\alpha.
\end{align*}
So the $h^{-1}g$ metric dual of $\alpha$, which is $Y$, is a rotational Beltrami field with proportionality factor $h^\frac12 f$ for the metric $h^{-1}g$. Furthermore, the metric dual of a transversely vanishing $1$-form is clearly a transversely vanishing vector field, thus concluding the proof of the theorem.
\end{proof}

\subsubsection{Curl eigenforms with prescribed zeros}

Theorem \ref{theorem: intrinsically rotationally Beltrami} allows us to construct curl eigenforms with prescribed near-contact germs with symmetric map \eqref{eq:symmetric_wedge_intro}. Notice that, in particular, since the condition of being near-contact with the symmetric requirement is a condition on the $2$-jet of a form at the singularity, every higher jet can be prescribed to be arbitrary.

\begin{corollary}\label{cor:prescribed_singularities}
Let $M$ be a closed oriented $3$-manifold. Choose $2n$ points $p_1,...,p_{2n}$ in $M$, and endow $p_i$ with a germ of (positive) near-contact form of signature $-1$ for $i=1,...,n$ and of signature $1$ for $i=n+1,...,2n$ for which the map \eqref{eq:symmetric_wedge_intro} is symmetric at each singularity. Then there exists a smooth metric $g$ on $M$ such that the associated curl operator admits a transversely vanishing eigenform $\alpha$ with positive eigenvalue which vanishes exactly on $p_1,...,p_{2n}$ and extends the given germs.
\end{corollary}
\begin{proof}
Let $\beta$ be any non-vanishing $1$-form in $M$. Choose a ball $B\subset M$ containing the points $\{p_i\}_{i=1,...,2n}$.  Take disjoint closed balls $B_i\subset B$ around $p_i$, and endow each $B_i$ with the given germ of near-contact form, which we denote by $\gamma_i$, which has exactly one transverse zero at $p_i$ of signature $-1$ or $1$ according to whether $i\leq n$ or $i>n$. We fix $\beta$ in a neighborhood of $\partial B$, and then the prescribed boundary data on $B\setminus \bigsqcup_{i=1}^{2n} B_i$, seen as a map to $S^2$, has degree zero. It follows that there is a smooth extension without additional zeros in $B\setminus \bigsqcup_{i=1}^{2n}B_i$. We denote by $\gamma$ this form, extended as $\beta$ away from $B$.
\medskip

By definition of a near-contact form, we know that $\gamma|_{B_i\setminus \{p_i\}}$ is a contact form. Let $U_i$ be a smaller closed ball around $p_i$ strictly contained in $B_i$, so that $\gamma$ is defined in a neighborhood of the boundary of $U_i$. Using the relative existence $h$-principle for contact structures \cite{El}, the form $\gamma|_{M\setminus \bigsqcup_{i=1}^{2n}{U_i}}$ is homotopic to a contact form relative to a neighborhood of the boundary of $\bigcup_i U_i$. We find a contact form $\alpha$ such that $\alpha$ near the boundary of $U_i$ coincides with $\gamma$. In particular, it extends as $\gamma$ inside the balls $U_i$; we keep calling $\alpha$ this global $1$-form in $M$. It is a form that vanishes transversely and is a near-contact form with symmetric $B_\alpha$ at each zero and extending the given germs along $p_i$. By Theorem \ref{theorem: intrinsically rotationally Beltrami} and Remark \ref{rem:rescaling}, the form $\alpha$ is a genuine curl eigenform for some Riemannian metric in $M$.
\end{proof}

\subsubsection{Upgrading the regularity of the metric in a curl eigenvalue problem}
A consequence of the ellipticity of the curl operator is that any curl eigenform inherits the regularity of the metric. In particular, any eigenform of the curl operator defined with respect to a smooth metric is itself smooth. What we will see now is that, assuming a $1$-form is transversely vanishing and satisfies the curl eigenvalue problem with respect to a continuous metric, its regularity is the only obstruction to it satisfying the curl eigenvalue problem with respect to a smooth metric.

\begin{corollary}\label{cor:upgrade_regularity}
Let $M$ be an oriented $3$-manifold, $g_0$ a continuous metric and $\alpha$ a transversely vanishing $1$-form satisfying $\star_{g_0}d\alpha=\lambda\alpha$ for $\lambda\neq 0$. Then there exists a smooth metric $g$ such that $\star_g d\alpha=\lambda\alpha$ if and only if $\alpha$ is smooth.
\end{corollary}
\begin{proof}
    By Convention \ref{conv:positive}, we may assume $\lambda>0$. The forward direction follows from elliptic regularity. Conversely, by Remark \ref{rem:transverse_C0}, $\alpha$ is a positive near-contact form with symmetric $B_\alpha$, so Theorem \ref{theorem: intrinsically rotationally Beltrami} and Remark \ref{rem:rescaling} give a smooth metric $g$ with $\star_gd\alpha=\lambda\alpha$. 
\end{proof}

\appendix

\section{Cylinders between convex surfaces}
In this appendix, we provide a proof of a ``folklore" result in convex surface theory. We are grateful to Eduardo Fernández-Fuertes for suggesting this approach.
\begin{lemma}\label{lem:folklore}
Let $\Sigma_1, \Sigma_2$ be two convex surfaces in a contact manifold $(M,\xi)$. Let $x,y$ be two points respectively in the dividing set of $\Sigma_1$ and $\Sigma_2$. Let $\gamma$ be an embedded Legendrian arc going from $x$ to $y$, with $\gamma\pitchfork \Sigma_1=\{x\}$ and $\gamma \pitchfork \Sigma_2=\{y\}$. Then, there exists a convex cylinder $C$ in $M$ such that:
\begin{itemize}
    \item[-] $\partial C=\left(C\cap \Sigma_1\right) \cup \left(C\cap \Sigma_2\right)$, and each intersection is a Legendrian circle,
    \item[-] $C\cap \Sigma_1$ (respectively $C\cap \Sigma_2$) intersects twice the dividing curve containing $x$ (respectively $y$),
    \item[-] the dividing set of $C$ is given by two arcs, each joining one boundary component to the other.
\end{itemize}
\end{lemma}
\begin{proof}
First, we fix some standard Legendrian neighborhood $N$ of $\gamma$. Choose two circles $c_1\subset\Sigma_1$ and $c_2\subset\Sigma_2$ around $x$ and $y$, bounding disks $D_1,D_2$ contained in $N$, and intersecting the dividing circles on which $x$ and $y$ lie exactly twice. Each $c_i$ is nonisolating, since both components of $\Sigma_i\setminus c_i$ meet the dividing set. By the Legendrian realization principle \cite[Theorem 3.7]{HondaClassification} (see also \cite{HondaRealization}), there is an isotopy of $\Sigma_i$ through convex surfaces, $C^0$-small and supported in a small neighborhood of $D_i$ in $N$, after which the dividing set is the image of the original one and $c_i$ is Legendrian. We replace $\Sigma_i$ by the resulting convex surface, still denoted $\Sigma_i$; the conclusion of the lemma is understood for these perturbed surfaces. In what follows, $\gamma$ is only used through the tight neighborhood $N$, which contains $D_1$ and $D_2$.

Consider now a smooth embedded cylinder $C$, contained in $N$, whose boundary is its (transverse) intersection with $\Sigma_1$ and $\Sigma_2$, which is given exactly by these two Legendrian circles in $\Sigma_1$ and $\Sigma_2$. Without loss of generality, these circles bound disks in $\Sigma_1$ and $\Sigma_2$ contained in $N$. This cylinder can be perturbed in its interior so that it becomes convex by \cite[Theorem 2.23]{EtnyreNotes}. Indeed, the twisting number of the Legendrian boundary with respect to $C$ is negative: Its value is equal to the twisting with respect to $\Sigma_1$ (or $\Sigma_2$), which is negative because it is given by minus half the number of intersection points with the dividing set (which is two). We claim that near each boundary component of $C$, the dividing set is given by two lines stemming out of the boundary circle. Indeed, since the Legendrian circle at the intersection of $C$ with $\Sigma_1$ (or $\Sigma_2$) intersects the dividing set of $\Sigma_1$ twice, it must do so twice as well for the dividing set of $C$, as these must alternate as one goes along the Legendrian circle (see e.g. \cite[Lemma 4.1]{EtnyreNotes}, or \cite{HondaClassification}). 

We are left with showing that the dividing set of $C$ is given by two arcs, each joining one boundary component to the other. A priori, there are three possibilities: either the dividing set of $C$ contains some closed curve, is given exactly by two arcs, each joining one boundary component to the other, or is given exactly by two boundary-parallel segments. Consider the edge-rounded (see \cite[Lemma 4.3]{EtnyreNotes}, or \cite{HondaClassification}) sphere $S'$ obtained from the cylinder together with the two disks in $\Sigma_1$ and $\Sigma_2$ containing $x$ and $y$. If the dividing set of $C$ had some closed curve, or if it were given by two boundary-parallel segments, the dividing set of $S'$ would have more than one component.  By Giroux's criterion, an arbitrarily small neighborhood of $S'$ would be overtwisted, but such a neighborhood would be contained in $N$, where the contact structure is tight. It must be the case that the dividing set of $C$ is given by two arcs, each joining one boundary component to the other, as claimed.
\end{proof}
\newpage
\bibliographystyle{amsplain0}
\bibliography{bibliography.bib}

@article{El,
  title={Classification of overtwisted contact structures on 3-manifolds},
  author={Eliashberg, Yakov},
  journal={Inventiones mathematicae},
  volume={98},
  number={3},
  pages={623--637},
  year={1989},
  publisher={Citeseer}
}

@article{BEM,
  title={Existence and classification of overtwisted contact structures in all dimensions},
  author={Borman, Matthew Strom and Eliashberg, Yakov and Murphy, Emmy},
  journal={Acta Mathematica},
  volume={215},
  number={2},
  pages={281--361},
  year={2015},
  publisher={Springer}
}

@article{honda2004local,
  title={Local properties of self-dual harmonic 2-forms on a 4-manifold},
  author={Honda, Ko},
  journal={Journal fur die Reine und Angewandte Mathematik},
  volume={577},
  pages={105--116},
  year={2004},
  publisher={Berlin, W. de Gruyter.}
}

@article{giroux1991convexity,
  title={Convexity in contact topology},
  author={Giroux, Emmanuel},
  journal={Comment. Math. Helv},
  volume={66},
  number={4},
  pages={637--677},
  year={1991}
}

@book{Geiges,
  title={An introduction to contact topology},
  author={Geiges, Hansj{\"o}rg},
  volume={109},
  year={2008},
  publisher={Cambridge University Press}
}

@article {Bar,
    AUTHOR = {B\"ar, Christian},
     TITLE = {The curl operator on odd-dimensional manifolds},
   JOURNAL = {J. Math. Phys.},
  FJOURNAL = {Journal of Mathematical Physics},
    VOLUME = {60},
      YEAR = {2019},
    NUMBER = {3},
     PAGES = {031501, 16},
      ISSN = {0022-2488,1089-7658},
   MRCLASS = {53A45},
  MRNUMBER = {3920324},
MRREVIEWER = {Alfonso\ Garc\'ia-Parrado},
       DOI = {10.1063/1.5082528},
       URL = {https://doi.org/10.1063/1.5082528},
}

@article {Enciso2012,
    AUTHOR = {Enciso, Alberto and Peralta-Salas, Daniel},
     TITLE = {Nondegeneracy of the eigenvalues of the {H}odge {L}aplacian
              for generic metrics on 3-manifolds},
   JOURNAL = {Trans. Amer. Math. Soc.},
  FJOURNAL = {Transactions of the American Mathematical Society},
    VOLUME = {364},
      YEAR = {2012},
    NUMBER = {8},
     PAGES = {4207--4224},
      ISSN = {0002-9947,1088-6850},
   MRCLASS = {58J50 (35B38 35P05 35R01 58D17)},
  MRNUMBER = {2912451},
MRREVIEWER = {Colette\ Ann\'e},
       DOI = {10.1090/S0002-9947-2012-05496-1},
       URL = {https://doi.org/10.1090/S0002-9947-2012-05496-1},
}

@article {kepplinger2022spectral,
    AUTHOR = {Kepplinger, Willi},
     TITLE = {On spectral simplicity of the {H}odge {L}aplacian and curl
              operator along paths of metrics},
   JOURNAL = {Trans. Amer. Math. Soc.},
  FJOURNAL = {Transactions of the American Mathematical Society},
    VOLUME = {377},
      YEAR = {2024},
    NUMBER = {11},
     PAGES = {7829--7845},
      ISSN = {0002-9947,1088-6850},
   MRCLASS = {58C40 (35P99 35R01 58A10 58A14)},
  MRNUMBER = {4806198},
       DOI = {10.1090/tran/9221},
       URL = {https://doi.org/10.1090/tran/9221},
}

@article {PeraltaSlobodeanu2023,
    AUTHOR = {Peralta-Salas, Daniel and Slobodeanu, Radu},
     TITLE = {Contact structures and {B}eltrami fields on the torus and the
              sphere},
   JOURNAL = {Indiana Univ. Math. J.},
  FJOURNAL = {Indiana University Mathematics Journal},
    VOLUME = {72},
      YEAR = {2023},
    NUMBER = {2},
     PAGES = {699--730},
      ISSN = {0022-2518,1943-5258},
   MRCLASS = {53D10},
  MRNUMBER = {4596189},
}

@article {R,
    AUTHOR = {Ruiz, David},
     TITLE = {Symmetry results for compactly supported steady solutions of
              the 2{D} {E}uler equations},
   JOURNAL = {Arch. Ration. Mech. Anal.},
  FJOURNAL = {Archive for Rational Mechanics and Analysis},
    VOLUME = {247},
      YEAR = {2023},
    NUMBER = {3},
     PAGES = {Paper No. 40, 25},
      ISSN = {0003-9527,1432-0673},
   MRCLASS = {35Q31 (76B03)},
  MRNUMBER = {4575796},
MRREVIEWER = {Mikhail\ M.\ Shvartsman},
       DOI = {10.1007/s00205-023-01877-6},
       URL = {https://doi.org/10.1007/s00205-023-01877-6},
}

@misc{greilhuber2023arnolds,
      title={On {A}rnold's {T}ransversality {C}onjecture for the {L}aplace-{B}eltrami {O}perator}, 
      author={Josef Greilhuber and Willi Kepplinger},
      year={2024},
      eprint={2312.16939},
      archivePrefix={arXiv},
      primaryClass={math.DG},
      url={https://arxiv.org/abs/2312.16939}, 
}

@article {EncisoPeralta2012,
    AUTHOR = {Enciso, Alberto and Peralta-Salas, Daniel},
     TITLE = {Knots and links in steady solutions of the {E}uler equation},
   JOURNAL = {Ann. of Math. (2)},
  FJOURNAL = {Annals of Mathematics. Second Series},
    VOLUME = {175},
      YEAR = {2012},
    NUMBER = {1},
     PAGES = {345--367},
      ISSN = {0003-486X,1939-8980},
   MRCLASS = {35Q31 (76B47)},
  MRNUMBER = {2874645},
MRREVIEWER = {Franck\ Sueur},
       DOI = {10.4007/annals.2012.175.1.9},
       URL = {https://doi.org/10.4007/annals.2012.175.1.9},
}

@article{Etnyre2012,
   author = {John B. Etnyre and Rafal Komendarczyk and Patrick Massot},
   doi = {10.1007/s00222-011-0355-2},
   issue = {3},
   journal = {Inventiones Mathematicae},
   title = {Tightness in contact metric 3-manifolds},
   volume = {188},
   pages={621–657},
   year = {2012},
}

@article{Etnyre2000I,
   author = {John Etnyre and Robert Ghrist},
   doi = {10.1088/0951-7715/13/2/306},
   issue = {2},
   pages={441–458},
   journal = {Nonlinearity},
   title = {{Contact topology and hydrodynamics: I. Beltrami fields and the Seifert conjecture}},
   volume = {13},
   year = {2000},
}

@article{Etnyre2000III,
   author = {John Etnyre and Robert Ghrist},
   doi = {10.1090/s0002-9947-00-02651-9},
   issue = {12},
   journal = {Transactions of the American Mathematical Society},
   title = {{Contact topology and hydrodynamics III: knotted orbits}},
   volume = {352},
   pages={5781–5794},
   year = {2000},
}

@article {Cardona2021,
    AUTHOR = {Cardona, Robert and Miranda, Eva and Peralta-Salas, Daniel and
              Presas, Francisco},
     TITLE = {Constructing {T}uring complete {E}uler flows in dimension 3},
   JOURNAL = {Proc. Natl. Acad. Sci. USA},
  FJOURNAL = {Proceedings of the National Academy of Sciences of the United
              States of America},
    VOLUME = {118},
      YEAR = {2021},
    NUMBER = {19},
     PAGES = {Paper No. e2026818118, 9},
      ISSN = {0027-8424,1091-6490},
   MRCLASS = {68Q04 (76B03)},
  MRNUMBER = {4294081},
MRREVIEWER = {Gabriel\ Ciobanu},
       DOI = {10.1073/pnas.2026818118},
       URL = {https://doi.org/10.1073/pnas.2026818118},
}

@article{Woltjer1958,
author = {L. Woltjer },
title = {A Theorem on force-free magnetic fields},
journal = {Proceedings of the National Academy of Sciences},
volume = {44},
number = {6},
pages = {489-491},
year = {1958},
doi = {10.1073/pnas.44.6.489},
URL = {https://www.pnas.org/doi/abs/10.1073/pnas.44.6.489},
eprint = {https://www.pnas.org/doi/pdf/10.1073/pnas.44.6.489}}

@book {Kato,
    AUTHOR = {Kato, Tosio},
     TITLE = {Perturbation theory for linear operators},
    SERIES = {Die Grundlehren der mathematischen Wissenschaften},
    VOLUME = {Band 132},
 PUBLISHER = {Springer-Verlag New York, Inc., New York},
      YEAR = {1966},
     PAGES = {xix+592},
   MRCLASS = {47.00 (47.48)},
  MRNUMBER = {203473},
MRREVIEWER = {L.\ de Branges},
}

@article {Hofer1993,
    AUTHOR = {Hofer, H.},
     TITLE = {Pseudoholomorphic curves in symplectizations with applications
              to the {W}einstein conjecture in dimension three},
   JOURNAL = {Invent. Math.},
  FJOURNAL = {Inventiones Mathematicae},
    VOLUME = {114},
      YEAR = {1993},
    NUMBER = {3},
     PAGES = {515--563},
      ISSN = {0020-9910,1432-1297},
   MRCLASS = {58F05 (58E99 58F22)},
  MRNUMBER = {1244912},
MRREVIEWER = {Hansj\"org\ Geiges},
       DOI = {10.1007/BF01232679},
       URL = {https://doi.org/10.1007/BF01232679},
}

@article{PollardAlexander,
  title={Singular contact geometry and beltrami fields in cholesteric liquid crystals},
  author={Pollard, Joseph and Alexander, Gareth P},
  journal={arXiv preprint arXiv:1911.10159},
  year={2019}
}

@article {AurouxDonaldsonKatzarkov2005,
    AUTHOR = {Auroux, Denis and Donaldson, Simon K. and Katzarkov, Ludmil},
     TITLE = {Singular {L}efschetz pencils},
   JOURNAL = {Geom. Topol.},
  FJOURNAL = {Geometry and Topology},
    VOLUME = {9},
      YEAR = {2005},
     PAGES = {1043--1114},
      ISSN = {1465-3060,1364-0380},
   MRCLASS = {53D35 (57M50 57R17)},
  MRNUMBER = {2140998},
MRREVIEWER = {Michael\ J.\ Usher},
       DOI = {10.2140/gt.2005.9.1043},
       URL = {https://doi.org/10.2140/gt.2005.9.1043},
}

@article {Taubes2002,
    AUTHOR = {Taubes, Clifford Henry},
     TITLE = {A compendium of pseudoholomorphic beasts in {$\Bbb R\times
              (S^1\times S^2)$}},
   JOURNAL = {Geom. Topol.},
  FJOURNAL = {Geometry and Topology},
    VOLUME = {6},
      YEAR = {2002},
     PAGES = {657--814},
      ISSN = {1465-3060,1364-0380},
   MRCLASS = {32Q65 (53D45 57R17 57R57)},
  MRNUMBER = {1943381},
MRREVIEWER = {Klaus\ Mohnke},
       DOI = {10.2140/gt.2002.6.657},
       URL = {https://doi.org/10.2140/gt.2002.6.657},
}

@article {Gerig2020,
    AUTHOR = {Gerig, Chris},
     TITLE = {Taming the pseudoholomorphic beasts in {$\Bbb{R}\times
              (S^1\times S^2)$}},
   JOURNAL = {Geom. Topol.},
  FJOURNAL = {Geometry \& Topology},
    VOLUME = {24},
      YEAR = {2020},
    NUMBER = {4},
     PAGES = {1791--1839},
      ISSN = {1465-3060,1364-0380},
   MRCLASS = {53D42 (57R57)},
  MRNUMBER = {4173922},
       DOI = {10.2140/gt.2020.24.1791},
       URL = {https://doi.org/10.2140/gt.2020.24.1791},
}

@article{EncisoKepplingerPeralta2025,
author = {Alberto Enciso  and Willi Kepplinger  and Daniel Peralta-Salas },
title = {Isolated steady solutions of the 3D Euler equations},
journal = {Proceedings of the National Academy of Sciences},
volume = {122},
number = {12},
pages = {e2414730122},
year = {2025},
doi = {10.1073/pnas.2414730122},
URL = {https://www.pnas.org/doi/abs/10.1073/pnas.2414730122},
eprint = {https://www.pnas.org/doi/pdf/10.1073/pnas.2414730122}}

@article {CieliebakVolkov2017,
    AUTHOR = {Cieliebak, K. and Volkov, E.},
     TITLE = {A note on the stationary {E}uler equations of hydrodynamics},
   JOURNAL = {Ergodic Theory Dynam. Systems},
  FJOURNAL = {Ergodic Theory and Dynamical Systems},
    VOLUME = {37},
      YEAR = {2017},
    NUMBER = {2},
     PAGES = {454--480},
      ISSN = {0143-3857,1469-4417},
   MRCLASS = {35Q31 (76B03)},
  MRNUMBER = {3614033},
MRREVIEWER = {Xinyu\ He},
       DOI = {10.1017/etds.2015.50},
       URL = {https://doi.org/10.1017/etds.2015.50},
}

@article{Hartman1960,
    author = {Philip Hartman},
    title = {A lemma in the theory of structural stability of differential equations},
    journal = {Proc. Amer. Math. Soc.},
    year = {1960},
    pages = {610--620}
}

@article{Giroux2001,
author = {Giroux, Emmanuel},
year = {2001},
month = {06},
pages = {218-262},
title = {Structures de contact sur les variétés fibrées en cercles au-dessus d'une surface},
volume = {76},
journal = {Commentarii Mathematici Helvetici},
doi = {10.1007/PL00000378}
}

@book{ArnoldKhesin,
    author = {Arnold, Vladimir and Khesin, Boris},
    title = {Topological Methods in Hydrodynamics},
    publisher = {Springer},
    SERIES = {Applied Mathematical Sciences},
    year = {2021},
    ISBN={978-3-030-74277-5}
}

@article{JiangMinhZung,
  title={Normalization of Singular Contact Forms and Primitive 1-forms: K. Jiang et al.},
  author={Jiang, Kai and Truong, Hong Minh and Zung, Nguyen Tien},
  journal={Acta Mathematica Vietnamica},
  volume={49},
  number={3},
  pages={327--345},
  year={2024},
  publisher={Springer}
}

@article{MirandaOms,
  title={Contact structures with singularities: from local to global},
  author={Miranda, Eva and Oms, C{\'e}dric},
  journal={Journal of Geometry and Physics},
  volume={192},
  pages={104957},
  year={2023},
  publisher={Elsevier}
}

@article{Volkov,
  title={Characterization of intrinsically harmonic forms},
  author={Volkov, Evgeny},
  journal={Journal of Topology},
  volume={1},
  number={3},
  pages={643--650},
  year={2008},
  publisher={Wiley Online Library}
}

@article{CO,
author = {Robert Cardona and C{\'e}dric Oms},
title = {{EXISTENCE AND CLASSIFICATION OF $b$-CONTACT STRUCTURES}},
volume = {70},
journal = {Publicacions Matemàtiques},
number = {1},
publisher = {Universitat Autònoma de Barcelona, Departament de Matemàtiques},
pages = {247 -- 275},
year = {2026},
doi = {10.5565/PUBLMAT7012609},
URL = {https://doi.org/10.5565/PUBLMAT7012609}
}

@phdthesis{VolkovPhD,
  author      = {Evgeny Volkov},
  title       = {Characterization of intrinsically harmonic forms},
  school      = {Ludwig-Maximilians-Universit{\"a}t M{\"u}nchen},
  year        = {2007},
  type        = {Dissertation},
  address     = {M{\"u}nchen},
  doi         = {10.5282/edoc.6695},
  url         = {https://edoc.ub.uni-muenchen.de/6695/}
}

@article{Calabi,
  title={An intrinsic characterization of harmonic one-forms},
  author={Calabi, Eugenio},
  journal={Global analysis (papers in honor of K. Kodaira)},
  pages={101--117},
  year={1969},
  publisher={Princeton University Press Tokyo, Japan}
}

@article{FarberKatzLevine,
  title={Morse theory of harmonic forms},
  author={Farber, Michael and Katz, Gabriel and Levine, Jerome},
  journal={Topology},
  volume={37},
  number={3},
  pages={469--483},
  year={1998},
  publisher={Elsevier}
}

@article {Perutz2006,
    AUTHOR = {Perutz, Timothy},
     TITLE = {Zero-sets of near-symplectic forms},
   JOURNAL = {J. Symplectic Geom.},
  FJOURNAL = {The Journal of Symplectic Geometry},
    VOLUME = {4},
      YEAR = {2006},
    NUMBER = {3},
     PAGES = {237--257},
      ISSN = {1527-5256,1540-2347},
   MRCLASS = {57R17 (53D35)},
  MRNUMBER = {2314214},
MRREVIEWER = {David\ T.\ Gay},
       DOI = {10.4310/jsg.2006.v4.n3.a1},
       URL = {https://doi.org/10.4310/jsg.2006.v4.n3.a1},
}

@article {GayKirby2004,
    AUTHOR = {Gay, David T. and Kirby, Robion},
     TITLE = {Constructing symplectic forms on 4-manifolds which vanish on
              circles},
   JOURNAL = {Geom. Topol.},
  FJOURNAL = {Geometry and Topology},
    VOLUME = {8},
      YEAR = {2004},
     PAGES = {743--777},
      ISSN = {1465-3060,1364-0380},
   MRCLASS = {53D35 (57M50 57R17)},
  MRNUMBER = {2057780},
MRREVIEWER = {Olgu\c ta\ Bu\c se},
       DOI = {10.2140/gt.2004.8.743},
       URL = {https://doi.org/10.2140/gt.2004.8.743},
}

@phdthesis{PerutzThesis2005,
    author = {Timothy Perutz},
    title = {Surface-fibrations, four-manifolds, and symplectic Floer homology},
    school = {Imperial College, University of London},
    year = {2005}
}

@article {HondaTransversality2004,
    AUTHOR = {Honda, Ko},
     TITLE = {Transversality theorems for harmonic forms},
   JOURNAL = {Rocky Mountain J. Math.},
  FJOURNAL = {The Rocky Mountain Journal of Mathematics},
    VOLUME = {34},
      YEAR = {2004},
    NUMBER = {2},
     PAGES = {629--664},
      ISSN = {0035-7596,1945-3795},
   MRCLASS = {58J50 (58J32 58J37)},
  MRNUMBER = {2072799},
MRREVIEWER = {Nader\ Yeganefar},
       DOI = {10.1216/rmjm/1181069872},
       URL = {https://doi.org/10.1216/rmjm/1181069872},
}

@article{EtnyreNotes,
  title={Convex surfaces in contact geometry: class notes},
  author={Etnyre, John B},
  journal={Lecture notes available on: https://etnyre.math.gatech.
edu/preprints/papers/surfaces.pdf},
  year={2004}
}

@article{HondaClassification,
author = {Ko Honda},
title = {{On the classification of tight contact structures I}},
volume = {4},
journal = {Geometry \& Topology},
number = {1},
publisher = {MSP},
pages = {309 -- 368},
year = {2000},
doi = {10.2140/gt.2000.4.309},
URL = {https://doi.org/10.2140/gt.2000.4.309}
}

@article{HondaRealization,
author = {Ko Honda},
title = {{Gluing tight contact structures}},
volume = {115},
journal = {Duke Mathematical Journal},
number = {3},
publisher = {Duke University Press},
pages = {435 -- 478},
year = {2002},
doi = {10.1215/S0012-7094-02-11532-4},
URL = {https://doi.org/10.1215/S0012-7094-02-11532-4}
}

@article{ShevSmir,
  title={Anti-self-dual blowups II},
  author={Shevchishin, Vsevolod and Smirnov, Gleb},
  journal={arXiv preprint arXiv:2512.03352},
  year={2025}
}

@article {Taubes1999,
    AUTHOR = {Taubes, Clifford Henry},
     TITLE = {Seiberg-{W}itten invariants and pseudo-holomorphic
              subvarieties for self-dual, harmonic {$2$}-forms},
   JOURNAL = {Geom. Topol.},
  FJOURNAL = {Geometry and Topology},
    VOLUME = {3},
      YEAR = {1999},
     PAGES = {167--210},
      ISSN = {1465-3060,1364-0380},
   MRCLASS = {53D35 (32Q65 57R57)},
  MRNUMBER = {1697181},
MRREVIEWER = {Ignasi\ Mundet-Riera},
       DOI = {10.2140/gt.1999.3.167},
       URL = {https://doi.org/10.2140/gt.1999.3.167},
}

@article {Gerig2022,
    AUTHOR = {Gerig, Chris},
     TITLE = {Seiberg-{W}itten and {G}romov invariants for self-dual
              harmonic 2-forms},
   JOURNAL = {Geom. Topol.},
  FJOURNAL = {Geometry \& Topology},
    VOLUME = {26},
      YEAR = {2022},
    NUMBER = {8},
     PAGES = {3307--3365},
      ISSN = {1465-3060,1364-0380},
   MRCLASS = {53D42 (57K41 57R57)},
  MRNUMBER = {4562564},
MRREVIEWER = {Alexander\ Fel\cprime shtyn},
       DOI = {10.2140/gt.2022.26.3307},
       URL = {https://doi.org/10.2140/gt.2022.26.3307},
}

@article {Altschuler1995,
    AUTHOR = {Altschuler, Steven J.},
     TITLE = {A geometric heat flow for one-forms on three-dimensional
              manifolds},
   JOURNAL = {Illinois J. Math.},
  FJOURNAL = {Illinois Journal of Mathematics},
    VOLUME = {39},
      YEAR = {1995},
    NUMBER = {1},
     PAGES = {98--118},
      ISSN = {0019-2082,1945-6581},
   MRCLASS = {58G11 (53C12 53C15 58G30)},
  MRNUMBER = {1299651},
MRREVIEWER = {Ben\ Andrews},
       URL = {http://projecteuclid.org/euclid.ijm/1255986631},
}

@article {CdV1986,
    AUTHOR = {Colin de Verdi\`ere, Yves},
     TITLE = {Sur la multiplicit\'e{} de la premi\`ere valeur propre non
              nulle du laplacien},
   JOURNAL = {Comment. Math. Helv.},
  FJOURNAL = {Commentarii Mathematici Helvetici},
    VOLUME = {61},
      YEAR = {1986},
    NUMBER = {2},
     PAGES = {254--270},
      ISSN = {0010-2571,1420-8946},
   MRCLASS = {58G25},
  MRNUMBER = {856089},
MRREVIEWER = {G\'erard\ Besson},
       DOI = {10.1007/BF02621914},
       URL = {https://doi.org/10.1007/BF02621914},
}

@article {EncisoPeralta2015,
    AUTHOR = {Enciso, Alberto and Peralta-Salas, Daniel},
     TITLE = {Eigenfunctions with prescribed nodal sets},
   JOURNAL = {J. Differential Geom.},
  FJOURNAL = {Journal of Differential Geometry},
    VOLUME = {101},
      YEAR = {2015},
    NUMBER = {2},
     PAGES = {197--211},
      ISSN = {0022-040X,1945-743X},
   MRCLASS = {58J50},
  MRNUMBER = {3399096},
MRREVIEWER = {Xiaolong\ Han},
       URL = {http://projecteuclid.org/euclid.jdg/1442364650},
}

@article {Lisi2011,
    AUTHOR = {Lisi, Samuel T.},
     TITLE = {Dividing sets as nodal sets of an eigenfunction of the
              {L}aplacian},
   JOURNAL = {Algebr. Geom. Topol.},
  FJOURNAL = {Algebraic \& Geometric Topology},
    VOLUME = {11},
      YEAR = {2011},
    NUMBER = {3},
     PAGES = {1435--1443},
      ISSN = {1472-2747,1472-2739},
   MRCLASS = {53D35 (58J50)},
  MRNUMBER = {2821429},
MRREVIEWER = {Rafal\ Komendarczyk},
       DOI = {10.2140/agt.2011.11.1435},
       URL = {https://doi.org/10.2140/agt.2011.11.1435},
}

\Addresses

\end{document}